\documentclass[oneside,english,a4paper]{amsart}
\usepackage[T1]{fontenc}
\usepackage[latin9]{inputenc}
\usepackage{verbatim}
\usepackage{mathrsfs}
\usepackage{amstext}
\usepackage{amsthm}
\usepackage{amssymb}

\makeatletter
\numberwithin{equation}{section}
\numberwithin{figure}{section}
\numberwithin{table}{section}
\theoremstyle{plain}
\newtheorem{thm}{\protect\theoremname}[section]
\theoremstyle{definition}
\newtheorem{defn}[thm]{\protect\definitionname}
\theoremstyle{plain}
\newtheorem{prop}[thm]{\protect\propositionname}
\theoremstyle{remark}
\newtheorem{rem}[thm]{\protect\remarkname}
\theoremstyle{plain}
\newtheorem{lem}[thm]{\protect\lemmaname}

\usepackage{extarrow}
\usepackage{enumerate}
\usepackage{fontenc}
\usepackage[T1]{fontenc}
\usepackage{graphicx}
\usepackage[colorlinks=true, linkcolor=blue]{hyperref}
\usepackage{latexsym}
\usepackage{mathrsfs}
\usepackage{mathtext}
\usepackage{tikz}
\usepackage{textcomp}
\usepackage{upgreek}
\usepackage{xy}

\usetikzlibrary{arrows}
\usetikzlibrary{matrix}
\usetikzlibrary{shapes}
\usetikzlibrary{snakes}
\usetikzlibrary{matrix}

\DeclareMathOperator{\Add}{\textup{Add}}
\DeclareMathOperator{\Aff}{\textup{Aff}}
\DeclareMathOperator{\Alg}{\textup{Alg}}
\DeclareMathOperator{\Ann}{\textup{Ann}}
\DeclareMathOperator{\Arr}{\textup{Arr}}
\DeclareMathOperator{\Art}{\textup{Art}}
\DeclareMathOperator{\Ass}{\textup{Ass}}
\DeclareMathOperator{\Aut}{\textup{Aut}}
\DeclareMathOperator{\Autsh}{\underline{\textup{Aut}}}
\DeclareMathOperator{\Bi}{\textup{B}}
\DeclareMathOperator{\CAdd}{\textup{CAdd}}
\DeclareMathOperator{\CAlg}{\textup{CAlg}}
\DeclareMathOperator{\CMon}{\textup{CMon}}
\DeclareMathOperator{\CPMon}{\textup{CPMon}}
\DeclareMathOperator{\CRings}{\textup{CRings}}
\DeclareMathOperator{\CSMon}{\textup{CSMon}}
\DeclareMathOperator{\CaCl}{\textup{CaCl}}
\DeclareMathOperator{\Cart}{\textup{Cart}}
\DeclareMathOperator{\Cl}{\textup{Cl}}
\DeclareMathOperator{\Coh}{\textup{Coh}}
\DeclareMathOperator{\Coker}{\textup{Coker}}
\DeclareMathOperator{\Cov}{\textup{Cov}}
\DeclareMathOperator{\Der}{\textup{Der}}
\DeclareMathOperator{\Div}{\textup{Div}}
\DeclareMathOperator{\End}{\textup{End}}
\DeclareMathOperator{\Endsh}{\underline{\textup{End}}}
\DeclareMathOperator{\Ext}{\textup{Ext}}
\DeclareMathOperator{\Extsh}{\underline{\textup{Ext}}}
\DeclareMathOperator{\FAdd}{\textup{FAdd}}
\DeclareMathOperator{\FCoh}{\textup{FCoh}}
\DeclareMathOperator{\FGrad}{\textup{FGrad}}
\DeclareMathOperator{\FLoc}{\textup{FLoc}}
\DeclareMathOperator{\FMod}{\textup{FMod}}
\DeclareMathOperator{\FPMon}{\textup{FPMon}}
\DeclareMathOperator{\FRep}{\textup{FRep}}
\DeclareMathOperator{\FSMon}{\textup{FSMon}}
\DeclareMathOperator{\FVect}{\textup{FVect}}
\DeclareMathOperator{\Fibr}{\textup{Fibr}}
\DeclareMathOperator{\Fix}{\textup{Fix}}
\DeclareMathOperator{\Fl}{\textup{Fl}}
\DeclareMathOperator{\Fr}{\textup{Fr}}
\DeclareMathOperator{\Funct}{\textup{Funct}}
\DeclareMathOperator{\GAlg}{\textup{GAlg}}
\DeclareMathOperator{\GExt}{\textup{GExt}}
\DeclareMathOperator{\GHom}{\textup{GHom}}
\DeclareMathOperator{\GL}{\textup{GL}}
\DeclareMathOperator{\GMod}{\textup{GMod}}
\DeclareMathOperator{\GRis}{\textup{GRis}}
\DeclareMathOperator{\GRiv}{\textup{GRiv}}
\DeclareMathOperator{\Gal}{\textup{Gal}}
\DeclareMathOperator{\Gl}{\textup{Gl}}
\DeclareMathOperator{\Grad}{\textup{Grad}}
\DeclareMathOperator{\Hilb}{\textup{Hilb}}
\DeclareMathOperator{\Hl}{\textup{H}}
\DeclareMathOperator{\Hom}{\textup{Hom}}
\DeclareMathOperator{\Homsh}{\underline{\textup{Hom}}}
\DeclareMathOperator{\ISym}{\textup{Sym}^*}
\DeclareMathOperator{\Imm}{\textup{Im}}
\DeclareMathOperator{\Irr}{\textup{Irr}}
\DeclareMathOperator{\Iso}{\textup{Iso}}
\DeclareMathOperator{\Isosh}{\underline{\textup{Iso}}}
\DeclareMathOperator{\Ker}{\textup{Ker}}
\DeclareMathOperator{\LAdd}{\textup{LAdd}}
\DeclareMathOperator{\LAlg}{\textup{LAlg}}
\DeclareMathOperator{\LMon}{\textup{LMon}}
\DeclareMathOperator{\LPMon}{\textup{LPMon}}
\DeclareMathOperator{\LRings}{\textup{LRings}}
\DeclareMathOperator{\LSMon}{\textup{LSMon}}
\DeclareMathOperator{\Left}{\textup{L}}
\DeclareMathOperator{\Lex}{\textup{Lex}}
\DeclareMathOperator{\Loc}{\textup{Vect}}
\DeclareMathOperator{\M}{\textup{M}}
\DeclareMathOperator{\ML}{\textup{ML}}
\DeclareMathOperator{\MLex}{\textup{MLex}}
\DeclareMathOperator{\Map}{\textup{Map}}
\DeclareMathOperator{\Mod}{\textup{Mod}}
\DeclareMathOperator{\Mon}{\textup{Mon}}
\DeclareMathOperator{\Ob}{\textup{Ob}}
\DeclareMathOperator{\Obj}{\textup{Obj}}
\DeclareMathOperator{\PDiv}{\textup{PDiv}}
\DeclareMathOperator{\PGL}{\textup{PGL}}
\DeclareMathOperator{\PML}{\textup{PML}}
\DeclareMathOperator{\PMLex}{\textup{PMLex}}
\DeclareMathOperator{\PMon}{\textup{PMon}}
\DeclareMathOperator{\Pic}{\textup{Pic}}
\DeclareMathOperator{\Picsh}{\underline{\textup{Pic}}}
\DeclareMathOperator{\Pro}{\textup{Pro}}
\DeclareMathOperator{\Proj}{\textup{Proj}}
\DeclareMathOperator{\QAdd}{\textup{QAdd}}
\DeclareMathOperator{\QAlg}{\textup{QAlg}}
\DeclareMathOperator{\QCoh}{\textup{QCoh}}
\DeclareMathOperator{\QMon}{\textup{QMon}}
\DeclareMathOperator{\QPMon}{\textup{QPMon}}
\DeclareMathOperator{\QRings}{\textup{QRings}}
\DeclareMathOperator{\QSMon}{\textup{QSMon}}
\DeclareMathOperator{\R}{\textup{R}}
\DeclareMathOperator{\Rep}{\textup{Rep}}
\DeclareMathOperator{\Rings}{\textup{Rings}}
\DeclareMathOperator{\Riv}{\textup{Riv}}
\DeclareMathOperator{\SFibr}{\textup{SFibr}}
\DeclareMathOperator{\SMLex}{\textup{SMLex}}
\DeclareMathOperator{\SMex}{\textup{SMex}}
\DeclareMathOperator{\SMon}{\textup{SMon}}
\DeclareMathOperator{\SchI}{\textup{SchI}}
\DeclareMathOperator{\Sh}{\textup{Sh}}
\DeclareMathOperator{\Soc}{\textup{Soc}}
\DeclareMathOperator{\Spec}{\textup{Spec}}
\DeclareMathOperator{\Specsh}{\underline{\textup{Spec}}}
\DeclareMathOperator{\Stab}{\textup{Stab}}
\DeclareMathOperator{\Supp}{\textup{Supp}}
\DeclareMathOperator{\Sym}{\textup{Sym}}
\DeclareMathOperator{\TMod}{\textup{TMod}}
\DeclareMathOperator{\Top}{\textup{Top}}
\DeclareMathOperator{\Tor}{\textup{Tor}}
\DeclareMathOperator{\Vect}{\textup{Vect}}
\DeclareMathOperator{\alt}{\textup{ht}}
\DeclareMathOperator{\car}{\textup{char}}
\DeclareMathOperator{\codim}{\textup{codim}}
\DeclareMathOperator{\degtr}{\textup{degtr}}
\DeclareMathOperator{\depth}{\textup{depth}}
\DeclareMathOperator{\divis}{\textup{div}}
\DeclareMathOperator{\et}{\textup{et}}
\DeclareMathOperator{\ffpSch}{\textup{ffpSch}}
\DeclareMathOperator{\h}{\textup{h}}
\DeclareMathOperator{\ilim}{\displaystyle{\lim_{\longrightarrow}}}
\DeclareMathOperator{\ind}{\textup{ind}}
\DeclareMathOperator{\indim}{\textup{inj dim}}
\DeclareMathOperator{\lf}{\textup{LF}}
\DeclareMathOperator{\op}{\textup{op}}
\DeclareMathOperator{\ord}{\textup{ord}}
\DeclareMathOperator{\pd}{\textup{pd}}
\DeclareMathOperator{\plim}{\displaystyle{\lim_{\longleftarrow}}}
\DeclareMathOperator{\pr}{\textup{pr}}
\DeclareMathOperator{\pt}{\textup{pt}}
\DeclareMathOperator{\rk}{\textup{rk}}
\DeclareMathOperator{\tr}{\textup{tr}}
\DeclareMathOperator{\type}{\textup{r}}
\DeclareMathOperator*{\colim}{\textup{colim}}

\usepackage[colorinlistoftodos]{todonotes}
\usepackage{enumitem}
\setenumerate[1]{label={\arabic*)}} 
\usepackage[a4paper]{geometry}

\theoremstyle{plain}
\newtheorem{thma}{Theorem}
\newtheorem{thmb}{Theorem}
\newtheorem{thmc}{Theorem}

\makeatother

\usepackage{babel}
\providecommand{\definitionname}{Definition}
\providecommand{\lemmaname}{Lemma}
\providecommand{\propositionname}{Proposition}
\providecommand{\remarkname}{Remark}
\providecommand{\theoremname}{Theorem}

\begin{document}
\title{Sheafification of linear functors}
\author{Fabio Tonini}
\address{Universitá degli Studi di Firenze, Dipartimento di Matematica e Informatica
\textquoteright Ulisse Dini\textquoteright , Viale Giovanni Battista
Morgagni, 67/A, 50134 Firenze, Italy}
\email{fabio.tonini@unifi.it}

\maketitle
\global\long\def\A{\mathbb{A}}%

\global\long\def\Ab{(\textup{Ab})}%

\global\long\def\C{\mathbb{C}}%

\global\long\def\Cat{(\textup{cat})}%

\global\long\def\Di#1{\textup{D}(#1)}%

\global\long\def\E{\mathcal{E}}%

\global\long\def\F{\mathbb{F}}%

\global\long\def\GCov{G\textup{-Cov}}%

\global\long\def\Gcat{(\textup{Galois cat})}%

\global\long\def\Gfsets#1{#1\textup{-fsets}}%

\global\long\def\Gm{\mathbb{G}_{m}}%

\global\long\def\GrCov#1{\textup{D}(#1)\textup{-Cov}}%

\global\long\def\Grp{(\textup{Grps})}%

\global\long\def\Gsets#1{(#1\textup{-sets})}%

\global\long\def\HCov{H\textup{-Cov}}%

\global\long\def\MCov{\textup{D}(M)\textup{-Cov}}%

\global\long\def\MHilb{M\textup{-Hilb}}%

\global\long\def\N{\mathbb{N}}%

\global\long\def\PGor{\textup{PGor}}%

\global\long\def\PGrp{(\textup{Profinite Grp})}%

\global\long\def\PP{\mathbb{P}}%

\global\long\def\Pj{\mathbb{P}}%

\global\long\def\Q{\mathbb{Q}}%

\global\long\def\RCov#1{#1\textup{-Cov}}%

\global\long\def\RR{\mathbb{R}}%

\global\long\def\Sch{\textup{Sch}}%

\global\long\def\WW{\textup{W}}%

\global\long\def\Z{\mathbb{Z}}%

\global\long\def\acts{\curvearrowright}%

\global\long\def\alA{\mathscr{A}}%

\global\long\def\alB{\mathscr{B}}%

\global\long\def\arr{\longrightarrow}%

\global\long\def\arrdi#1{\xlongrightarrow{#1}}%

\global\long\def\catC{\mathscr{C}}%

\global\long\def\catD{\mathscr{D}}%

\global\long\def\catF{\mathscr{F}}%

\global\long\def\catG{\mathscr{G}}%

\global\long\def\comma{,\ }%

\global\long\def\covU{\mathcal{U}}%

\global\long\def\covV{\mathcal{V}}%

\global\long\def\covW{\mathcal{W}}%

\global\long\def\duale#1{{#1}^{\vee}}%

\global\long\def\fasc#1{\widetilde{#1}}%

\global\long\def\fsets{(\textup{f-sets})}%

\global\long\def\iL{r\mathscr{L}}%

\global\long\def\id{\textup{id}}%

\global\long\def\la{\langle}%

\global\long\def\odi#1{\mathcal{O}_{#1}}%

\global\long\def\ra{\rangle}%

\global\long\def\set{(\textup{Sets})}%

\global\long\def\sets{(\textup{Sets})}%

\global\long\def\shA{\mathcal{A}}%

\global\long\def\shB{\mathcal{B}}%

\global\long\def\shC{\mathcal{C}}%

\global\long\def\shD{\mathcal{D}}%

\global\long\def\shE{\mathcal{E}}%

\global\long\def\shF{\mathcal{F}}%

\global\long\def\shG{\mathcal{G}}%

\global\long\def\shH{\mathcal{H}}%

\global\long\def\shI{\mathcal{I}}%

\global\long\def\shJ{\mathcal{J}}%

\global\long\def\shK{\mathcal{K}}%

\global\long\def\shL{\mathcal{L}}%

\global\long\def\shM{\mathcal{M}}%

\global\long\def\shN{\mathcal{N}}%

\global\long\def\shO{\mathcal{O}}%

\global\long\def\shP{\mathcal{P}}%

\global\long\def\shQ{\mathcal{Q}}%

\global\long\def\shR{\mathcal{R}}%

\global\long\def\shS{\mathcal{S}}%

\global\long\def\shT{\mathcal{T}}%

\global\long\def\shU{\mathcal{U}}%

\global\long\def\shV{\mathcal{V}}%

\global\long\def\shW{\mathcal{W}}%

\global\long\def\shX{\mathcal{X}}%

\global\long\def\shY{\mathcal{Y}}%

\global\long\def\shZ{\mathcal{Z}}%

\global\long\def\st{\ | \ }%

\global\long\def\stA{\mathcal{A}}%

\global\long\def\stB{\mathcal{B}}%

\global\long\def\stC{\mathcal{C}}%

\global\long\def\stD{\mathcal{D}}%

\global\long\def\stE{\mathcal{E}}%

\global\long\def\stF{\mathcal{F}}%

\global\long\def\stG{\mathcal{G}}%

\global\long\def\stH{\mathcal{H}}%

\global\long\def\stI{\mathcal{I}}%

\global\long\def\stJ{\mathcal{J}}%

\global\long\def\stK{\mathcal{K}}%

\global\long\def\stL{\mathcal{L}}%

\global\long\def\stM{\mathcal{M}}%

\global\long\def\stN{\mathcal{N}}%

\global\long\def\stO{\mathcal{O}}%

\global\long\def\stP{\mathcal{P}}%

\global\long\def\stQ{\mathcal{Q}}%

\global\long\def\stR{\mathcal{R}}%

\global\long\def\stS{\mathcal{S}}%

\global\long\def\stT{\mathcal{T}}%

\global\long\def\stU{\mathcal{U}}%

\global\long\def\stV{\mathcal{V}}%

\global\long\def\stW{\mathcal{W}}%

\global\long\def\stX{\mathcal{X}}%

\global\long\def\stY{\mathcal{Y}}%

\global\long\def\stZ{\mathcal{Z}}%

\global\long\def\then{\ \Longrightarrow\ }%

\global\long\def\L{\textup{L}}%

\global\long\def\l{\textup{l}}%

\begin{abstract}
We introduce ``sheafification'' functors from categories of (lax
monoidal) linear functors to categories of quasi-coherent sheaves
(of algebras) of stacks. They generalize the homogeneous sheafification
of graded modules for projective schemes.
\end{abstract}

\section*{Introduction}

If $\stX$ is a projective scheme over a ring $R$ with very ample
invertible sheaf $\odi{\stX}(1)$ then quasi-coherent sheaves on $\stX$
can be interpreted as graded modules over the homogeneous coordinate
ring $S_{\stX}$ of $\stX$. If $\GMod(-)$ denote the category of
graded modules then the Serre functor 
\[
\Gamma_{*}\colon\QCoh(\stX)\to\GMod(S_{\stX})\comma\Gamma_{*}(\shF)=\bigoplus_{n\in\Z}\Hl^{0}(\stX,\shF(n))
\]
is fully faithful. Moreover there is an inverse operation, called
homogeneous sheafification, which defines an exact functor $\widetilde{-}\colon\GMod(S_{\stX})\to\QCoh(\stX)$,
left adjoint to $\Gamma_{*},$ and such that $\widetilde{\Gamma_{*}(\shG)}\to\shG$
is an isomorphism for all $\shG\in\QCoh(\stX)$.

Let us consider now a very different situation. Let $G$ be an affine
group scheme over a field $k$. Classical Tannaka's reconstruction
problem consists in reconstructing the group $G$ and more generally
$G$-torsors from the category of representations $\Rep^{G}k$: a
$G$-torsor $\pi\colon P\to S$ is completely determined by the associated
exact \emph{strong} monoidal functor 
\[
(\pi_{*}\odi P\otimes-)^{G}\colon\Rep^{G}k\to\Loc(S)
\]
where $\Loc(S)$ denotes the category of locally free sheaves on $S$.

When $G$ is finite, in \cite{Tonini2013} I have introduced the notion
of $G$-cover $f:X\to S$ extending the notion of $G$-torsor and,
in order to handle the non abelian case, I showed in my Ph.D. thesis
\cite{Tonini2013a} that those objects can also be reconstructed from
the exact and \emph{lax} monoidal functors $(f_{*}\odi X\otimes-)^{G}\colon\Rep^{G}k\to\Loc(S)$.
The idea is that a $G$-cover is a weak version of a $G$-torsor,
therefore we have to look for a weak version of a strong monoidal
functor, that is, as the words suggest, a lax monoidal functor.

In this paper we develop a theory of sheafification functors which
generalizes the two above situations. Let us introduce some notations
and definitions. We fix a base commutative ring $R$ and a category
fibered in groupoids $\stX$ over $R$. We say that $\stX$ is \emph{pseudo-algebraic}
(resp. \emph{quasi-compact}) if there exist a scheme (resp. an affine
scheme) $X$ and a map $X\to\stX$ representable by fpqc covering
of algebraic spaces. We denote by $\QCoh\stX$ the category of quasi-coherent
sheaves on $\stX$ (see Section \ref{sec: Preliminaries on sheaves}).
Given a full subcategory $\shC$ of $\QCoh\stX$ we say that $\shC$
generates $\QCoh\stX$ if all quasi-coherent sheaves on $\stX$ are
quotient of a (possibly infinite) direct sum of sheaves in $\shC$.

In what follows $A$ will denote an $R$-algebra and $\shC$ a full
subcategory of $\QCoh\stX$. We define the \emph{Yoneda functors}
associated with $\shC$ as the $R$-linear contravariant functors
\[
\Omega^{\shF}\colon\shC\to\Mod A\comma\Omega_{\E}^{\shF}=\Hom_{\stX}(\pi^{*}\E,\shF)\text{ for }\shF\in\QCoh(\stX_{A})
\]
where $\pi\colon\stX_{A}=\stX\times_{R}A\to\stX$ is the projection.
We also denote by $\L_{R}(\shC,A)$ the category of contravariant
$R$-linear functors $\shC\to\Mod A$, so that we obtain a functor
\[
\Omega^{*}\colon\QCoh(\stX_{A})\to\L_{R}(\shC,A)
\]
When $R=k$, $\stX=\Bi G$, $\shC=\Loc\stX=\Rep^{G}k$ and $f:X\to\Spec A$
is a $G$-torsor or a $G$-cover one has that $f_{*}\odi{\stX}$ can
be thought of as a quasi-coherent sheaf of algebras over $\Bi G\times A$
and $\Omega^{f_{*}\odi X}=\Hom_{\Bi G}(-,f_{*}\odi X)=(f_{*}\odi X\otimes(-)^{\vee})^{G}$.
Up to a dual, this is the functor we associated with a $G$-cover
in the above discussion.

Now consider the case when $\stX$ is a projective scheme over a ring
$R$ with ample invertible sheaf $\odi{\stX}(1)$ and coordinate ring
$S_{\stX}$. Set $\shC_{\stX}=\{\odi{\stX}(n)\}_{n\in\Z}$. It is
easy to realize that $\L_{R}(\shC_{\stX},R)$ is equivalent to $\GMod(S_{\stX})$
and that $\Omega^{*}\colon\QCoh(\stX)\to\L_{R}(\shC,R)$ is just the
Serre functor $\Gamma_{*}\colon\QCoh(\stX)\to\GMod(\stX)$ we started
with (see \ref{sec:Quasi-projective-schemes and more} for details).

Once we have found this common language some questions arise naturally:
\begin{itemize}
\item Is the functor $\Omega^{*}\colon\QCoh(\stX_{A})\to\L_{R}(\shC,A)$
fully faithful like the Serre functor?
\item Does it also admits a left adjoint with the nice property of the sheafification
functor?
\item In the context of Tannaka duality sheaves of algebras yield (lax)
monoidal functors. Is it true in general?
\item Which functors are in the essential image of $\Omega^{*}$?
\end{itemize}
For the second problem, it turns out quite easily that also $\Omega^{*}\colon\QCoh(\stX_{A})\to\L_{R}(\shC,A)$
has a left adjoint 
\[
\shF_{*,\shC}\colon\L_{R}(\shC,A)\to\QCoh(\stX_{A})
\]
 when $\shC$ is essentially small: by analogy we call this functor
a \emph{sheafification }functor, which also explains the name of the
paper.

The third problem is also easy. If $\shC$ is a monoidal subcategory
of $\QCoh\stX$, $\ML_{R}(\shC,A)$ denotes the category of $R$-linear,
contravariant and (lax) monoidal functors $\shC\to\Mod A$ and $\QAlg(\stX)$
the category of quasi-coherent sheaves of algebras on $\stX$ then
$\Omega^{*}$ and $\shF_{*,\shC}$ extend to a pair of left adjoint
functors 
\[
\Omega^{*}\colon\QAlg(\stX_{A})\to\ML_{R}(\shC,A)\text{ and }\alA_{*,\shC}\colon\ML_{R}(\shC,A)\to\QAlg(\stX_{A})
\]

Coming back to the fully faithfulness of $\Omega^{*}$ we prove the
following.

\begin{thma} [\ref{prop:excatness properties of Omega* and shF*},\ref{thm:general theorem for equivalences}]\label{thma}
Let $\stX$ be a pseudo-algebraic category fibered in groupoids over
$R$ and $\shC\subseteq\QCoh\stX$ be a full subcategory generating
$\QCoh\stX$. Then the functor $\Omega^{*}\colon\QCoh(\stX_{A})\to\L_{R}(\shC,A)$
is fully faithful and, if $\shC$ is essentially small, the functor
$\shF_{*,\shC}\colon\L_{R}(\shC,A)\to\QCoh(\stX_{A})$ is exact and
the natural map $\shG\to\shF_{\Omega^{\shG},\shC}$ is an isomorphism.

\end{thma}

If $\stX$ is a projective scheme over $R$ then it is a classical
fact that $\shC_{\stX}=\{\odi{\stX}(n)\}_{n\in\Z}$ generates $\QCoh X$
and thus we recover the classical properties of $\Gamma_{*}\colon\QCoh(\stX_{A})\to\GMod(S_{\stX}\otimes_{R}A)$.
We moreover extend this to quasi-compact and quasi-projective schemes
over $R$ and to certain quotient stacks (see \ref{thm:for projective schemes}).

Theorem above when $\shC$ consists of a single object is a rephrasing
of classical Gabriel-Popescu's theorem for the category $\QCoh\stX$
(see \ref{thm: Gabriel Popescu's theorem}). When $\shC$ is monoidal
and generates $\QCoh\stX$ we also have that $\Omega^{*}\colon\QAlg(\stX_{A})\to\ML_{R}(\shC,A)$
is fully faithful (see \ref{thm:fundamental for sheaf of algebras and monoidal functors}).

The last problem we address is to describe the essential image of
$\Omega^{*}$. The main idea is just that $\Hom_{\stX}(-,\shF)$ for
$\shF\in\QCoh\stX$ is a left exact functor. Since the domain $\shC$
of the functors $\Omega^{\shF}$ is not abelian we need an ad hoc
definition of exactness. A \emph{test sequence }in $\shC$ is an exact
sequence (in the ambient category $\QCoh\stX$) 
\[
\shT_{*}\colon\bigoplus_{k\in K}\E_{k}\to\bigoplus_{i\in I}\E_{i}\to\E\to0\text{ where }\E,\E_{i},\E_{k}\in\shC\text{ and }I,J\text{ are sets}
\]
such that $\alpha(\E_{k})$ is contained in a finite sum for all $k\in K$.
A test sequence is called finite if $K$ and $I$ are finite sets.
Given $\Gamma\in\L_{R}(\shC,A)$ we will say that $\Gamma$ is exact
on a test sequence $\shT_{*}$ in $\shC$ if the complex of $A$-modules
(see \ref{def: left exact functors})
\[
0\to\Gamma_{\E}\to\prod_{i\in I}\Gamma_{\E_{i}}\to\prod_{k\in K}\Gamma_{\E_{k}}
\]
is exact. We denote by $\Lex_{R}(\shC,A)$ (resp. $\MLex_{R}(\shC,A)$
if $\shC$ is monoidal) the full subcategory of $\L_{R}(\shC,A)$
(resp. $\ML_{R}(\shC,A)$) of functors which are exact an all test
sequences. We have the following:

\begin{thmb}[\ref{thm:general theorem for equivalences}, \ref{ prop: Lex and left exact functors}]\label{thmb}
Let $\stX$ be a pseudo-algebraic category fibered in groupoids over
$R$ and $\shC\subseteq\QCoh\stX$ be a full subcategory generating
$\QCoh\stX$. Then $\Lex_{R}(\shC,A)$ is the essential image of the
(fully faithful) functor $\Omega^{*}\colon\QCoh(\stX_{A})\to\L_{R}(\shC,A)$.

If $\stX$ is quasi-compact and all sheaves in $\shC$ are finitely
presented then $\Lex_{R}(\shC,A)$ is the subcategory of $\L_{R}(\shC,A)$
of functors which are exact on finite test sequences.

\end{thmb}

In particular, when $\shC$ is essentially small, 
\[
\Omega^{*}\colon\QCoh(\stX_{A})\to\Lex_{R}(\shC,A)\text{ and }\shF_{*,\shC}\colon\Lex_{R}(\shC,A)\to\QCoh(\stX_{A})
\]
are quasi-inverses of each other.

When $R=A=\Z$, $\stX$ is a quasi-compact and quasi-separated scheme
and $\shC=\Loc(\stX)$, the category of locally free sheaves on $\stX$,
the results above have already been proved in \cite[Section 3.1]{Bhatt2014}:
the conclusion, which also extends in our setting, is that if $\stX$
has the resolution property (that is $\Loc(\stX)$ generates $\QCoh(\stX)$)
then $\QCoh\stX$ can be recovered by $\Loc(\stX)$ and the distinguished
class of morphisms which are surjective in the ambient category $\QCoh\stX$. 

When $\shC$ is monoidal we obtain analogous statements for monoidal
functors by replacing $\QCoh$, $\L_{R}$, $\Lex_{R}$ and $\shF_{*,\shC}$
with $\QAlg$, $\ML_{R}$, $\MLex_{R}$ and $\alA_{*,\shC}$ respectively
(see \ref{thm:fundamental for sheaf of algebras and monoidal functors}).

Theorem \ref{thma} and \ref{thmb} apply in the following situations
in which $\shC$ generates $\QCoh(\stX)$:
\begin{itemize}
\item if $\shC=\QCoh\stX$ then $\Lex_{R}(\QCoh\stX,A)$ is the category
of contraviant, $R$-linear and left exact functors $\QCoh\stX\to\Mod A$
which transform direct sums into products (see \ref{thm:equivalence for QCoh X});
\item if $\shC=\Coh\stX$ (resp. $\shC=\Loc(\stX)$) and $\stX$ is a Noetherian
algebraic stack (resp. $\stX$ is quasi-compact and has the resolution
property) then $\Lex_{R}(\shC,A)$ is the category of contraviant,
$R$-linear and left exact functors $\shC\to\Mod A$ (see \ref{thm:equivalence for Loc X}
and \ref{thm:equivalence for noetherian algebraic stacks});
\item if $\shC=\QCoh_{\text{fp}}\stX$, the category of quasi-coherent sheaves
of finite presentations, and $\stX$ is a quasi-compact and quasi-separated
scheme then $\Lex_{R}(\QCoh_{\text{fp}}\stX,A)$ is the category of
contraviant, $R$-linear functors $\QCoh_{\text{fp}}\stX\to\Mod A$
which are left exact on right exact sequences in $\QCoh_{\text{fp}}\stX$
(see \ref{thm:equivalence for finite presented sheaves}).
\end{itemize}
When $\shC$ is essentially small there is another cohomological characterization
of the functors in $\Lex_{R}(\shC,A)$. A collection of maps $\shU=\{\E_{i}\to\E\}$
in $\shC$ is called jointly surjective if the map $\bigoplus_{i\in I}\E_{i}\to\E$
is surjective. Given such a collection $\shU$ we set $\Delta_{\shU}=\Imm(\bigoplus_{i\in I}\Omega^{\E_{i}}\to\Omega^{\E})\in\L_{R}(\shC,R)$.
Denote by $\shC^{\oplus}$ the subcategory of $\QCoh\stX$ consisting
of all possible finite direct sums of sheaves in $\shC$. We have:

\begin{thmc}[\ref{prop:test sequences and Extone}, \ref{ prop: Lex and left exact functors}
and \ref{prop:Lex and Grothendieck topology}] Let $\stX$ be a pseudo-algebraic
category fibered in groupoids over $R$ and $\shC\subseteq\QCoh\stX$
be a full and essentially small subcategory generating $\QCoh\stX$.
Then $\Lex_{R}(\shC,A)$ is the full subcategory of $\L_{R}(\shC,A)$
of functors $\Gamma$ satisfying 
\[
\Hom_{\L_{R}(\shC,R)}(\Omega^{\E}/\Delta_{\shU},\Gamma)=\Ext_{\L_{R}(\shC,R)}^{1}(\Omega^{\E}/\Delta_{\shU},\Gamma)=0
\]
for all jointly surjective collections of maps $\shU=\{\E_{i}\to\E\}_{i\in I}$
in $\shC$. If $\stX$ is quasi-compact and the sheaves in $\shC$
are finitely presented we can consider only finite collections $\shU$.

We have $\Lex_{R}(\shC^{\oplus},A)\simeq\Lex_{R}(\shC,A)$ via the
restriction $\shC\to\shC^{\oplus}$ and, if $\shC$ is additive and
$\shJ$ is the smallest Grothendieck topology on $\shC$ containing
the sieves $\Delta_{\shU}$ for all jointly surjective collections
$\shU=\{\E_{i}\to\E\}_{i\in I}$ in $\shC$, then $\Lex_{R}(\shC,A)$
coincides with the category of sheaves of $A$-modules $\shC^{\op}\to\Mod A$
on the site $(\shC,\shJ)$ which are $R$-linear.

\end{thmc}

Theory above has been applied in \cite{Schappi2015} to show the existence
of colimits in the $2$-category of quasi-compact stacks for the fpqc
topology with affine diagonal and in \cite{Tonini2020}, to show a
version of Tannaka's reconstruction for stacks and study certain stacks
of fiber functors.

We now describe some other applications of the above theory.

The first, which is also the original motivation, is the theory of
Galois covers and it is developed in \cite{Tonini2015}. As explained
in the beginning of this introduction, in my Ph.D. thesis \cite{Tonini2013a}
I have worked out theory above in the case $\stX=\Bi G$ and $\shC=\Loc\stX$,
where $G$ is a finite, flat and finitely presented group scheme over
$R$ ($\shC$ generates $\QCoh(\Bi G)$ is this case, see \ref{rem: resolution property for BG cases}).
The proof presented in \cite{Tonini2013a} makes use of representation
theory and can not be generalized to arbitrary categories fibered
in groupoids. Moreover the results in the present paper also apply
to more general affine group schemes over $R$, for instance any flat
affine group scheme defined over a Dedekind domain (see \ref{rem: resolution property for BG cases}
and \ref{thm:general theorem for BG}). The goal of \cite{Tonini2013a}
and \cite{Tonini2015} was to look at Galois covers with group $G$
as particular monoidal functors, as $G$-torsors can be thought of
as particular strong monoidal functors, and the motivation was the
study of non-abelian Galois covers, where a direct approach as in
the abelian case (see \cite{Tonini2013}) fails due to the complexity
of the representation theory.

A second application, which will be hopefully the subject of a future
paper, is to the theory of Cox rings and homogeneous sheafifications.
This is a generalization of the results on quasi-projective schemes
described in Section \ref{sec:Quasi-projective-schemes and more}.
The idea is to consider $\shC_{H}=\{\shL\}_{\shL\in H}\subseteq\Loc\stX$
where $H$ is a subgroup of $\Pic\stX$. As in the projective case
we have a homogeneous coordinate ring
\[
S_{H}=\bigoplus_{\shL\in H}\Hl^{0}(\stX,\shL)
\]
(opportunely defined, see \cite[Theorem A]{Hochenegger2020}), $\L_{R}(\shC_{H},A)$
is equivalent to $\GMod(S_{H}\otimes_{R}A)$, the category of $H$-graded
$(S_{H}\otimes_{R}A)$-modules, $\Omega^{*}$ corresponds to
\[
\Gamma_{*}\colon\QCoh(\stX_{A})\to\GMod(S_{H}\otimes_{R}A)\comma\Gamma_{*}(\shF)=\bigoplus_{\shL\in H}\Hl^{0}(\stX,\shF\otimes\shL)
\]
and its adjoint $\shF_{*,\shC_{H}}$ behaves like a homogeneous sheafification.
Moreover in more concrete geometric situations, e.g. when $\stX$
is a normal variety, there are analogous constructions for reflexive
sheaves of rank $1$ (see \cite[Theorem B]{Hochenegger2020}). We
expect that this theory covers all known cases where $\Gamma_{*}$
is proved to be fully faithful (see for instance \cite[Appendix of Chapter 6]{Cox2011}
and \cite[Section 2]{Kajiwara1998}).

As a last application there is a generalization of Drinfeld's lemma,
which will be the subject of a future paper. Let $\stX$ be a fibered
category over a finite field $\F_{q}$, $\F_{q}\subseteq k$ an algebraically
closed field, $\phi_{k}\colon k\to k$, $\phi_{k}(x)=x^{q}$ the Frobenius,
$\stX_{k}=\stX\times_{\F_{q}}k$ and $\phi_{k}=\id\times\phi_{k}\colon\stX_{k}\to\stX_{k}$
the geometric Frobenius. Then there is a natural functor

\[
\Psi_{\stX}\colon\QCoh_{\text{fp}}(\stX)\to\QCoh_{\text{fp}}(\stX,\phi_{k})
\]
where $\QCoh_{\text{fp}}(\stX)$ is the category of finitely presented
quasi-coherent sheaves and $\QCoh_{\text{fp}}(\stX,\phi_{k})$ is
the category of pairs $(\shF,\sigma)$ where $\shF\in\QCoh_{\text{fp}}(\stX_{k})$
and $\sigma\colon\phi_{k}^{*}\shF\to\shF$ is an isomorphism. Drinfeld
proved that the functor $\Psi_{\stX}$ is an equivalence when $\stX$
is a projective scheme over $\F_{q}$ and its proof used the classical
correspondence between quasi-coherent sheaves on a projective space
and graded modules. Using the generalized correspondence between sheaves
and linear functors, it is possible to extend Drinfeld argument to
more general schemes and stacks, for instance proper algebraic stacks
or arbitrary affine gerbes over $\F_{q}$.

The paper is divided as follows. In the first section we explain how
to work with quasi-coherent sheaves on fibered categories, while in
the second one we introduce sheafification functors and prove basic
properties. In the third section we discuss when the Yoneda functor
$\Omega^{*}$ is fully faithful, while in the fourth and fifth ones
we determine its essential image. The last sections are about applications.

\section*{Notation}

In this paper we work over a base commutative, associative ring $R$
with unity. If not stated otherwise a fiber category will be a category
fibered in groupoids over $\Aff/R$, the category of affine schemes
over $\Spec R$, or, equivalently, the opposite of the category of
$R$-algebras. Recall that by the $2$-Yoneda lemma objects of a fibered
category $\stX$ can be thought of as maps $T\to\stX$ from an affine
scheme. An fpqc stack will be a stack for the fpqc topology.

A map $f\colon\stX'\to\stX$ of fibered categories is called representable
if for all maps $T\to\stX$ from an affine scheme (or an algebraic
space) the fiber product $T\times_{\stX}\stX'$ is (equivalent to)
an algebraic space.

Given a flat and affine group scheme $G$ over $R$ we denote by $\Bi_{R}G$
the stack of $G$-torsors for the \emph{fpqc} topology, which is an
fpqc stack with affine diagonal. When $G\to\Spec R$ is finitely presented
(resp. smooth) then $\Bi_{R}G$ coincides with the stack of $G$-torsors
for the fppf (resp. étale) topology.

By a ``subcategory'' of a given category we mean a ``full subcategory''
if not stated otherwise.

We are going to look at functors whose source category $\shC$ will
be a full subcategory of quasi-coherent sheaves of some fibered category.
Many construction makes sense only when $\shC$ is small, yet the
results obtained are clearly true also for an essentially small category
$\shC$. Therefore in the text we will only consider small categories,
but keeping in mind the previous observation.

\section*{Acknowledgments}

I would like to thank Jarod Alper, Daniel Schäppi, David Rydh, Mattia
Talpo and Angelo Vistoli for the useful conversations I had with them
and all the suggestions they gave me.

\section{\label{sec: Preliminaries on sheaves} Preliminaries on sheaves and
fibered categories}

Many definition and properties showed in this section can be also
found in \cite[Section 4.1]{Talpo2018}. The notion of modules and
quasi-coherent sheaves makes sense on an arbitrary ringed site.
\begin{defn}
\label{def:quasi-coherent site} Let $(\shC,\odi{\shC})$ be a ringed
category. A presheaf of $\odi{\shC}$-modules if a contravariant functor
$\shF\colon\shC\to\Ab$ endowed, for any $T\in\shC$, of an $\odi{\shC}(T)$-module
structure on $\shF(T)$ such that, for any $T'\to T$ in $\shC$ the
map $\shF(T)\to\shF(T')$ is $\odi{\shC}(T)$-linear. We denote by
$\Mod(\shC,\odi{\shC})$ or simply $\Mod\odi{\shC}$ the category
of presheaves of $\odi{\shC}$-modules.

Assume now $(\shC,\odi{\shC})$ is a ringed site. A quasi-coherent
sheaf on $\shC$ is a sheaf of $\odi{\shC}$-modules such that, for
any $T\in\shC$, there exist a covering $\{T_{i}\to T\}$ and an exact
sequence of sheaves on the restricted site $\shC/T_{i}$
\[
(\odi{\shC})_{|T_{i}}^{(I)}\to(\odi{\shC})_{|T_{i}}^{(J)}\to\shF_{|T_{i}}\to0
\]
where $I$ and $J$ are sets and $\shG_{|T_{i}}=\shG\circ r_{T_{i}}$,
$r_{T_{i}}\colon\shC/T_{i}\to\shC$ and $\shG\in\Mod\odi{\shC}$.
We denote by $\QCoh\shC$ the category of quasi-coherent sheaves on
$\shC$.
\end{defn}

Let $\pi\colon\stX\to\Aff/R$ be a fibered category. There is a functor
of rings $\odi{\stX}\colon\stX^{\op}\to\sets$ defined by $\odi{\stX}(\xi)=\Hl^{0}(\odi{\pi(\xi)})$,
so that $(\stX,\odi{\stX})$ is a ringed category.
\begin{defn}
\label{def:quasi-coherent fibered} A quasi-coherent sheaf over $\stX$
is a presheaf of $\odi{\stX}$-modules such that for all maps $\xi\to\eta$
in $\stX$ the induced map
\[
\shF(\eta)\otimes_{\Hl^{0}(\odi{\pi(\eta)})}\Hl^{0}(\odi{\pi(\xi)})\to\shF(\xi)
\]
is an isomorphism. We denote by $\QCoh\stX$ the full subcategory
of $\Mod\odi{\stX}$ of quasi-coherent sheaves.
\end{defn}

If $\catC$ is a (not full) subcategory of $\stX$ we similarly define
presheaves of $(\odi{\stX})_{|\catC}$-modules and quasi-coherent
sheaves on $\catC$ just replacing all occurrences of $\stX$ with
$\catC$. We denote by $\Mod(\odi{\stX})_{|\catC}$ and $\QCoh\catC$
the resulting categories.
\begin{prop}
Let $\stX\to\Aff/R$ be a fibered category and $\shF\in\Mod\odi{\stX}$.
Then $\shF$ is quasi-coherent if and only if it is quasi-coherent
on the site $\stX$ with the Zariski (étale, fppf, fpqc) topology.
In particular quasi-coherent sheaves are sheaves in the fpqc topology.
\end{prop}

\begin{proof}
Let us denote by $\stX_{Z}$, $\stX_{E}$, $\stX_{F}$, $\stX_{\overline{F}}$
the category $\stX$ with the Zariski, étale, fppf, fpqc topology
respectively. Inside $\Mod\odi{\stX}$ there are inclusions
\[
\QCoh\stX_{Z}\subseteq\QCoh\stX_{E}\subseteq\QCoh\stX_{F}\subseteq\QCoh\stX_{\overline{F}}
\]
Moreover it is clear that we can assume that $\stX$ is an affine
scheme, in which case the result has been proved in \cite[Prop 4.13]{Talpo2018}.
\end{proof}
If $f\colon\stY\to\stX$ is a morphism of fibered categories and $\shF\in\Mod\odi{\stX}$
we define $f^{*}\shF=\shF\circ f\colon\stY^{\op}\to\stX^{\op}\to\sets$.
This association defines a functor $f^{*}\colon\Mod\odi{\stX}\to\Mod\odi{\stY}$,
called the pull-back functor, and restricts to a functor $f^{*}\colon\QCoh\stX\to\QCoh\stY$.
Notice that $f^{*}\odi{\stX}=\odi{\stY}$ tautologically.

The category $\Mod\odi{\stX}$ is an abelian category: if $\alpha\colon\shF\to\shG$
is a map and $\xi\colon\Spec B\to\stX$ an object then
\[
\widetilde{\Ker}(\alpha)(\xi)=\Ker(\alpha_{\xi}\colon\shF(\xi)\to\shG(\xi))\comma\widetilde{\Imm}(\alpha)(\xi)=\Imm(\alpha_{\xi}\colon\shF(\xi)\to\shG(\xi))
\]
\[
\Coker(\alpha)(\xi)=\Coker(\alpha_{\xi}\colon\shF(\xi)\to\shG(\xi))
\]
are the kernel, image and cokernel of $\alpha$ in $\Mod\odi{\stX}$.
If $\shF$ and $\shG$ are quasi-coherent then so is $\Coker(\alpha)$
because our category $\stX$ is fibered over affine schemes. In particular
$\alpha\colon\shF\to\shG$ is an epimorphism if and only if it is
surjective objectwise. On the other hand kernels and images are almost
never quasi-coherent, essentially because pull-backs are not left
exact. The category $\QCoh\stX$ is $R$-linear but it is unclear
if it is abelian. There is a natural condition on $\stX$ which allows
us to prove that $\QCoh\stX$ is an $R$-linear abelian category.
\begin{defn}
An \emph{fpqc atlas }(or simply \emph{atlas}) of a fibered category
$\stX$ is a representable fpqc covering $X\to\stX$ from a scheme.
A fiber category is called \emph{pseudo-algebraic }if it has an atlas,
it is called \emph{quasi-compact }if it has an atlas from an affine
scheme.

Let$f\colon\stY\to\stX$ be a morphism of fibered categories. The
map $f$ is called \emph{pseudo-algebraic} (resp. \emph{quasi-compact})
if for all maps $T\to\stX$ from a scheme (resp. quasi-compact scheme)
the fiber product $T\times_{\stX}\stY$ is pseudo-algebraic (resp.
quasi-compact). It is called \emph{quasi-separated} if the diagonal
$\stY\to\stY\times_{\stX}\stY$ is quasi-compact.
\end{defn}

\begin{rem}
Notice that in \cite[Definition 3.18]{Talpo2018} an atlas $X\to\stX$
is required to be schematically representable. In this paper we don't
need this assumption. Clearly if $\stX$ has schematically representable
diagonal the two notions agree.
\end{rem}

\begin{rem}
A map of fibered categories which is locally (in some topology) representable
is not in general representable: if $F\to S$ is a map of functors,
$T\to S$ is a covering in some topology and $F\times_{S}T$ is an
algebraic space we cannot conclude that $F$ is a sheaf in the same
topology. On the other hand a map of stacks in the étale topology
which is fppf locally representable it is representable: this is because
a sheaf in the étale topology which is fppf local an algebraic space
is an algebraic space. It is not clear if a map of fpqc stacks which
is fpqc locally (schematically) representable is representable. Again
the problem is the following: if an fpqc sheaf is fpqc locally on
the base an algebraic space (or even a scheme), then is it an algebraic
space itself?

In our situation we can see that if $\stX$ is pseudo-algebraic fibered
category then the diagonal $\Delta_{\stX}\colon\stX\to\stX\times_{R}\stX$
is only fpqc locally representable.
\end{rem}

Let $f\colon\stY\to\stX$ be a map of fibered categories. If $\stX$
and $f$ are pseudo-algebraic then $\stY$ is pseudo-algebraic. If
$\stY$ is pseudo-algebraic and $\Delta_{\stX}$ is representable
then $f$ is pseudo-algebraic.
\begin{defn}
We define $\stX_{\text{fl}}$ (resp. $\stX_{\text{sm}}$, $\stX_{\text{et}}$)
as the (not full) subcategory of $\stX$ of objects $\xi\colon\Spec B\to\stX$
which are representable and flat (resp. smooth, étale) and the arrows
are morphisms in $\stX$ whose underlying map of affine schemes is
flat (resp. smooth, étale).
\end{defn}

If $X\to\stX$ is an fpqc atlas then by definition $V=X\times_{\stX}X$
is an algebraic space and the two projections $V\rightrightarrows X$
extends to a groupoid in algebraic spaces. We denote by $\QCoh(V\rightrightarrows X)$
the category of quasi-coherent sheaves on $V\rightrightarrows X$
(see \cite[Tag 0440]{SP014}). By standard arguments of fpqc descent
for modules we have (see also \cite[Proposition 4.6]{Talpo2018}):

\begin{prop}
\label{prop:reducing to atlases for quasi-geometric categories} If
$\stX$ admits an fpqc (resp. smooth, étale) atlas then the restriction
$\QCoh\stX\to\QCoh\stX_{\text{fl}}$ (resp. $\QCoh\stX_{\text{sm}}$,
$\QCoh\stX_{\text{et}}$) is an equivalence of categories. If $f\colon X\to\stX$
is an fpqc atlas then $f^{*}\colon\QCoh\stX\to\QCoh X$ is faithful
and it induces an equivalence $\QCoh\stX\to\QCoh(V\rightrightarrows X)$.
\end{prop}

We see that if $\stX$ is pseudo-algebraic then $\QCoh\stX$ is equivalent
to an $R$-linear abelian category, namely $\QCoh(R\rightrightarrows X)$.
Moreover if $\alpha\colon\shF\to\shG$ is a map of quasi-coherent
sheaves then $\Ker(\alpha)$ is defined by taking $\Ker(\alpha_{|\stX_{\text{fl}}})\in\QCoh\stX_{\text{fl}}$,
which is just given by $\Ker(\alpha_{|\stX_{\text{fl}}})(\Spec B\to\stX)=\Ker(\alpha(\xi)\colon\shF(\xi)\to\shG(\xi))$
for $\xi\in\stX_{\text{fl}}$, and then extending it to the whole
$\stX$. If $\stX$ is an algebraic stack or a scheme we see that
$\QCoh\stX$ is equivalent to the usual category of quasi-coherent
sheaves via an $R$-linear and exact functor.
\begin{prop}
\label{prop:right exact sequence for pseudo-algebraic} If $\stX$
is pseudo-algebraic then a sequence $\shF\to\shG\to\shH\to0$ of quasi-coherent
sheaves is exact in $\QCoh(\stX)$ if and only if it is exact in $\Mod(\odi{\stX})$.
In particular if $f\colon\stY\to\stX$ is a map from a pseudo-algebraic
fibered category then $f^{*}\colon\QCoh(\stX)\to\QCoh(\stY)$ is right
exact on right exact sequences.
\end{prop}

\begin{proof}
Via the equivalence $\QCoh(\stX)\to\QCoh(\stX_{\text{fl}})$ the sequence
$\shS$ is the statement is exact in $\QCoh(\stX)$ if and only if
$\shS(B)$ is exact for all $\Spec B\to\stX$ flat and representable.
So if $\shS$ is exact in $\Mod(\stX)$ then it is exact in $\QCoh(\stX)$.
Conversely if $X\to\stX$ is an atlas and $\Spec B\to\stX$ is any
map then $\shS(B)$ is exact because it is so after the fpqc covering
$X\times_{\stX}B\to\Spec B$.
\end{proof}
\begin{defn}
Given a subcategory $\shD$ of $\QCoh\stX$, by an exact sequence
of sheaves in $\shD$ will always mean exact sequence in $\QCoh\stX$
of sheaves belonging to $\shD$.
\end{defn}

We now deal with the problem of defining a right adjoint of a pull-back
functor, that is a push-forward. Given $\shF\in\Mod\odi{\stX}$ we
define the global section $\shF(\stX)=\Hom(\odi{\stX},\shF)$ of $\shF$,
also denoted by $\Hl^{0}(\stX,\shF)$ or simply $\Hl^{0}(\shF)$,
which is an $\odi{\stX}(\stX)$-module. More generally given a map
of fibered categories $g\colon\stZ\to\stX$ we define $\shF(\stZ\to\stX)=(g^{*}\shF)(\stZ)$,
sometimes simply written as $\shF(\stZ)$. If $\stZ=\Spec B$ is affine
we will often write $\shF(B)$ instead of $\shF(\Spec B)$.

Let $f\colon\stY\to\stX$ be a map of fibered categories. The Weil
restriction defines a functor $f_{p}\colon\Mod\odi{\stY}\to\Mod\odi{\stX}$:
given $\shH\in\Mod\odi{\stY}$ and an object $\xi\colon T\to\stX$
of $\stX$ we define 
\[
(f_{p}\shH)(\xi)=\shH(T\times_{\stX}\stY)
\]

\begin{prop}
Let $f\colon\stY\to\stX$ be a map of fibered categories. Then $f_{p}$
is a right adjoint of $f^{*}$ and, if   \[   \begin{tikzpicture}[xscale=1.5,yscale=-1.2]     \node (A0_0) at (0, 0) {$\stY'$};     \node (A0_1) at (1, 0) {$\stY$};     \node (A1_0) at (0, 1) {$\stX'$};     \node (A1_1) at (1, 1) {$\stX$};     \path (A0_0) edge [->]node [auto] {$\scriptstyle{g'}$} (A0_1);     \path (A1_0) edge [->]node [auto] {$\scriptstyle{g}$} (A1_1);     \path (A0_1) edge [->]node [auto] {$\scriptstyle{f}$} (A1_1);     \path (A0_0) edge [->]node [auto] {$\scriptstyle{f'}$} (A1_0);   \end{tikzpicture}   \] 
is a $2$-cartesian diagram of fibered categories, there is an isomorphism
of functors
\[
g^{*}f_{p}\to f'_{p}g'^{*}\colon\Mod\odi{\stY}\to\Mod\odi{\stX'}
\]
If $f$ is affine then $f_{p}(\QCoh\stY)\subseteq\QCoh\stX$ and $(f_{p})_{|\QCoh\stY}\colon\QCoh\stY\to\QCoh\stX$
is right adjoint to $f^{*}\colon\QCoh\stX\to\QCoh\stY$.
\end{prop}

\begin{proof}
We define $f^{*}f_{p}\shH\to\shH$ for $\shH\in\Mod\odi{\stY}$ as
\[
f^{*}(f_{p}\shH)(T\to\stY)=(f_{p}\shH)(T\to\stY\to\stX)=\shH(\stY\times_{\stX}T\to\stY)\to\shH(T\to\stY)
\]
Conversely we define $\shG\to f_{p}f^{*}\shG$ for $\shG\in\Mod\odi{\stX}$
as 
\[
\shG(T\to\stX)\to\shG(\stY\times_{\stX}T\to T\to\stX)=f^{*}\shG(\stY\times_{\stX}T\to\stY)=f_{p}(f^{*}\shG)(T\to\stX)
\]
where the last map is induced by the given section $T\to T\times_{\stX}\stY$.
We have to prove that the corresponding maps
\[
\Psi\colon\Hom(f^{*}\shG,\shH)\to\Hom(\shG,f_{p}\shH)\comma\Phi\colon\Hom(\shG,f_{p}\shH)\to\Hom(f^{*}\shG,\shH)
\]
are inverses of each other. Given a map $T\to\stX$ we set $\stY_{T}=\stY\times_{\stX}T$.
Given $\sigma\colon f^{*}\shG\to\shH$ and $T\to\stY$ we have commutative
diagrams     \[   \begin{tikzpicture}[xscale=3.8,yscale=-1.2]     \node (A0_0) at (0, 0.3) {$\shG(T\to \stX)$};     \node (A0_1) at (1, 0.3) {$\shG(\stY_T\to\stX)$};     \node (A1_0) at (0, 1) {$f^*\shG(T\to \stY)$};     \node (A1_1) at (1, 1) {$f^*\shG(\stY_T\to \stY)$};     \node (A1_2) at (2, 1) {$\shH(\stY_T \to \stY)$};     \node (A2_1) at (1, 2) {$f^*\shG(T\to \stY)$};     \node (A2_2) at (2, 2) {$\shH(T\to \stY)$};     
\node[rotate=90] (A3_1) at (0, 0.6) {$=$};     \node[rotate=90] (A3_2) at (1, 0.6) {$=$};     
\path (A1_0) edge [->]node [auto,swap] {$\scriptstyle{\id}$} (A2_1);     \path (A0_0) edge [->]node [auto] {$\scriptstyle{}$} (A0_1);     \path (A1_0) edge [->]node [auto] {$\scriptstyle{a}$} (A1_1);     \path (A1_1) edge [->]node [auto] {$\scriptstyle{\sigma_{\stY_T}}$} (A1_2);     \path (A1_1) edge [->]node [auto] {$\scriptstyle{}$} (A2_1);     \path (A2_1) edge [->]node [auto] {$\scriptstyle{\sigma_T}$} (A2_2);     \path (A1_2) edge [->]node [auto] {$\scriptstyle{c}$} (A2_2);   \end{tikzpicture}   \] The vertical maps are induced by the section $T\to\stY_{T}$, while
$a$ by the projection $\stY_{T}\to T$. The composition $c\sigma_{\stY_{T}}a$
equals $\Phi\Psi(\sigma)_{T}$ and, thanks to the above diagram, $\Phi\Psi(\sigma)_{T}=\sigma$.
Conversely let $\delta\colon\shG\to f_{p}\shH$ and $T\to\stX$. We
have a commutative diagram    \[   \begin{tikzpicture}[xscale=2.3,yscale=-1.2]     \node (A0_0) at (0, 0) {$\shG(T)$};     \node (A0_1) at (1, 0) {$f_p\shH(T)$};     \node (A1_0) at (0, 1) {$\shG(\stY_T)$};     \node (A1_1) at (1, 1) {$f_p \shH(\stY_T)$};     \node (A1_2) at (2, 1) {$f_p\shH (T)$};     \path (A0_1) edge [->]node [auto] {$\scriptstyle{\id}$} (A1_2);     \path (A0_0) edge [->]node [auto] {$\scriptstyle{\delta_T}$} (A0_1);     \path (A1_0) edge [->]node [auto] {$\scriptstyle{\delta_{\stY_T}}$} (A1_1);     \path (A1_1) edge [->]node [auto] {$\scriptstyle{c}$} (A1_2);     \path (A0_0) edge [->]node [auto] {$\scriptstyle{a}$} (A1_0);     \path (A0_1) edge [->]node [auto] {$\scriptstyle{}$} (A1_1);   \end{tikzpicture}   \] The
vertical maps are induced by $\stY_{T}\to T$, while $c$ by the diagonal
$\stY_{T}\to\stY_{T}\times_{\stX}\stY_{T}$. One can check that $c\delta_{\stY_{T}}a$
equals $\Psi\Phi(\delta)_{T}$ and, thanks to the above diagram, $\psi\Phi(\delta)_{T}=\delta_{T}$.

The isomorphism for the base change is tautological. For the last
claim we can assume that $\stX$ is an affine scheme in which case
the result follows because (usual) push-forwards commutes with arbitrary
base changes.
\end{proof}
In general $f_{p}$ does not preserve quasi-coherent sheaves, even
if $f$ is a proper map of schemes. To get a right adjoint of pullback
we have to require more. 
\begin{defn}
A pseudo-algebraic map $f\colon\stY\to\stX$ of fibered categories
is called \emph{flat} if given an object $\xi\colon T\to\stX$ of
$\stX$ and an atlas $V\to T\times_{\stX}\stY$ the resulting map
$V\to T$ is flat.
\end{defn}

If $f\colon\stY\to\stX$ is a map of algebraic stacks then the above
notion extends the classical one, which is made only using smooth
atlases. Indeed one reduces easily to the case of schemes, in which
case by hypothesis there is an fpqc covering $g\colon\stZ\to\stX$
such that $\stZ\to\stX$ is flat. It follows easily that $\stY\to\stX$
is also flat.
\begin{prop}
\label{prop:flat maps are exact} If $f\colon\stY\to\stX$ is a flat
map of pseudo-algebraic fibered categories then $f^{*}\colon\QCoh(\stX)\to\QCoh(\stY)$
is exact.
\end{prop}

\begin{proof}
By \ref{prop:right exact sequence for pseudo-algebraic} we have to
show that $f^{*}$ maps a monomorphism $\shF\to\shG$ in $\QCoh(\stX)$
to a monomorphism in $\QCoh(\stY)$. From \ref{prop:reducing to atlases for quasi-geometric categories},
if $\xi\colon\Spec B\to\stY$ is a representable and flat map we have
to show that $\xi^{*}f^{*}\shF\to\xi^{*}f^{*}\shG$ is injective.
So we can assume $\stY=\Spec B$. If $\pi\colon X\to\stX$ is an atlas
consider the Cartesian diagram      \[   \begin{tikzpicture}[xscale=2.0,yscale=-1.2]     \node (A0_0) at (0, 0) {$Y$};     \node (A0_1) at (1, 0) {$\Spec B$};     \node (A1_0) at (0, 1) {$X$};     \node (A1_1) at (1, 1) {$\stX$};     \path (A0_0) edge [->]node [auto] {$\scriptstyle{h}$} (A0_1);     \path (A1_0) edge [->]node [auto] {$\scriptstyle{\pi}$} (A1_1);     \path (A0_1) edge [->]node [auto] {$\scriptstyle{f}$} (A1_1);     \path (A0_0) edge [->]node [auto] {$\scriptstyle{g}$} (A1_0);   \end{tikzpicture}   \] By
hypothesis $Y$ is an algebraic space, $g\colon Y\to X$ is flat and
$h\colon Y\to\Spec B$ is an fpqc covering. Moreover $\pi^{*}\shF\to\pi^{*}\shG$
is injective by \ref{prop:reducing to atlases for quasi-geometric categories}.
Thus $h^{*}f^{*}\shF\to h^{*}f^{*}\shG$ is injective and therefore
$f^{*}\shF\to f^{*}\shG$ is injective as well.
\end{proof}
\begin{prop}
Let $f\colon\stY\to\stX$ be a map from a pseudo-algebraic stack to
a quasi-separated scheme and such that $f^{*}\colon\QCoh(\stX)\to\QCoh(\stY)$
is exact. Then $f$ is flat.
\end{prop}

\begin{proof}
We first reduce to the case when $\stX$ is affine. Let $U\subseteq\stX$
be an affine open subset. The inclusion map $i\colon U\to\stX$ is
quasi-compact since $\stX$ is quasi-separated and quasi-separated
because a monomorphism. In particular $i_{*}\colon\QCoh(U)\to\QCoh(\stX)$
is well defined. Consider the Cartesian diagram   \[   \begin{tikzpicture}[xscale=1.5,yscale=-1.2]     \node (A0_0) at (0, 0) {$Y$};     \node (A0_1) at (1, 0) {$\stY$};     \node (A1_0) at (0, 1) {$U$};     \node (A1_1) at (1, 1) {$\stX$};     \path (A0_0) edge [->]node [auto] {$\scriptstyle{j}$} (A0_1);     \path (A1_0) edge [->]node [auto] {$\scriptstyle{i}$} (A1_1);     \path (A0_1) edge [->]node [auto] {$\scriptstyle{f}$} (A1_1);     \path (A0_0) edge [->]node [auto] {$\scriptstyle{g}$} (A1_0);   \end{tikzpicture}   \] Given
an injective map $\shG\to\shF$ in $\QCoh(U)$ we have to show that
$g^{*}(\shG\to\shF)$ is still injective. Consider the exact sequence
$\shH_{*}\colon0\to\shK\to i_{*}\shG\to i_{*}\shF$. Since $i^{*}i_{*}\simeq\id$
it follows that 
\[
j^{*}f^{*}\shH_{*}\simeq g^{*}i^{*}\shH_{*}\simeq g^{*}(0\to0\to\shG\to\shF)
\]
Since $f^{*}$ is exact and $j$ is flat the above sequence is exact
as required.

So we can assume $\stX$ affine. if $V\to\stY$ is an fpqc atlas from
a scheme then $V\to\stX$ is flat: if $W\subseteq V$ is an open affine
subset then $\QCoh(\stX)\to\QCoh(V)$ is exact and therefore $W\to\stX$
is flat. Since for any $T\to\stX$ the map $V\times_{\stX}T\to\stY\times_{\stX}T$
is an fpqc atlas, by definition of flatness we can conclude that $\stY\to\stX$
is flat.
\end{proof}
\begin{prop}
\label{prop:properties of push-forward} Let $f\colon\stY\to\stX$
be a quasi-compact and quasi-separated map of pseudo-algebraic fibered
categories. Then the composition $\QCoh\stY\to\Mod\odi{\stX}\to\Mod(\odi{\stX})_{|\stX_{\text{fl}}}$
has values in $\QCoh\stX_{\text{fl}}$. The induced map $f_{*}\colon\QCoh\stY\to\QCoh\stX$
is a right adjoint of $f^{*}\colon\QCoh\stX\to\QCoh\stY$. If   \[   \begin{tikzpicture}[xscale=1.5,yscale=-1.2]     \node (A0_0) at (0, 0) {$\stY'$};     \node (A0_1) at (1, 0) {$\stY$};     \node (A1_0) at (0, 1) {$\stX'$};     \node (A1_1) at (1, 1) {$\stX$};     \path (A0_0) edge [->]node [auto] {$\scriptstyle{g'}$} (A0_1);     \path (A1_0) edge [->]node [auto] {$\scriptstyle{g}$} (A1_1);     \path (A0_1) edge [->]node [auto] {$\scriptstyle{f}$} (A1_1);     \path (A0_0) edge [->]node [auto] {$\scriptstyle{f'}$} (A1_0);   \end{tikzpicture}   \] 
is a $2$-cartesian diagram of fibered categories with $\stX'$ pseudo-algebraic
then $\stY'$ is pseudo-algebraic, $f'$ is quasi-compact and quasi-separated
and there is a natural transformation of functors 
\[
g^{*}f_{*}\to f'_{*}g'^{*}\colon\QCoh\stX\to\QCoh\stY'
\]
which is an isomorphism if $g$ is flat.
\end{prop}

\begin{proof}
Consider the $2$-Cartesian diagram in the statement. The diagonal
of $f'$ is quasi-compact because it is base change of the diagonal
of $f$. To see that $f_{p}(\shF)_{|\stX_{\text{fl}}}$ is quasi-coherent
for $\shF\in\QCoh\stY$, we can assume $\stX=\Spec B$ affine and
that $\stY$ is quasi-compact with quasi-compact diagonal. If $U=\Spec A\to\stY$
is a fpqc atlas, it follows that $R=U\times_{\stY}U$ is a quasi-compact
algebraic space. By covering $R$ by finitely many affine schemes
$\Spec A_{i}$ we can write $\shF(\stY)$ as kernel of a map $\shF(A)\to\oplus_{i}\shF(A_{i})$.
If we base change along a flat map $B\to B'$ it is now easy to see
that $\shF(\stY\times_{B}B')\simeq\shF(\stY)\otimes_{B}B'$, as required.

To define the natural transformation $\alpha\colon g^{*}f_{*}\to f'_{*}g'^{*}$
notice that there is a natural map $f_{*}\shF\to f_{p}\shF$ which
extends the identity on $\stX_{\text{fl}}$. Applying $g^{*}$ we
get $g^{*}f_{*}\shF\to g^{*}f_{p}\shF\simeq f'_{p}g'^{*}\shF$ and
then, restricting to $\stX'_{\text{fl}}$, a map $(g^{*}f_{*}\shF)_{|\stX'_{\text{fl}}}\to(f'_{*}g'^{*}\shF)_{|\stX'_{\text{fl}}}$.
Since both sides are in $\QCoh\stX'_{\text{fl}}$ this map uniquely
extends to a natural transformation $\alpha$ as required. Finally
assume that $g$ is flat and let $\xi\colon\Spec B\to\stX'\in\stX'_{\text{fl}}$.
If the composition $\Spec B\to\stX$ is in $\stX_{\text{fl}}$ then
one can easily check that $\alpha(\xi)$ is an isomorphism. Otherwise,
by definition of flatness, there exists an fpqc covering $\Spec B'\to\Spec B$
whose composition $\xi'\colon\Spec B'\to\stX'$ satisfies the previous
condition. Since $\alpha(\xi)\otimes_{B}B'\simeq\alpha(\xi')$ we
get the desired result.
\end{proof}
\begin{rem}
There are set-theoretic problems in considering global sections of
presheaves and therefore push-forwards, because $\Mod\odi{\stX}$
is in general not locally small. The common way to solve this problem
is to use Grothendieck universes. Take a universe $U$ and define
rings inside $U$, so that $\Aff/R$ is small (with respect to a bigger
universe). Fibered categories should then be required to be small
too. In this situation it is easy to show that $\Mod\odi{\stX}$ is
locally small and therefore global sections and push-forwards are
well defined. With this approach we have to be careful in considering
(big) rings defined starting from some $\shF\in\Mod\odi{\stX}$: for
instance $\Spec\odi{\stX}(\stX)$ is in general not an object of $\Aff/R$.

Notice that global sections and pushforwards of quasi-coherent sheaves
are always well defined for a pseudo-algebraic fibered category and
a pseudo-algebraic map respectively. The reason is that if $\shF\in\QCoh\stX$
and $p\colon X\to\stX$ is a fpqc atlas then $\shF(\stX)\to(p^{*}\shF)(X)$
is injective and thus $\shF(\stX)$ is a set.

In the rest of the paper we will not be concerned about those set-theoretic
problems.
\end{rem}

\begin{defn}
\label{rem: A module structures of odix modules} If $A$ is an $R$-algebra
and $\shF\in\Mod\odi{\stX}$ then an $A$-module structure on $\shF$
is an $R$-algebra homomorphism $A\to\End_{\stX}(\shF)$. This is
the same data of $A$-module structures on $\shF(\xi)$ commuting
with the $\Hl^{0}(\odi{\pi(\xi)})$-module structure on $\shF(\xi)$
for all $\xi\in\stX$ and such that, for all $\xi\to\eta$ in $\stX$,
the map $\shF(\eta)\to\shF(\xi)$ is $A$-linear. We define $\QCoh_{A}\stX$
as the category of quasi-coherent sheaves over $\stX$ with an $A$-module
structure. We also define $\stX_{A}$ as the fiber product $\Spec A\times_{R}\stX$.
\end{defn}

Notice that if $\stY\to\stX$ is a map of fibered categories and $\shF$
is a presheaf of $\odi{\stX}$-modules with an $A$-module structure
then $g^{*}\shF$ inherits an $A$-module structure. In particular
$g^{*}\colon\QCoh\stX\to\QCoh\stY$ extends to a functor $\QCoh_{A}\stX\to\QCoh_{A}\stY$.
\begin{prop}
Let $A$ be an $R$-algebra. Then the push-forward map $\QCoh\stX_{A}\to\QCoh\stX$
extends naturally to an equivalence $\QCoh\stX_{A}\to\QCoh_{A}\stX$.
\end{prop}

\begin{proof}
The result is very simple if $\stX$ is an affine scheme. In general,
if we set $g\colon\stX_{A}\to\stX$ for the projection and we consider
$\shG\in\QCoh\stX_{A}$, then $g_{*}\shG\in\QCoh\stX$ and it inherits
an $A$-module structure from the action of $A$ on $\shG$. Therefore
$g_{*}\shG\in\QCoh_{A}\stX$. If $h\colon\Spec B\to\stX$ is a map
consider the diagrams   \[   \begin{tikzpicture}[xscale=3.4,yscale=-1.0]     
\node (A0_0) at (0.3, 0) {$\Spec (B\otimes_R A)$};     
\node (A0_1) at (1, 0) {$\stX_A$};     \node (A0_2) at (2, 0) {$\QCoh \stX_A$};     \node (A0_3) at (3, 0) {$\QCoh \Spec(B\otimes_R A)$};     \node (A1_0) at (0.3, 1) {$\Spec B$};    
\node (A1_1) at (1, 1) {$\stX$};     \node (A1_2) at (2, 1) {$\QCoh_A \stX$};     \node (A1_3) at (3, 1) {$\QCoh_A \Spec B$};     \path (A0_0) edge [->]node [auto] {$\scriptstyle{h'}$} (A0_1);     \path (A0_1) edge [->]node [auto] {$\scriptstyle{g}$} (A1_1);     \path (A1_0) edge [->]node [auto] {$\scriptstyle{h}$} (A1_1);     \path (A0_3) edge [->]node [auto] {$\scriptstyle{g'_*}$} (A1_3);     \path (A0_2) edge [->]node [auto] {$\scriptstyle{g_*}$} (A1_2);     \path (A0_0) edge [->]node [auto] {$\scriptstyle{g'}$} (A1_0);     \path (A0_2) edge [->]node [auto] {$\scriptstyle{h'^*}$} (A0_3);     \path (A1_2) edge [->]node [auto] {$\scriptstyle{h^*}$} (A1_3);   \end{tikzpicture}   \]  The second diagram is $2$-commutative and the last vertical map
is an equivalence. Using those diagrams it is easy to define a quasi-inverse
$\QCoh_{A}\stX\to\QCoh\stX_{A}$ of $g_{*}$. %
\end{proof}
We will almost always regard quasi-coherent sheaves over $\stX_{A}$
as objects of $\QCoh_{A}\stX$.
\begin{rem}
\label{rem: remark on the tensor product over A} The category $\Mod\odi{\stX}$
is symmetric monoidal: if $\shF,\shG\in\Mod\odi{\stX}$ then the formula
\[
(\shF\otimes\shG)(\xi)=\shF(\xi)\otimes\shG(\xi)
\]
defines a presheaf of $\odi{\stX}$-module. Tensor products of quasi-coherent
sheaves are quasi-coherent.

If $\shF,\shG\in\QCoh_{A}\stX$ then $\shF\otimes_{\odi{\stX}}\shG$
does not correspond to the tensor product in $\QCoh\stX_{A}$, but
to the tensor product of their pushforward along $\stX_{A}\to\stX$.
The sheaf $\shF\otimes_{\odi{\stX}}\shG$ has two distinct $A$-module
structures. Under the equivalence $\QCoh\stX_{A}\to\QCoh_{A}\stX$
the tensor product of $\shF$ and $\shG$, that we will denote by
$\shF\otimes_{\odi{\stX_{A}}}\shG$, is given by 
\[
U\longmapsto\shF(U)\otimes_{\Hl^{0}(\odi U)}\shG(U)/\langle ax\otimes y-x\otimes ay\st x\in\shF(U),y\in\shG(U)\rangle
\]
\end{rem}

\begin{defn}
A locally free sheaf or vector bundle $\E$ (of rank $n$) over $\stX$
is a quasi-coherent sheaf such that $\E(\Spec B\to\stX)$ is a finitely
generated projective $B$-module (of rank $n$) for all maps $\Spec B\to\stX$.
We denote by $\Loc\stX$ the subcategory of $\QCoh\stX$ of locally
free sheaves.
\end{defn}

\section{Sheafification functors.}

In this section we define and describe particular functors that generalize
sheafification functors for affine schemes or projective schemes.
The idea is to interpret the category of modules or graded modules
respectively as a category of $R$-linear functors. More precisely:
\begin{defn}
Given a fibered category $\stX$ over a ring $R$, an $R$-algebra
$A$ and a subcategory $\shD$ of $\QCoh\stX$ we define $\L_{R}(\shD,A)$
as the category of contravariant $R$-linear functors $\Gamma\colon\shD\to\Mod A$
and natural transformations as arrows. We define a functor $\Omega^{*}\colon\QCoh_{A}\stX\to\L_{R}(\shD,A)$
by 
\[
\Omega_{-}^{\shF}=\Hom_{\stX}(-,\shF)\colon\shD\to\Mod A
\]
The functor $\Omega^{*}$ is called the \emph{Yoneda functor }associated
with $\shD$. A left adjoint of $\Omega^{*}$ is called a \emph{sheafification
}functor associated with $\shD$. If $\shF\in\QCoh_{A}\stX$ we will
call $\Omega^{\shF}$ the Yoneda functor associated with $\shF$.
\end{defn}

The analogy with the sheafification functor is described in Section
\ref{sec:Quasi-projective-schemes and more}.

Let us fix an $R$-algebra $A$ and a fibered category $\pi\colon\stX\to\Aff/R$.

\subsection{Sheafifying $R$-linear functors.}

In this section we want to explicitly describe sheafification functors
for small subcategories of $\QCoh\stX$. In particular we fix a \emph{small
}(and non empty) subcategory $\shC$ of $\QCoh\stX$.

In the construction of the sheafification functors we will make use
of the coend construction in the settings of categories enriched over
categories of modules over a ring. The general theory simplifies considerably
in this context and we will also apply such construction only in particular
cases. In the following remark we collect all the properties we will
need.
\begin{rem}
\label{rem: coends} Let $\stY$ be a fibered category over $R$,
$F\colon\shC\to\QCoh\stY$ be an $R$-linear functor and $\Gamma\in\L_{R}(\shC,A)$.
The coend of the $R$-linear functor $\Gamma_{-}\otimes_{R}F_{-}\colon\shC^{\op}\times\shC\to\QCoh_{A}\stY$,
denoted by
\[
\int^{\E\in\shC}\Gamma_{\E}\otimes_{R}F_{\E}\in\QCoh_{A}\stY
\]
is the cokernel of the map
\[
\bigoplus_{\E\to\overline{\E}}(\Gamma_{u}\otimes\id_{F_{\E}}-\id_{\Gamma_{\overline{\E}}}\otimes F_{u})\colon\bigoplus_{\E\to\overline{\E}}\Gamma_{\overline{\E}}\otimes_{R}F_{\E}\to\bigoplus_{\E\in\shC}\Gamma_{\E}\otimes_{R}F_{\E}
\]
Moreover it comes equipped with an $A$-linear natural isomorphism
\[
\Hom_{\QCoh_{A}\stY}(\int^{\E\in\shC}\Gamma_{\E}\otimes_{R}F_{\E},\shH)\to\Hom_{\L_{R}(\shC,A)}(\Gamma,\Hom_{\stY}(F_{-},\shH))\text{ for }\shH\in\QCoh_{A}\stY
\]
The proof of this is just observing that data of a map $\omega\colon\Gamma_{\E}\otimes_{R}\shF_{\E}\to\shH$
and a map $\tilde{\omega}\colon\Gamma_{\E}\to\Hom_{\stY}(\shF_{\E},\shH)$
are the same, and the condition that $\tilde{\omega}$ is a natural
transformation is exactly the condition that $\omega$ pass to the
quotient $\int^{\E\in\shC}\Gamma_{\E}\otimes_{R}F_{\E}$. Everything
can be written down by the following expression:
\[
\alpha(\int^{\E\in\shC}\Gamma_{\E}\otimes_{R}F_{\E}\to\shH)(x)=\omega\circ p_{\overline{\E}}(x\otimes-)\colon F_{\overline{\E}}\to\shH\text{ for }\overline{\E}\in\shC,x\in\Gamma_{\overline{\E}}
\]
where $p_{\E}\colon\Gamma_{\E}\otimes_{R}F_{\E}\to\int^{\E\in\shC}\Gamma_{\E}\otimes_{R}F_{\E}$
for $\E\in\shC$ are the structure morphisms. Its inverse is uniquely
determined by the expression 
\[
\alpha^{-1}(\Gamma\to\Hom_{\stY}(F_{-},\shH))\circ p_{\E}\colon\Gamma_{\E}\otimes_{R}\E\to\shH\comma x\otimes y\longmapsto v_{\E}(x)(y)\text{ for }\E\in\shC
\]
Natural transformations $F\to F'$ and $\Gamma\to\Gamma'$ yields
morphisms $\int^{\E\in\shC}\Gamma_{\E}\otimes_{R}F_{\E}\to\int^{\E\in\shC}\Gamma_{\E}\otimes_{R}F'_{\E}$
and $\int^{\E\in\shC}\Gamma_{\E}\otimes_{R}F_{\E}\to\int^{\E\in\shC}\Gamma'_{\E}\otimes_{R}F_{\E}$
respectively. Those can be defined either using Yoneda's lemma and
the above characterization of $\Hom(\int^{\E\in\shC}\Gamma_{\E}\otimes_{R}F_{\E},-)$
or directly using the description of $\int^{\E\in\shC}\Gamma_{\E}\otimes_{R}F_{\E}$
as a cokernel.

All the above claims are standard in the theory of coend in the enriched
settings (in our case enriched by $\Mod R$), but, in this simplified
context, it is elementary to prove them directly.
\end{rem}

We start by showing that $\shC$ (and therefore any essentially small
subcategory of $\QCoh\stX$) admits a sheafification functor.
\begin{prop}
\label{prop:adjoint of Omegastar with the cokernel} The Yoneda functor
$\Omega^{*}\colon\QCoh_{A}\stX\to\L_{R}(\shC,A)$ has a left adjoint
$\shF_{-,\shC}\colon\L_{R}(\shC,A)\to\QCoh_{A}\stX$ given by

\[
\shF_{\Gamma,\shC}=\int^{\E\in\shC}\Gamma_{\E}\otimes_{R}\E\in\QCoh_{A}\stX
\]
where $\E$ denotes the inclusion $\shC\to\QCoh\stX$. Given $\xi\in\stX$
we have that 
\[
\shF_{\Gamma,\shC}(\xi)=\int^{\E\in\shC}\Gamma_{\E}\otimes_{R}\E(\xi)\in\Mod(\Hl^{0}(\odi{\pi(\xi)})\otimes_{R}A)
\]
where $\E(\xi)$ denotes the evaluation $\shC\to\Mod\Hl^{0}(\odi{\pi(\xi)})$
of sheaves in $\xi\in\stX$.
\end{prop}

\begin{proof}
It is enough to apply \ref{rem: coends} with $\stY=\stX$ and $F\colon\shC\to\QCoh\stX$
the inclusion. Using the description of coend as cokernel one can
check that the two functors defined in the statement are canonically
isomorphic.
\end{proof}
\begin{rem}
\label{rem:explicit adjunction}Given $\shH\in\QCoh_{A}\stX$ and
$\Gamma\in\L_{R}(\shC,A)$ the adjunction is
\[
\Hom_{\stX_{A}}(\shF_{\Gamma,\shC},\shH)\simeq\Hom_{\L_{R}(\shC,A)}(\Gamma,\Hom_{\stX}(-,\shH))
\]
Explicitly this map is just
\[
(\Gamma_{\E}\otimes\E\to\shH)\longmapsto(\Gamma_{\E}\to\Hom_{\stX}(\E,\shH))\text{ for }\E\in\shC
\]
where we think of $\shF_{\Gamma,\shC}$ as a quotient of $\bigoplus_{\E}\Gamma_{\E}\otimes\E$
as in \ref{rem: coends}.
\end{rem}

\begin{defn}
\label{def: unit and counit for adjunction} We denote by $\gamma_{\Gamma}\colon\Gamma_{-}\to\Omega_{-}^{\shF_{\Gamma,\shC}}=\Hom_{\stX}(-,\shF_{\Gamma,\shC})$
and $\delta_{\shG}\colon\shF_{\Omega^{\shG},\shC}\to\shG$ for $\Gamma\in\L_{R}(\shC,A)$
and $\shG\in\QCoh_{A}\stX$ the unit and the counit of the adjunction
between $\Omega^{*}\colon\QCoh_{A}\stX\to\L_{R}(\shC,A)$ and $\shF_{*,\shC}\colon\L_{R}(\shC,A)\to\QCoh_{A}\stX$
respectively.

Given $\xi\in\stX$, $\E\in\shC$, $\psi\in\E(\xi)$ and $x\in\Gamma_{\E}$
we denote by $x_{\E,\psi}\in\shF_{\Gamma,\shC}(\xi)$ the image of
$x\otimes\psi$ under the map $\Gamma_{\E}\otimes_{R}\E(\xi)\to\shF_{\Gamma,\shC}(\xi)$
\end{defn}

\begin{prop}
\label{prop:FGammaC on affine schemes} Let $\Gamma\in\L_{R}(\shC,A)$.
The unit $\gamma_{\Gamma}\colon\Gamma_{-}\to\Omega_{-}^{\shF_{\Gamma,\shC}}=\Hom_{\stX}(-,\shF_{\Gamma,\shC})$
is given by   \[   \begin{tikzpicture}[xscale=3.8,yscale=-0.6]     
\node (A0_0) at (0.2, 0) {$\Gamma_\E$};   
\node (A1_0) at (0.2, 1) {$x$}; 
\node (A0_1) at (1, 0) {$\Hom_{\stX}(\E,\shF_{\Gamma,\shC})$};     
\node (A1_1) at (1, 1) {$(\phi \longmapsto x_{\E,\phi})$};     
\path (A0_0) edge [->]node [auto] {$\scriptstyle{}$} (A0_1);     
\path (A1_0) edge [|->,gray]node [auto] {$\scriptstyle{}$} (A1_1);        
\end{tikzpicture}   \]  If $\shG\in\QCoh_{A}\stX$ the counit $\delta_{\shG}\colon\shF_{\Omega^{\shG},\shC}\to\shG$
is given by 
\[
\shF_{\Omega^{\shG},\shC}(\xi)\ni x_{\E,\psi}\longmapsto x(\psi)\in\shG(\xi)\text{ for }\E\in\shC,x\in\Omega_{\E}^{\shG}=\Hom_{\stX}(\E,\shG),\xi\in\stX,\psi\in\E(\xi)
\]
\end{prop}

\begin{proof}
This follows easily from the description in \ref{rem:explicit adjunction}.
\end{proof}
Given a map $g\colon\stY\to\stX$ of fibered categories we want to
express $g^{*}\shF_{\Gamma,\shC}\in\QCoh_{A}\stY$ for $\Gamma\in\L_{R}(\shC,A)$
as $\shF_{g^{*}\Gamma,g^{*}\shC}$ for a suitable choice of $g^{*}\shC\subseteq\QCoh\stY$
and $g^{*}\Gamma\in\L_{R}(g^{*}\shC,A)$.
\begin{defn}
Let $\stY$ be a fibered category, $g\colon\stY\to\stX$ be a morphism
and $\shD$ be a subcategory of $\QCoh\stX$. We set $g^{*}\shD$
for the subcategory of $\QCoh\stY$ of sheaves $g^{*}\E$ for $\E\in\shD$.
If $\shD'\subseteq\QCoh\stY$ is a subcategory containing $g^{*}\shD$
we can define a restriction functor   \[   \begin{tikzpicture}[xscale=2.7,yscale=-0.6]     \node (A0_0) at (0, 0) {$\L_R(\shD',A)$};     \node (A0_1) at (1, 0) {$\L_R(\shD,A)$};     \node (A1_0) at (0, 1) {$\Gamma$};     \node (A1_1) at (1, 1) {$\Gamma\circ g^*$};     \path (A0_0) edge [->]node [auto] {$\scriptstyle{g_*}$} (A0_1);     \path (A1_0) edge [|->,gray]node [auto] {$\scriptstyle{}$} (A1_1);   \end{tikzpicture}   \] 
\end{defn}

\begin{prop}
\label{thm:adjoint of push forward} Let $\stY$ be a fibered category,
$g\colon\stY\to\stX$ be a morphism and $\shD$ be a subcategory of
$\QCoh\stY$ such that $g^{*}\shC\subseteq\shD$. Then $g_{*}\colon\L_{R}(\shD,A)\to\L_{R}(\shC,A)$
has a left adjoint $g^{*}\colon\L_{R}(\shC,A)\to\L_{R}(\shD,A)$ and
it is given by
\[
(g^{*}\Gamma)_{\shG}=\int^{\E\in\shC}\Gamma_{\E}\otimes_{R}\Hom_{\stY}(\shG,g^{*}\E)\in\Mod A\text{ for }\Gamma\in\L_{R}(\shC,A)\comma\shG\in\shD
\]
where $\Hom_{\stY}(\shG,g^{*}-)$ is thought of as a functor $\shC\to\Mod R$.
If $\stY=\stX$ and $g=\id_{\stX}$, so that $\shC\subseteq\shD$
and $(\id_{\stX})_{*}\colon\L_{R}(\shD,A)\to\L_{R}(\shC,A)$ is the
restriction, then the unit $\Gamma\to(\id_{\stX}^{*}\Gamma)_{|\shC}$
is an isomorphism for $\Gamma\in\L_{R}(\shC,A)$.
\end{prop}

\begin{proof}
Let $\Omega\in\L_{R}(\shD,A)$, $\Gamma\in\L_{R}(\shC,A)$ and $\shG\in\shD$.
Applying \ref{rem: coends} with $F=\Hom_{\stY}(\shG,g^{*}-)\colon\shC\to\Mod(R)$
and $\Gamma\colon\shC\to\Mod(A)$ we get a bijection between $A$-linear
maps $(g^{*}\Gamma)_{\shG}\to\Omega_{\shG}$ and the set of $A$-linear
natural transformations $\Gamma_{\E}\to\Hom_{R}(\Hom(\shG,g^{*}\E),\Omega_{\shG})$
for $\E\in\shC$. Thinking of $(g^{*}\Gamma)_{\shG}$ as a quotient
of $\bigoplus_{\E\in\shC}\Gamma_{\E}\otimes\Hom(\shG,g^{*}\E)$ the
previous map is just the canonical map
\[
\Gamma_{\E}\to\Hom_{R}(\Hom(\shG,g^{*}\E),\Omega_{\shG})\longmapsto\Gamma_{\E}\otimes\Hom(\shG,g^{*}\E)\to\Omega_{\shG}
\]
With this description in mind it is elementary to check that a natural
transformation $g^{*}\Gamma\to\Omega$ corresponds to a collection
$\underline{\gamma}$ of $A$-linear maps $\gamma_{\shG,\E}\colon\Gamma_{\E}\to\Hom_{R}(\Hom(\shG,g^{*}\E),\Omega_{\shG})$
such that
\[
\gamma_{\shG,\E}(x)(\phi\circ u)=\Omega_{u}(\gamma_{\overline{\shG},\E}(x)(\phi))\text{ for }x\in\Gamma_{\E},\phi\colon\overline{\shG}\to g^{*}\E,u\colon\shG\to\overline{\shG}\in\shD
\]
\[
\gamma_{\shG,\E'}(\Gamma_{v}(x))(\phi')=\gamma_{\shG,\E}(x)(g^{*}v\circ\phi')\text{ for }x\in\Gamma_{\E},\phi'\colon\shG\to g^{*}\E',v\colon\E'\to\E
\]
The first is the functoriality in $\shG$, the second in $\E$. Given
any collection of maps $\gamma_{\shG,\E}\colon\Gamma_{\E}\to\Hom_{R}(\Hom(\shG,g^{*}\E),\Omega_{\shG})$
define
\[
\mu(\gamma)_{\E}\colon\Gamma_{\E}\to\Omega_{g^{*}\E}\comma\mu(\gamma)_{\E}(x)=\gamma_{g^{*}\E,\E}(x)(\id_{g^{*}\E})
\]
Conversely, given any collection of maps $\mu_{\E}\colon\Gamma_{\E}\to\Omega_{g^{*}\E}$
define
\[
\gamma(\mu)_{\shG,\E}\colon\Gamma_{\E}\to\Hom_{R}(\Hom(\shG,g^{*}\E),\Omega_{\shG})\comma\gamma(\mu)_{\shG,\E}(x)(\psi)=\Omega_{\psi}(\mu(x))
\]
Now come the tedious verification that $\gamma_{\shG,\E}$ yields
a natural transformation $g^{*}\Gamma\to\Omega$ if and only if $\mu_{\E}$
define a natural transformation $\Gamma\to g_{*}\Omega$ and that
this defines a bijection $\Hom(g^{*}\Gamma,\Omega)\simeq\Hom(\Gamma,g_{*}\Omega)$.
This is elementary and can be carried on just using symbols and expressions
above.

Assume now $\stY=\stX$ and $g=\id_{\stX}$ and let $\Gamma\in\L_{R}(\shC,A)$
and $\overline{\E}\in\shC$. Denote by $\alpha\colon\Gamma\to(\id_{\stX}^{*}\Gamma)_{|\shC}$
the unit morphism. If $p_{\widetilde{\E}}\colon\Gamma_{\widetilde{\E}}\otimes\Hom_{\stX}(\overline{\E},\widetilde{\E})\to(\id_{\stX}^{*}\Gamma)_{\overline{\E}}$
are the structure morphisms as in \ref{rem: coends}, then
\[
\alpha_{\overline{\E}}\colon\Gamma_{\overline{\E}}\to(\id_{\stX}^{*}\Gamma)_{\overline{\E}}\comma\alpha_{\overline{\E}}(x)=p_{\overline{\E}}(x\otimes\id_{\overline{\E}})
\]
In particular, given $\shH\in\Mod A$ and using \ref{rem: coends},
the map $\Hom_{A}(\alpha_{\overline{\E}},\shH)\colon\Hom_{A}((\id_{\stX}^{*}\Gamma)_{\overline{\E}},\shH)\to\Hom_{A}(\Gamma_{\overline{\E}},\shH)$
sends an $A$-linear natural transformation $\delta\colon\Gamma_{-}\to\Hom_{R}(\Hom_{\stX}(\overline{\E},-),\shH)$
to $\Gamma_{\overline{\E}}\ni x\longmapsto\delta(x)(\id_{\overline{\E}})\in\shH$.
Since $\delta$ corresponds to an $R$-linear natural transformation
$\Hom_{\stX}(\overline{\E},-)\to\Hom_{A}(\Gamma_{-},\shH)$, we can
rewrite $\Hom_{A}(\alpha_{\overline{\E}},\shH)$ as
\[
\Hom_{\L_{R}(\shC,R)}(\Hom(\overline{\E},-),\Delta)\to\Delta_{\overline{\E}}\comma\omega\mapsto\omega(\id_{\overline{\E}})\comma\Delta=\Hom_{A}(\Gamma_{*},\shH)
\]
By enriched Yoneda lemma we can conclude that $\Hom_{A}(\alpha_{\overline{\E}},\shH)$
is an isomorphism and, therefore, that $\alpha_{\overline{\E}}$ is
an isomorphism as required.
\end{proof}
\begin{rem}
Proposition above can also be reinterpreted using Kan extension. Indeed
given an $R$-linear functor $\Gamma\colon\shC\to(\Mod A)^{\op}$
(notice the opposite category) and $g^{*}\colon\shC\to\shD$ then
$g^{*}\Gamma\colon\shC\to(\Mod A)^{\op}$ is the right Kan extension
of $\Gamma$ along $g^{*}$. This follows from the definition of the
Kan extension and the adjointness we proved. In particular one could
have proved proposition above using standard results about cocomplete
categories. On the other hand we need the coend description.
\end{rem}

The above proposition yields a natural extension of any $\Gamma\in\L_{R}(\shC,A)$
to a functor $\Gamma^{ex}\in\L_{R}(\QCoh\stX,A)$. By abuse of notation
we will denote them by the same symbol $\Gamma$. This means that
if $\Gamma\in L_{R}(\shC,A)$ and $\shG\in\QCoh\stX$ then we can
evaluate $\Gamma$ on $\shG$, writing $\Gamma_{\shG}$.

Given a map $g\colon\stY\to\stX$ we will denote by $g^{*}\colon\L_{R}(\shC,A)\to\L_{R}(g^{*}\shC,A)$
the left adjoint of the restriction $\L_{R}(g^{*}\shC,A)\to\L_{R}(\shC,A)$.
So, given $\Gamma\in\L_{R}(\shC,A)$, $g^{*}\Gamma$ is a functor
$g^{*}\shC\to\Mod A$ but it also defines a functor $\QCoh\stY\to\Mod A$
denoted, by our convention, by the same symbol. By \ref{thm:adjoint of push forward}
the functor $g^{*}\Gamma\colon\QCoh\stY\to\Mod A$ coincides with
the value of the left adjoint of the restriction $\L_{R}(\QCoh\stY,A)\to\L_{R}(\shC,A)$.
\begin{rem}
\label{rem: extends to h0OX}Given $\Gamma\in\L_{R}(\shC,A)$ and
$\E\in\shC$ we have $R$-linear morphisms of rings 
\[
\Hl^{0}(\odi{\stX})\simeq\End_{\stX}(\odi{\stX})\to\End_{\stX}(\E)\to\End_{A}(\Gamma_{\E})
\]
This defines a lifting of $\Gamma$ to an $R$-linear functor $\Gamma\colon\shC\to\Mod(\Hl^{0}(\odi{\stX})\otimes_{R}A)$
and an equivalence
\[
\L_{R}(\shC,A)\to\L_{R}(\shC,\Hl^{0}(\odi{\stX})\otimes_{R}A)
\]
In particular, if $g\colon\Spec B\to\stX$ is a map and $\Gamma\in\L_{R}(\shC,A)$
then $(g^{*}\Gamma)_{B}$ has a $B\otimes_{R}A$-module structure.
By \ref{prop:adjoint of Omegastar with the cokernel} and \ref{thm:adjoint of push forward}
there is a canonical $A$-linear isomorphism 
\[
\shF_{\Gamma,\shC}(B)\simeq(g^{*}\Gamma)_{B}
\]
and it is easy to see that it is also $B$-linear.

\end{rem}

\begin{prop}
\label{prop:pullbacks and compatibilities} Let $g\colon\stY\to\stX$
be a morphism of fibered categories and $\shD\subseteq\QCoh(\stY)$
such that $g^{*}\shC\subseteq\shD$. If $g^{*}\colon\QCoh(\stX)\to\QCoh(\stY)$
has a right adjoint $g_{*}\colon\QCoh(\stY)\to\QCoh(\stX)$ then there
is a canonical isomorphism $g_{*}(\Omega_{|\shD}^{\shN})\simeq\Omega_{|\shC}^{g_{*}\shN}$
for $\shN\in\QCoh_{A}\stY,$that is a $2$-commutative diagram   \[   \begin{tikzpicture}[xscale=3.0,yscale=-1.2]     \node (A0_0) at (0, 0) {$\QCoh_A(\stY)$};     \node (A0_1) at (1, 0) {$\L_R(\shD,A)$};     \node (A1_0) at (0, 1) {$\QCoh_A(\stX)$};     \node (A1_1) at (1, 1) {$\L_R(\shC,A)$};     \path (A0_0) edge [->]node [auto] {$\scriptstyle{\Omega^*}$} (A0_1);     \path (A0_0) edge [->]node [auto] {$\scriptstyle{g_*}$} (A1_0);     \path (A0_1) edge [->]node [auto] {$\scriptstyle{g_*}$} (A1_1);     \path (A1_0) edge [->]node [auto] {$\scriptstyle{\Omega^*}$} (A1_1);   \end{tikzpicture}   \] Moreover,
without any assumption of $g$, but assuming $\shD$ small, there
exists an isomorphism $g^{*}\shF_{\Gamma,\shC}\simeq\shF_{g^{*}\Gamma,g^{*}\shC}$
natural in $\Gamma\in\L_{R}(\shC,A)$, that is a $2$-commutative
diagram   \[   \begin{tikzpicture}[xscale=3.1,yscale=-1.2]     \node (A0_0) at (0, 0) {$\L_R(\shC,A)$};     \node (A0_1) at (1, 0) {$\QCoh_A \stX$};     \node (A1_0) at (0, 1) {$\L_R(\shD,A)$};     \node (A1_1) at (1, 1) {$\QCoh_A \stY$};     \path (A0_0) edge [->]node [auto] {$\scriptstyle{\shF_{-,\shC}}$} (A0_1);     \path (A0_0) edge [->]node [auto] {$\scriptstyle{g^*}$} (A1_0);     \path (A0_1) edge [->]node [auto] {$\scriptstyle{g^*}$} (A1_1);     \path (A1_0) edge [->]node [auto] {$\scriptstyle{\shF_{-,\shD}}$} (A1_1);   \end{tikzpicture}   \] Concretely
the isomorphism is of the form   \[   \begin{tikzpicture}[xscale=3.0,yscale=-1.2]     \node (A0_0) at (0, 0) {$g^*(\Gamma_\E \otimes \E)$};     \node (A0_1) at (1, 0) {$(g*\Gamma)_{g^*\E}\otimes g^*\E$};     \node (A1_0) at (0, 1) {$g^*\shF_{\Gamma,\shC}$};     \node (A1_1) at (1, 1) {$\shF_{g^*\Gamma,\shD}$};     \path (A0_0) edge [->]node [auto] {$\scriptstyle{}$} (A0_1);     \path (A1_0) edge [->]node [auto] {$\scriptstyle{}$} (A1_1);     \path (A0_1) edge [->]node [auto] {$\scriptstyle{}$} (A1_1);     \path (A0_0) edge [->]node [auto] {$\scriptstyle{}$} (A1_0);   \end{tikzpicture}   \] where
$\Gamma_{\E}\to(g^{*}\Gamma)_{g^{*}\E}=(g_{*}g^{*}\Gamma)_{\E}$ is
the adjunction map. In particular the isomorphism $g^{*}\shF_{\Gamma,\shC}\to\shF_{g^{*}\Gamma,g^{*}\shC}$
is compatible with composition of maps $g\colon\stY\to\stX$.
\end{prop}

\begin{proof}
The first isomorphism is clear because 
\[
\Omega_{\E}^{g_{*}\shN}=\Hom_{\stX}(\E,g_{*}\shN)\simeq\Hom_{\stY}(g^{*}\E,\shN)=\Omega_{g^{*}\E}^{\shN}=(g_{*}\Omega^{\shN})_{\E}
\]
In particular, assuming the existence of $g_{*}$, the second part
of the statement is a consequence of the first by adjointness. We
have to prove the second part without this.

Let $\Gamma\in\L_{R}(\shC,A)$, $\xi\colon\Spec B\to\stY$ be a map
and $N\in\Mod B\otimes_{R}A$. Denote by $F\colon\shC\to\Mod B$ and
$G\colon\shD\to\Mod B$ the functors obtained evaluating the sheaves
in $g\xi$ and $\xi$ respectively. In particular $F=G\circ g^{*}=g_{*}G$.
We have
\[
(g^{*}\shF_{\Gamma,\shC})(B)=\shF_{\Gamma,\shC}(g\xi)=\int^{\E\in\shC}\Gamma_{\E}\otimes F_{\E}\text{ and }\shF_{g^{*}\Gamma,\shD}(B)=\int^{\shH\in\shD}(g^{*}\Gamma)_{\shH}\otimes G_{\shH}
\]
They implies respectively that 
\[
\Hom_{B\otimes_{R}A}((g^{*}\shF_{\Gamma,\shC})(B),N)\simeq\Hom_{\L_{R}(\shC,A)}(\Gamma,\Hom_{B}(F,N))
\]
\[
\Hom_{B\otimes_{R}A}(\shF_{g^{*}\Gamma,\shD}(B),N)\simeq\Hom_{\L_{R}(\shD,A)}(g^{*}\Gamma,\Hom_{B}(G,N))
\]
Since $\Hom_{B}(F,N)=g_{*}\Hom_{B}(G,N)$ the above modules are canonically
isomorphic. In particular we get an isomorphism $(g^{*}\shF_{\Gamma,\shC})(B)\simeq\shF_{g^{*}\Gamma,\shD}(B)$.
By a direct check we see that this map fits in the last commutative
diagram in the statement evaluated in $\xi\colon\Spec B\to\stY$.
This shows that $g^{*}\shF_{\Gamma,\shC}\to\shF_{g^{*}\Gamma,\shD}$
is well defined and an isomorphism. Naturality in $\Gamma\in\L_{R}(\shC,A)$
also follows easily.
\end{proof}
\begin{rem}
\label{rem:sheafification of an extension} For $g=\id_{\stX}\colon\stX\to\stX$
and $\shC\subseteq\shD$ proposition above is telling us that if $\Gamma\in\L_{R}(\shC,A)$
and we extend it to $\Gamma^{ex}\in\L_{R}(\shD,A)$ then $\shF_{\Gamma,\shC}\simeq\shF_{\Gamma^{ex},\shD}$.
\end{rem}

\begin{rem}
\label{rem: LR and QCoh as stacks} If $A\to A'$ is a morphism of
$R$-algebras then we have pull-back functors $\L_{R}(\shC,A)\to\L_{R}(\shC,A')$
and $\QCoh_{A}\stX\to\QCoh_{A'}\stX$. The first one is obtained considering
the tensor product $-\otimes_{A}A'$, while the second one corresponds
to the pullback $\QCoh\stX_{A}\to\QCoh\stX_{A'}$ along the projection
$\stX_{A'}\to\stX_{A}$. Alternatively, those functors are left adjoints
to the restriction of scalars $\L_{R}(\shC,A')\to\L_{R}(\shC,A)$
and $\QCoh_{A'}\stX\to\QCoh_{A}\stX$ respectively. It is easy to
see that in this way we obtain two fpqc stacks (not in groupoids)
$\L_{R}(\shC,-)$ and $\QCoh_{-}\stX$ over the category of affine
$R$-schemes. Notice that the functor $\Omega^{*}\colon\QCoh_{-}\stX\to\L_{R}(\shC,-)$
is not a morphism of stacks because $\Hom_{\stX}(\E,\shG)\otimes_{A}A'\not\simeq\Hom_{\stX}(\E,\shG\otimes_{A}A')$
in general for $\E\in\shC$ and $\shG\in\QCoh_{A}\stX$.
\end{rem}

\begin{prop}
\label{prop:base change for FstarC} The functor $\shF_{*,\shC}\colon\L_{R}(\shC,-)\to\QCoh_{-}\stX$
is a morphism of stacks.
\end{prop}

\begin{proof}
Given a morphism $A\to A'$ of $R$-algebras we have a $2$-commutative
diagram   \[   \begin{tikzpicture}[xscale=2.9,yscale=-1.0]     \node (A0_0) at (0, 0) {$\QCoh_{A'}\stX$};     \node (A0_1) at (1, 0) {$\L_R(\shC,A')$};     \node (A1_0) at (0, 1) {$\QCoh_A \stX$};     \node (A1_1) at (1, 1) {$\L_R(\shC,A)$};     \path (A0_0) edge [->]node [auto] {$\scriptstyle{\Omega^*}$} (A0_1);     \path (A0_0) edge [->]node [auto] {$\scriptstyle{}$} (A1_0);     \path (A0_1) edge [->]node [auto] {$\scriptstyle{}$} (A1_1);     \path (A1_0) edge [->]node [auto] {$\scriptstyle{\Omega^*}$} (A1_1);   \end{tikzpicture}   \] 
where the vertical arrows are obtained by restricting the scalars
from $A'$ to $A$. Using \ref{rem: LR and QCoh as stacks} and taking
the left adjoint functors of the functors in the diagram we exactly
get the $2$-commutative diagram expressing the fact that $\shF_{*,\shC}$
preserves Cartesian arrows.
\end{proof}
We conclude this section by showing that, when considering sheafification
functors $\shF_{-,\shC}$, we can always reduce problems to the case
when $\shC$ is an additive category. Moreover in this case the sections
of $\shF_{-,\shC}$ have a nice expression in terms of a direct limit.
\begin{defn}
Given a subcategory $\shD$ of $\QCoh\stX$ we denote by $\shD^{\oplus}$
the subcategory of $\QCoh\stX$ whose objects are all finite direct
sums of sheaves in $\shD$.
\end{defn}

Notice that if $\shD$ is small then $\shD^{\oplus}$ is small.
\begin{prop}
\label{prop:find a directed category} Let $\shD\subseteq\QCoh(\stX)$
be a subcategory and $\alB$ be an additive $R$-linear category.
Then $\shD\subseteq\shD^{\oplus}$ induces an equivalence between
the category of (contravariant) $R$-linear functors $\shD^{\oplus}\to\alB$
and the category of (contravariant) $R$-linear functors $\shD\to\alB$.

In particular the restriction $\L_{R}(\shD^{\oplus},A)\to\L_{R}(\shD,A)$
is an equivalence. If $\shD$ is small and $\Gamma\in\L_{R}(\shD^{\oplus},A)$
then we have a canonical isomorphism $\shF_{\Gamma,\shC^{\oplus}}\simeq\shF_{\Gamma_{|\shC},\shC}$.
\end{prop}

\begin{proof}
The contravariant case follows from the covariant one replacing $\alB$
by $\alB^{\op}$. Given $\E\in\shD^{\oplus}$ let us fix a finite
set $I_{\E}\subseteq\text{Obj}(\shD)$ with an isomorphism $\E\simeq\bigoplus_{V\in I_{\E}}V$.
If $\E\in\shD$ we choose $I_{\E}=\{\E\}$. Given an $R$-linear functor
$\Gamma\colon\shD\to\alB$ we define an extension $\Delta^{\Gamma}\colon\shD^{\oplus}\to\alB$
as follows. On objects we set
\[
\Delta_{\E}^{\Gamma}=\bigoplus_{V\in I_{\E}}\Gamma_{V}\in\alB
\]
A morphism $\phi\colon\E\to\E'$ in $\shD^{\oplus}$ is completely
determined by a matrix $(\phi_{V,W})_{V\in I_{\E},W\in I_{\E'}}$
with entries $\phi_{V,W}\in\Hom_{\shD}(V,W)$. We define $\Delta_{\E}^{\Gamma}\to\Delta_{\E'}^{\Gamma}$
as the map corresponding to the matrix $(\Gamma_{\phi_{V,W}})_{V\in I_{\E},W\in I_{\E'}}$.
Using the linearity of $\Gamma$ we can deduce that $\Delta^{\Gamma}\colon\shD^{\oplus}\to\alB$
is a well defined $R$-linear functor extending $\Gamma$. Moreover
it is easy to see that this is a quasi-inverse to the restriction
of functors. Last claim follows from \ref{rem:sheafification of an extension}.
\end{proof}
\begin{rem}
\label{rem: natural transf are linear for additive categories} If
$\catC$ is an $R$-linear and additive category and $F,G\colon\catC\to\Mod A$
are $R$-linear (covariant or contravariant) functors then any natural
transformation $\lambda\colon F\to G$ of functors of sets is $R$-linear.
Indeed by considering $\catC^{\op}$ we can consider only covariant
functors. In this case it is easy to show that the maps $\lambda_{X}\colon F(X)\to G(X)$
for $X\in\catC$ are $R$-linear using functoriality on the map $r\id_{X}\colon X\to X$
for $r\in R$ and $\pr_{1},\pr_{2},\pr_{1}+\pr_{2}\colon X\oplus X\to X$,
where $\pr_{*}$ are the projections.
\end{rem}

\begin{rem}
\label{rem:non filtered colimit} If $I$ is a small category and
$F\colon I\to\QCoh\stY$ is a functor (covariant or contravariant),
for some category fibered in groupoids $\stY$ over $R$, then the
colimit $\varinjlim_{i}F(i)$ always exists in $\QCoh(\stY)$. Indeed,
replacing $I$ by $I^{\op}$, we can assume $F$ contravariant. In
that case the colimit is the cokernel of the map
\[
\bigoplus_{i\to j}F(j)\to\bigoplus_{k}F(k)\comma(\alpha\colon i\to j,x\in F(j))\longmapsto F(\alpha)(x)-x\in F(i)\oplus F(j))
\]
Moreover for later reference we observe the following. If for all
objects $i,j\in I$ there exists $k\in I$ and maps $k\to i,k\to j$
then all elements of $\varinjlim_{i}F(i)$ comes from some $F(q)$.
\end{rem}

\begin{defn}
Let $\Spec B\to\stX$ be a map. We denote by $J_{B,\shC}$ the category
of pairs $(\E,\psi)$ where $\E\in\shC^{\oplus}$ and $\psi\in\E(B)$.
Given $\Gamma\in\L_{R}(\shC,A)$ we have a functor $\Gamma\colon J_{B,\shC}\to\Mod A$
given by $\Gamma_{\E,\psi}=\Gamma_{\E}$.
\end{defn}

We will make colimits over the category $J_{B,\shC}$, but we warn
the reader that this is not filtered in general.
\begin{prop}
\label{prop:FGammaCB as a direct limit} Let $\Spec B\to\stX$ be
a map and $\Gamma\in\L_{R}(\shC,A)$. The category $J_{B,\shC}$ is
non-empty and for all $\xi,\xi'\in J_{B,\shC}$ there exists $\xi''\in J_{B,\shC}$
and maps $\xi''\to\xi$, $\xi''\to\xi'$. The $A$-linear maps $\Gamma_{\E,\psi}=\Gamma_{\E}\to\shF_{\Gamma,\shC^{\oplus}}(B)\simeq\shF_{\Gamma,\shC}(B)$,
$x\longmapsto x_{\E,\psi}$ for $(\E,\psi)\in J_{B,\shC}$ (see \ref{def: unit and counit for adjunction})
induce an $A$-linear isomorphism 
\[
\varinjlim_{(\E,\psi)\in J_{B,\shC}^{\op}}\Gamma_{\E,\psi}\to\shF_{\Gamma,\shC}(B)
\]
The multiplication by $b\in B$ on the first limit is induced by mapping
$\Gamma_{\E,\psi}$ to $\Gamma_{\E,b\psi}$ using $\id_{\Gamma_{\E}}$
for $(\E,\psi)\in J_{B,\shC}$. In other words 
\[
x_{\E,b\psi}=bx_{\E,\psi}\text{ for }\E\in\shC^{\oplus},x\in\Gamma_{\E},\psi\in\E(B),b\in B
\]
Finally every element of $\shF_{\Gamma,\shC}(B)$ is of the form $x_{\E,\psi}$
for some $(\E,\psi)\in J_{B,\shC}$.
\end{prop}

\begin{proof}
By \ref{prop:find a directed category} we can assume $\shC=\shC^{\oplus}$.
Denote by $\shH$ and $\alpha\colon\shH\to\shF_{\Gamma,\shC}(B)$
the limit and the map in the statement respectively. The category
$J_{B,\shC}$ is not empty because $(\E,0)\in J_{B,\shC}$ for all
$\E\in\shC$ and the map $\alpha$ is well defined because, for all
$x\in\Gamma_{\overline{\E}}$ and for all $u\colon(\E,\psi)\to(\overline{\E},u(\psi))$
we have $x_{\overline{\E},u(\psi)}=(\Gamma_{u}(x))_{\E,\psi}$, by
definition of $\shF_{\Gamma,\shC}(B)$ as coend. Moreover if $(\E_{1},\psi_{1}),(\E_{2},\psi_{2})\in J_{B,\shC}$
then we have maps $\pr_{i}\colon(\E_{1}\oplus\E_{2},\psi_{1}\oplus\psi_{2})\to(\E_{i},\psi_{i})$
for $i=1,2$, where $\pr_{i}$ is the projection.

Given an $A$-module $N$ then $\Hom_{A}(\shH,N)$ is $A$-linearly
isomorphic to the set of natural transformations of sets $\beta_{\E}\colon\E(B)\to\Hom_{A}(\Gamma_{\E},N)$.
Since $\shC$ is additive, by \ref{rem: natural transf are linear for additive categories},
those transformations are automatically $R$-linear. So, if we denote
by $F\colon\shC\to\Mod(B)$ the evaluation if $\Spec B\to\stX$, we
have
\[
\Hom_{A}(\shH,N)\simeq\Hom_{\L_{R}(C,R)}(F,\Hom_{A}(\Gamma,N))\simeq\Hom_{\L_{R}(\shC,A)}(\Gamma,\Hom_{R}(F,N))
\]
By \ref{rem: coends} and \ref{prop:adjoint of Omegastar with the cokernel}
the last module is $\Hom_{A}(\shF_{\Gamma,\shC}(B),N)$. One can now
check that the induced map $\shH\to\shF_{\Gamma,\shC}(B)$ is the
one in the statement. The last claim follows by construction: $x_{\E,b\psi}$
is the image of $x\otimes(b\psi)=b(x\otimes\psi)$ under the map $\Gamma_{\E}\otimes_{R}\E(B)\to\shF_{\Gamma,\shC}(B)$.
The last claim follows from \ref{rem:non filtered colimit}. %
\end{proof}

\subsection{Sheafifying $R$-linear monoidal functors.}

In this section we show how ``ring structures'' on a quasi-coherent
sheaf over $\stX$ correspond to ``monoidal'' structures on the
corresponding Yoneda functor.

We start setting up some definitions:
\begin{defn}
\label{def: pseudo monoidal commutativi associative} Let $\catC$
and $\catD$ be $R$-linear symmetric monoidal categories. A (contravariant)
\emph{pseudo-monoidal }functor $\Omega\colon\catC\to\catD$ is an
$R$-linear (and contravariant) functor together with a natural transformation
\[
\iota_{V,W}^{\Omega}\colon\Omega_{V}\otimes\Omega_{W}\to\Omega_{V\otimes W}\text{ for }V,W\in\catC
\]
A (contravariant) pseudo-monoidal functor $\Omega\colon\catC\to\catD$
is

\begin{enumerate}
\item \emph{\label{enu:commutative for functors} symmetric} or commutative
if for all $V,W\in\catC$ the following diagram is commutative$$
\begin{tikzpicture}[xscale=2.7,yscale=-1.2]     \node (A0_1) at (1, 0) {$\Omega_V\otimes \Omega_W$};     \node (A0_2) at (2, 0) {$\Omega_{V\otimes W}$};     \node (A1_1) at (1, 1) {$\Omega_W\otimes \Omega_V$};     \node (A1_2) at (2, 1) {$\Omega_{W\otimes V}$};     \path (A0_1) edge [->]node [auto] {$\scriptstyle{\iota_{V,W}^\Omega}$} (A0_2);     \path (A0_2) edge [->]node [auto] {$\scriptstyle{}$} (A1_2);     \path (A1_1) edge [->]node [auto] {$\scriptstyle{\iota_{W,V}^\Omega}$} (A1_2);     \path (A0_1) edge [->]node [auto] {$\scriptstyle{}$} (A1_1);   \end{tikzpicture}
$$where the vertical arrows are the obvious isomorphisms;
\item \emph{\label{enu:associative for functors} associative} if for all
$V,W,Z\in\catC$ the following diagram is commutative $$
 \begin{tikzpicture}[xscale=4,yscale=-1.2]     \node (A0_0) at (0, 0) {$\Omega_V\otimes \Omega_W\otimes \Omega_Z$};     \node (A0_1) at (1, 0) {$\Omega_{V\otimes W}\otimes \Omega_Z$};     \node (A1_0) at (0, 1) {$\Omega_V\otimes \Omega_{W\otimes Z}$};     \node (A1_1) at (1, 1) {$\Omega_{V\otimes W\otimes Z}$};     \path (A0_0) edge [->]node [auto] {$\scriptstyle{\iota_{V,W}^\Omega \otimes \id}$} (A0_1);     \path (A1_0) edge [->]node [auto] {$\scriptstyle{\iota_{V,W\otimes Z}^\Omega}$} (A1_1);     \path (A0_0) edge [->]node [auto] {$\scriptstyle{\id\otimes \iota_{W,Z}^\Omega}$} (A1_0);     \path (A0_1) edge [->]node [auto] {$\scriptstyle{\iota_{V\otimes W,Z}^\Omega}$} (A1_1);   \end{tikzpicture}
$$
\end{enumerate}
If $I$ and $J$ are the unit objects of $\catC$ and $\catD$ respectively,
a unity for $\Omega$ is a morphism $1\colon J\to\Omega_{I}$ such
that, for all $V\in\catC$, the compositions
\[
\Omega_{V}\otimes J\to\Omega_{V}\otimes\Omega_{I}\to\Omega_{V\otimes I}\to\Omega_{V}\text{ and }J\otimes\Omega_{V}\to\Omega_{I}\otimes\Omega_{V}\to\Omega_{I\otimes V}\to\Omega_{V}
\]
coincide with the natural isomorphisms $\Omega_{V}\otimes J\to\Omega_{V}$
and $J\otimes\Omega_{V}\to\Omega_{V}$ respectively. A (contravariant)\emph{
monoidal} functor $\Omega\colon\catC\to\catD$ is a symmetric and
associative pseudo-monoidal (contravariant) functor with a unity $1$.
A (contravariant) strong monoidal functor $\Omega\colon\catC\to\catD$
is a (contravariant) monoidal functors such that all the maps $\iota_{V,W}^{\Omega}$and
$1\colon J\to\Omega_{I}$ are isomorphisms.
\end{defn}

A morphism of pseudo-monoidal functors $(\Omega,\iota^{\Omega})\to(\Gamma,\iota^{\Gamma})$,
called a monoidal morphism or transformation, is a natural transformation
$\Omega\to\Gamma$ which commutes with the monoidal structures $\iota^{*}$.
A morphism of monoidal functors is a monoidal transformation preserving
the unities.
\begin{defn}
\label{def: monoidal subcategories} We define the categories:

\begin{itemize}
\item $\Rings_{A}\stX$, whose objects are $\alB\in\QCoh_{A}\stX$ with
an $A$-linear map $m\colon\alB\otimes_{\odi{\stX_{A}}}\alB\to\alB$,
called the multiplication;
\item $\QAlg_{A}\stX$, as the (not full) subcategory of $\Rings_{A}\stX$
whose objects are $\alB$ with a commutative, associative multiplication
with a unity and the arrows are morphisms preserving unities;
\end{itemize}
We also set $\Rings\stX=\Rings_{R}\stX$ and $\QAlg\stX=\QAlg_{R}\stX$. 

Let $\shD$ be a monoidal subcategory of $\QCoh\stX$, that is a subcategory
such that $\odi{\stX}\in\shD$ and for all $\E,\E'\in\shD$ we have
$\E\otimes\E'\in\shD$. We define the category $\PML_{R}(\shD,A)$
(resp. $\ML_{R}(\shD,A)$), whose objects are $\Gamma\in\L_{R}(\shD,A)$
with a pseudo-monoidal (resp. monoidal) structure.
\end{defn}

\begin{rem}
If $q\colon\stX_{A}\to\stX$ is the projection, the equivalence $\QCoh\stX_{A}\to\QCoh_{A}\stX$
extends to an equivalence 
\[
q_{*}\colon\QRings\stX_{A}\to\QRings_{A}\stX
\]
Indeed, if $\shG\in\QCoh\stX_{A}$ then $q_{*}(\shG\otimes_{\odi{\stX_{A}}}\shG)=q_{*}\shG\otimes_{\odi{\stX_{A}}}q_{*}\shG$
(see \ref{rem: remark on the tensor product over A}). We will use
the following notation, which is somehow implicit in the definition
of $\QAlg_{A}\stX$: a sheaf $\alB\in\QRings_{A}\stX$ with $\alB\simeq q_{*}\alB'$
is associative (resp. commutative, has a unity, ...) if $\alB'$ has
the same property.

If $\alB\in\QRings_{A}\stX$ with multiplication $m$, then the composition
$\alB\otimes_{\odi{\stX}}\alB\to\alB\otimes_{\odi{\stX_{A}}}\alB\to\alB$
induces a ring structure on $\alB$ as an $\odi{\stX}$-module, i.e.
$\alB\in\QRings\stX$. Moreover $\alB\in\QRings_{A}\stX$ is associative
(resp. commutative, has a unity) if and only if $\alB\in\QRings\stX$
has the same property. If $\alB\in\QAlg_{A}\stX$ we can form the
relative spectrum $\Spec\alB$ over $\stX_{A}$ and also over $\stX$.
The final result is the same.
\end{rem}

\begin{rem}
For $\alB\in\Rings_{A}\stX$ or $\Gamma\in\PMon_{R}(\shD,A)$ having
a unity is a property, not an additional datum. Indeed in both cases
unities are unique. For rings it is obvious. In the case of functors,
using that left and right units $u\colon J\otimes J\to J$, $v\colon I\otimes I\to I$
coincides, the following diagram shows that two unities $\alpha,\beta\colon J\to\Gamma_{I}$
coincide.    \[   \begin{tikzpicture}[xscale=2.0,yscale=-1.2]     \node (A0_0) at (0, 0) {$\Gamma_I$};     \node (A0_1) at (1, 0) {$\Gamma_I \otimes J$};     \node (A1_0) at (0, 1) {$J$};     \node (A1_1) at (1, 1) {$J\otimes J$};     \node (A1_2) at (2, 1) {$\Gamma_I \otimes \Gamma_I$};     \node (A1_3) at (3, 1) {$\Gamma_{I\otimes I}$};     \node (A1_4) at (4, 1) {$\Gamma_I$};     \node (A2_0) at (0, 2) {$\Gamma_I$};     \node (A2_1) at (1, 2) {$J\otimes \Gamma_I$};     \path (A0_1) edge [->]node [auto] {$\scriptstyle{\id\otimes \beta}$} (A1_2);     \path (A0_0) edge [->]node [auto] {$\scriptstyle{}$} (A0_1);     \path (A2_0) edge [->]node [auto] {$\scriptstyle{}$} (A2_1);     \path (A1_0) edge [->]node [auto] {$\scriptstyle{u^{-1}}$} (A1_1);     \path (A1_0) edge [->]node [auto] {$\scriptstyle{\alpha}$} (A0_0);     \path (A1_1) edge [->]node [auto] {$\scriptstyle{\alpha\otimes \id}$} (A0_1);     \path (A1_1) edge [->]node [auto] {$\scriptstyle{\alpha\otimes \beta}$} (A1_2);     \path (A1_3) edge [->]node [auto] {$\scriptstyle{\Gamma_v}$} (A1_4);     \path (A1_0) edge [->]node [auto] {$\scriptstyle{\beta}$} (A2_0);     \path (A1_1) edge [->]node [auto,swap] {$\scriptstyle{\id\otimes \beta}$} (A2_1);     \path (A1_2) edge [->]node [auto] {$\scriptstyle{}$} (A1_3);     \path (A2_0) edge [->,bend left=40]node [auto,swap] {$\scriptstyle{\id}$} (A1_4);     \path (A0_0) edge [->,bend right=40]node [auto] {$\scriptstyle{\id}$} (A1_4);     \path (A2_1) edge [->]node [auto,swap] {$\scriptstyle{\alpha\otimes \id}$} (A1_2);   \end{tikzpicture}   \]  
\end{rem}

Let $\shD$ be a monoidal subcategory of $\QCoh\stX$. If $\alB\in\Rings_{A}\stX$
with multiplication $m$, we endow $\Omega^{\alB}\in\L_{R}(\shD,A)$
with the pseudo monoidal structure
\[
\iota^{\alB}\colon\Hom(\E,\alB)\otimes_{A}\Hom(\overline{\E},\alB)\to\Hom(\E\otimes\overline{\E},\alB\otimes_{\odi{\stX_{A}}}\alB)\to\Hom(\E\otimes\overline{\E},\alB)\comma\E,\overline{\E}\in\shD
\]
that is $\iota_{\E,\overline{\E}}^{\alB}(\phi\otimes\psi)=m\circ(\phi\otimes\psi)$.
If $1\in\alB$ is a unity then we set 
\[
A\to\Omega_{\odi{\stX}}^{\alB}=\Hl^{0}(\alB)\comma1\mapsto1_{\alB}=1
\]

\begin{prop}
\label{prop:Omegastar on algebras} The structures defined above yield
an extension of the functor $\Omega^{*}\colon\QCoh_{A}\stX\to\L_{R}(\shD,A)$
to a functor $\Omega^{*}\colon\QRings_{A}\stX\to\PML_{R}(\shD,A)$.
Moreover if $\alB\in\Rings_{A}\stX$ is associative (resp. commutative,
has a unity $1\in\alB$) then $\Omega^{\alB}$ is associative (resp.
symmetric, has unity $1_{\alB}\in\Omega_{\odi{\stX}}^{\alB}$). In
particular we also get a functor $\Omega^{*}\colon\QAlg_{A}\stX\to\ML_{R}(\shD,A)$.
\end{prop}

\begin{proof}
Consider the expression
\[
\iota_{\E,\overline{\E}}^{\alB}(\phi\otimes\psi):\E\otimes\overline{\E}\ni x\otimes y\longmapsto\phi(x)\cdot\psi(y)\in\alB\text{ for }\phi\colon\E\to\alB,\psi\colon\overline{\E}\to\alB
\]
Here $\cdot$ indicates the multiplication in $\alB$. Associativity
of $\Omega^{\alB}$ in particular is
\[
(\phi_{1}(x_{1})\cdot\phi_{2}(x_{2}))\cdot\phi_{3}(x_{3})=\phi_{1}(x_{1})\cdot(\phi_{2}(x_{2})\cdot\phi_{3}(x_{3}))\text{ for }\phi_{i}\colon\E_{i}\to\alB,x_{i}\in\E_{i}
\]
which clearly follows if $\alB$ is associative. Similarly $\Omega^{\alB}$
is commutative and has a unity if 
\[
\phi_{1}(x_{1})\cdot\phi_{2}(x_{2})=\phi_{2}(x_{2})\cdot\phi_{1}(x_{1})\text{ and }\phi_{1}(x_{1})\cdot1=1\cdot\phi_{1}(x_{1})=\phi_{1}(x_{1})\text{ for }\phi_{i}\colon\E_{i}\to\alB,x_{i}\in\E_{i}
\]
respectively.
\end{proof}
Let $\shC$ be a small monoidal subcategory of $\QCoh\stX$. Given
$\Gamma\in\PMon_{R}(\shC,A)$ with monoidal structure $\iota$, we
denote $\alA_{\Gamma,\shC}=\shF_{\Gamma,\shC}$ and define the multiplication
$m_{\Gamma}\colon\alA_{\Gamma,\shC}\otimes_{\odi{\stX_{A}}}\alA_{\Gamma,\shC}\to\alA_{\Gamma,\shC}$
by
\[
\alA_{\Gamma,\shC}(B)\otimes_{B\otimes_{R}A}\alA_{\Gamma,\shC}(B)\ni x_{\E,\psi}\otimes\overline{x}_{\overline{\E},\overline{\psi}}\to[\iota_{\E,\overline{\E}}(x\otimes\overline{x})]_{\E\otimes\overline{\E},\psi\otimes\overline{\psi}}\in\alA_{\Gamma,\shC}(B)
\]
where $\Spec B\to\stX$ is a map, $\E,\overline{\E}\in\shC$, $\psi\in\E(B)$,
$\overline{\psi}\in\overline{\E}(B)$, $x\in\Gamma_{\E}$, $\overline{x}\in\Gamma_{\overline{\E}}$
(see \ref{def: unit and counit for adjunction}). In other words we
claim that the sum over all $\E,\overline{\E}\in\shC$ of the compositions
\[
\Gamma_{\E}\otimes\E\otimes\Gamma_{\overline{\E}}\otimes\overline{\E}\to\Gamma_{\E}\otimes\Gamma_{\overline{\E}}\otimes(\E\otimes\overline{\E})\to\Gamma_{\E\otimes\overline{\E}}\otimes(\E\otimes\overline{\E})\to\alA_{\Gamma,\shC}
\]
induces a map $m_{\Gamma}\colon\alA_{\Gamma,\shC}\otimes_{\odi{\stX_{A}}}\alA_{\Gamma,\shC}\to\alA_{\Gamma,\shC}$
. We continue to denote by $\alA_{\Gamma,\shC}$ the sheaf $\shF_{\Gamma,\shC}$
together with the multiplication map just defined. If $1\in\Gamma_{\odi{\stX}}$
is a unity we set $1_{\Gamma}\in\alA_{\Gamma,\shC}$ the image of
$1$ under the morphism $\Gamma_{\odi{\stX}}\to\Omega_{\odi{\stX}}^{\alA_{\Gamma,\shC}}=\Hl^{0}(\alA_{\Gamma,\shC})$.
\begin{prop}
\label{prop: from sheaves to algebras}The structures defined above
yield an extension of the functor $\shF_{*,\shC}\colon\L_{R}(\shC,A)\to\QCoh_{A}\stX$
to a functor $\alA_{*,\shC}\colon\PML_{R}(\shC,A)\to\Rings_{A}\stX$
which is left adjoint to $\Omega^{*}\colon\Rings_{A}\stX\to\PML_{R}(\shC,A)$.
More precisely, if $\alB\in\Rings_{A}\stX$ then the morphism $\delta_{\alB}\colon\alA_{\Omega^{\alB},\shC}\to\alB$
preserves multiplications and unities, while if $\Gamma\in\PMon_{R}(\shC,A)$
then the natural transformation $\gamma_{\Gamma}\colon\Gamma\to\Omega^{\alA_{\Gamma,\shC}}$
is monoidal and preserves unities (see \ref{def: unit and counit for adjunction}). 

If $\Gamma\in\PMon_{R}(\shC,A)$ is associative (resp. symmetric,
has a unity $1\in\Gamma_{\odi{\stX}}$) then $\alA_{\Gamma,\shC}$
is associative (resp. commutative, has unity $1_{\Gamma}\in\alA_{\Gamma,\shC}$).
In particular we get a functor $\alA_{*,\shC}\colon\ML_{R}(\shC,A)\to\QAlg_{A}\stX$
which is left adjoint to $\Omega^{*}\colon\QAlg_{A}\stX\to\ML_{R}(\shC,A)$.
\end{prop}

\begin{proof}
We first show that the multiplication is well defined. Using twice
the universal property of the coend definition of $\alA_{\Gamma,\shC}(B)$
we have that 
\[
\Hom_{B\otimes_{R}A}(\alA_{\Gamma,\shC}(B),\End_{B\otimes_{R}A}(\alA_{\Gamma,\shC}(B)))
\]
 equals
\[
\Hom_{\L_{R}(\shC,A)}(\Gamma,\Hom_{B}(F,\Hom_{\L_{R}(\shC,A)}(\Gamma,\Hom_{B}(F,\alA_{\Gamma,\shC}(B)))))
\]
An object in the above set can be seen as a map
\[
\Gamma_{\E}\times\E(B)\times\Gamma_{\overline{\E}}\times\overline{\E}(B)\to\alA_{\Gamma,\shC}
\]
which is $A$-linear in the first and third coordinates and $B$-linear
in the others. Moreover there is a compatibility with respect to morphisms
in the variables $\E$ and $\overline{\E}$. Thus everything boils
down in proving these properties for the map
\[
(x,\psi,\overline{x},\overline{\psi})\mapsto[\iota_{\E,\overline{\E}}(x\otimes\overline{x})]_{\E\otimes\overline{\E},\psi\otimes\overline{\psi}}
\]
This is elementary and left to the reader. Also the fact that if $\Gamma$
is a associative (resp. symmetric, has a unity $1\in\Gamma_{\odi{\stX}}$)
then $\alA_{\Gamma,\shC}$ is associative (resp. commutative, has
unity $1_{\Gamma}\in\alA_{\Gamma,\shC}$) is a tedious elementary
proof.

The fact that $\gamma_{\Gamma}\colon\Gamma\to\Omega^{\alA_{\Gamma,\shC}}$
is monoidal and preserve the unit really comes from construction.
Same for the proof that $\delta_{\alB}\colon\alA_{\Omega^{\alB},\shC}\to\alB$
preserves multiplications and unities. The adjointness of the functor
in the statement follows formally from this.
\end{proof}

\section{Yoneda embeddings.}

In this section we address the problem of when the Yoneda functor
$\QCoh_{A}\stX\to\L_{R}(\shC,A)$ is fully faithful and describe its
essential image. This will led us to the notion of generating category
and left exactness for functors in $\L_{R}(\shC,A)$.

We fix an $R$-algebra $A$, a pseudo-geometric fibered category $\stX$
and a small subcategory $\shC$ of $\QCoh\stX$. In particular $\QCoh(\stX)$
is an abelian category.
\begin{defn}
Let $\shD\subseteq\QCoh\stX$ be a subcategory. A sheaf $\shG\in\QCoh(\stX)$
is generated by $\shD$ if there exists a surjective morphism 
\[
\bigoplus_{i\in I}\E_{i}\to\shG
\]
where $I$ is a set and $\E_{i}\in\shD$ for all $i\in I$. A sheaf
$\shG\in\QCoh_{A}\stX$ is generated by $\shD$ if it is so as an
object of $\QCoh\stX$. Equivalently, a sheaf $\shG\in\QCoh\stX_{A}$
is generated by $\shD$ if $h_{*}\shG\in\QCoh\stX$ is so, where $h\colon\stX_{A}\to\stX$
is the projection. We define $\QCoh_{A}^{\shD}\stX$ as the subcategory
of $\QCoh_{A}\stX$ of sheaves $\shG$ generated by $\shD$ and such
that, for all maps $\E\to\shG$ with $\E\in\shD^{\oplus}$, also $\Ker\psi$
is generated by $\shD$.

If $\shD'$ is another subcategory of $\QCoh\stX$ we will say that
$\shD$ generates $\shD'$ if all quasi-coherent sheaves in $\shD'$
are generated by $\shD$.
\end{defn}

\begin{rem}
\label{rem: local characterization of generation by D} Consider a
set of morphisms $\{U_{j}=\Spec B_{j}\to\stX\}_{j\in J}$ such that
$\sqcup_{j}U_{j}\to\stX$ is an atlas. By \ref{prop:reducing to atlases for quasi-geometric categories}
we have the following characterizations. If $\shG\in\QCoh\stX$ then
$\shG$ is generated by $\shD$ if and only if 
\[
\forall j\in J\comma x\in\shG(B_{j})\comma\exists\E\in\shD^{\oplus}\comma\phi\in\E(B_{j})\comma u\colon\E\to\shG\text{ such that }u(\phi)=x
\]
If $\shH\in\QCoh(\stX)$ and $\psi\colon\shH\to\shG$ is a map then
$\Ker\psi$ is generated by $\shD$ if and only
\[
\forall j\in J\comma y\in\shH(B_{j})\text{ with }\psi(y)=0\comma\exists\ \overline{\E}\in\shD^{\oplus}\comma\phi\in\overline{\E}(B_{j})\comma v\colon\overline{\E}\to\shH\text{ such that }\psi v=0\comma v(\phi)=y
\]
In particular if $\shG\in\QCoh^{\shD}\stX$, $\shH\in\QCoh\stX$ is
generated by $\shD$ and $\shH\to\shG$ is a map then $\Ker\alpha$
is generated by $\shD$.
\end{rem}

\begin{prop}
\label{prop:D autogenerates stable sum}If $\shD$ is a subcategory
of $\QCoh\stX$ then the category $\QCoh_{A}^{\shD}\stX$ is stable
by direct sums. In particular $\shD\subseteq\QCoh^{\shD}\stX\iff\shD^{\oplus}\subseteq\QCoh^{\shD}\stX$.
\end{prop}

\begin{proof}
Let $\shF,\shG\in\QCoh_{A}^{\shD}\stX$. Clearly $\shF\oplus\shG$
is generated by $\shD$. Now consider a map $\alpha\colon\E\to\shF\oplus\shG$
with $\E\in\shD^{\oplus}$ and write $\alpha=\phi\oplus\psi$. By
\ref{rem: local characterization of generation by D} it follows that
$\Ker\alpha=\Ker(\phi_{|\Ker(\psi)}\colon\Ker\psi\to\shF)$ is generated
by $\shD$ because $\Ker(\psi)$ is generated by $\shD$ and $\shF\in\QCoh_{A}^{\shD}\stX$.
\end{proof}
If $g\colon U=\Spec B\to\stX$ is a map from a scheme then $(J_{B,\shC})^{\op}$
(see \ref{prop:FGammaCB as a direct limit}) is not a filtered category
in general. Thus if $\Gamma\in\L_{R}(\shC,A)$ and $x_{\E,\phi}\in\shF_{\Gamma,\shC}(B)$
it is very difficult to understand when $x_{\E,\phi}$ is zero in
$\shF_{\Gamma,\shC}(B)$. Luckily, under some hypothesis this is possible.
\begin{lem}
\label{lem:when J is filtered} Assume $\shC\subseteq\QCoh^{\shC}\stX$.
Then for all flat maps $g\colon\Spec B\to\stX$ the category $(J_{B,\shC^{\oplus}})^{\text{op}}$
is filtered. In this case, given $\Gamma\in\L_{R}(\shC,A)$, every
element of $\shF_{\Gamma,\shC}(B)$ are of the form $x_{\E,\phi}$
for some $(\E,\phi)\in J_{B,\shC}$ and such an element is zero if
and only if there exists $(\overline{\E},\overline{\phi})\to(\E,\phi)$
in $J_{B,\shC}$ such that $\Gamma_{u}(x)=0$.
\end{lem}

\begin{proof}
By \ref{prop:find a directed category} and \ref{prop:D autogenerates stable sum}
we can assume $\shC=\shC^{\oplus}$. The last claim follows from \ref{prop:FGammaCB as a direct limit}
and the fact that $J_{B,\shC}^{\op}$ is filtered, which we now prove.
By \ref{prop:FGammaCB as a direct limit} it remains to show that
for all maps $\alpha,\beta\colon(\E,\phi)\to(\overline{\E},\overline{\phi})$
in $J_{B,\shC}$ there exists $u\colon(\E',\phi')\to(\E,\phi)$ in
$J_{B,\shC}$ such that $\alpha u=\beta u$. Let $\shK=\Ker((\alpha-\beta)\colon\E\to\overline{\E})$.
Since $\Spec B\to\stX$ is flat, by \ref{prop:flat maps are exact}
one has 
\[
\phi\in\shK(B)=\Ker((\alpha_{B}-\beta_{B})\colon\E(B)\to\overline{\E}(B))
\]
By assumption $\shK$ is generated by $\shC$ and, since $\shC$ is
additive, there exist $\E'\in\shC$, a map $u\colon\E'\to\shK$ and
$\phi'\in\E'(B)$ such that $u(\phi')=\phi$. So $(\E',\phi')\to(\E,\phi)$
is an arrow in $J_{B,\shC}$ such that $(\alpha-\beta)u=0$ as required.
\end{proof}
In what follows we work out sufficient (and sometimes necessary) conditions
for the surjectivity or injectivity of $\delta_{\shG}\colon\shF_{\Omega^{\shG},\shC}\to\shG$.
Recall that $\delta_{\shG}(u_{\E,\phi})=u(\phi)$ for $\E\in\shC^{\oplus}$,
$\Spec B\to\stX$, $\phi\in\E(B)$, $u\in\Omega_{\E}^{\shG}=\Hom_{\stX}(\E,\shG)$
(see \ref{prop:FGammaC on affine schemes}).
\begin{lem}
\label{lem:surjectivity natural map F} If $\shG\in\QCoh_{A}\stX$
the map $\delta_{\shG}\colon\shF_{\Omega^{\shG},\shC}\to\shG$ is
surjective if and only if $\shG$ is generated by $\shC$.
\end{lem}

\begin{proof}
By \ref{prop:find a directed category} we can assume $\shC=\shC^{\oplus}$.
Let $\{g_{i}\colon U_{i}=\Spec B_{i}\to\stX\}$ be a set of maps such
that $\sqcup_{i}U_{i}\to\stX$ is an atlas. By \ref{prop:reducing to atlases for quasi-geometric categories}
$\delta_{\shG}$ is surjective if and only if $\delta_{\shG,U_{i}}\colon\shF_{\Omega^{\shG},\shC}(U_{i})\to\shG(U_{i})$
is surjective for all $i\in I$. By \ref{prop:FGammaCB as a direct limit}
$\Imm\delta_{\shG,U_{i}}$ is the set of elements of $\shG(U_{i})$
of the form $\delta_{\shG}(u_{\E,\phi})=u(\phi)$ for $\E\in\shC^{\oplus}$,
$\phi\in\E(U_{i})$, $u\in\Omega_{\E}^{\shG}=\Hom_{\stX}(\E,\shG)$.
So the claim follows from \ref{rem: local characterization of generation by D}.
\end{proof}
\begin{lem}
\label{lem:injectivity natural map F} Let $\shG\in\QCoh_{A}\stX$.
If for all maps $\E\to\shG$ with $\E\in\shC^{\oplus}$ the kernel
$\Ker\phi$ is generated by $\shC$ then the map $\delta_{\shG}\colon\shF_{\Omega^{\shG},\shC}\to\shG$
is injective. The converse holds if $\shC\subseteq\QCoh^{\shC}\stX$.
\end{lem}

\begin{proof}
By \ref{prop:find a directed category} we can assume $\shC=\shC^{\oplus}$.
Let $\{g_{i}\colon U_{i}=\Spec B_{i}\to\stX\}$ be a set of maps such
that $\sqcup_{i}U_{i}\to\stX$ is an atlas. By \ref{prop:reducing to atlases for quasi-geometric categories}
$\delta_{\shG}$ is injective if and only if $\delta_{\shG,U_{i}}\colon\shF_{\Omega^{\shG},\shC}(U_{i})\to\shG(U_{i})$
is injective for all $i\in I$. We start proving that $\delta_{\shG}$
is injective if the hypothesis in the statement are fulfilled. So
let $z\in\Ker\delta_{\shG,U_{i}}$. By \ref{prop:FGammaCB as a direct limit}
there exists $\E\in\shC$, $\phi\in\E(U_{i})$ and $u\colon\E\to\shG$
such that $z=u_{\E,\phi}$. Moreover $\delta_{\shG}(u_{\E,\phi})=u(\phi)=0$.
Set $\shK=\Ker u$. Since $\phi\in\shK(U_{i})$ and, by hypothesis,
$\shK$ is generated by $\shC$ there exist $\overline{\E}\in\shC$,
$\overline{\phi}\in\overline{\E}(U_{i})$ and a map $v\colon\overline{\E}\to\shK$
such that $v(\overline{\phi})=\phi$. If we denote by $v$ also the
composition $\overline{\E}\to\shK\to\E$ we have
\[
u_{\E,\phi}=(\Omega_{v}^{\shG}(u))_{\overline{\E},\overline{\phi}}=(uv)_{\overline{\E},\overline{\phi}}=0_{\overline{\E},\overline{\phi}}=0\text{ in \ensuremath{\shF_{\Gamma,\shC}(U_{i})}}
\]
Assume now that $\delta_{\shG}$ is injective and $\shC\subseteq\QCoh^{\shC}\stX$
and let $u\colon\E\to\shG$ be a map with $\E\in\shC$. We have to
prove that $\shK=\Ker u$ is generated by $\shC$. If $\phi\in\shK(U_{i})\subseteq\E(U_{i})$,
then $u(\phi)=\delta_{\shG}(u_{\E,\phi})=0$. So $u_{\E,\phi}=0$
and the conclusion follows from \ref{rem: local characterization of generation by D}
and \ref{lem:when J is filtered}.
\end{proof}
In general we can still conclude that:
\begin{prop}
\label{prop: delta E is an isomorphism} If $\E\in\shC^{\oplus}$
then the map $\delta_{\E}\colon\shF_{\Omega^{\E},\shC}\to\E$ is an
isomorphism.
\end{prop}

\begin{proof}
By \ref{prop:find a directed category} we can assume $\shC=\shC^{\oplus}$.
Let $\shH\in\QCoh\stX$. The map 
\[
\Hom_{\stX}(\E,\shH)\to\Hom_{\L_{R}(\shC,R)}(\Omega^{\E},\Omega^{\shH})\simeq\Hom_{\stX}(\shF_{\Omega^{\E},\shC},\shH)
\]
maps $\id_{\E}$ to $\delta_{\E}$ and thus is induced by $\delta_{\E}$.
By the enriched Yoneda's lemma or a direct check we see that the above
map and therefore $\delta_{\E}$ are isomorphisms. %
\end{proof}
\begin{thm}
\label{thm:when Omegastar is fully faithful} Let $\shD_{A}$ be the
subcategory of $\QCoh_{A}\stX$ of sheaves $\shG$ such that $\delta_{\shG}\colon\shF_{\Omega^{\shG},\shC}\to\shG$
is an isomorphism. Then $\shD_{A}$ is an additive category containing
$\QCoh_{A}^{\shC}\stX$, $\shC^{\oplus}\subseteq\shD_{R}$ and the
functor
\[
\Omega^{*}\colon\shD_{A}\to\L_{R}(\shC,A)
\]
is fully faithful. Moreover if $\shC\subseteq\QCoh^{\shC}\stX$ then
$\shD_{A}=\QCoh_{A}^{\shC}\stX$.
\end{thm}

\begin{proof}
The category $\shD_{A}$ is additive because $\Omega^{*}$ and $\shF_{*,\shC}$
are additive. All the other claims follows from \ref{lem:surjectivity natural map F},
\ref{lem:injectivity natural map F}, \ref{prop: delta E is an isomorphism}
and the fact that $\delta_{\shG}$ is the counit of an adjiunction.
\end{proof}

\section{Left exactness}

We keep the notation from the previous section. So $\stX$ is a pseudo-algebraic
fibered category over a ring $R$, $\shC\subseteq\QCoh(\stX)$ is
a small full subcategory and $\shD\subseteq\QCoh(\stX)$ is any full
subcategory. The symbol $A$ instead will always denote an $R$-algebra.

Now we want to address the problem of what is the essential image
of the Yoneda functor $\Omega^{*}\colon\QCoh_{A}\stX\to\L_{R}(\shD,A)$.
We will see that for $\shF\in\QCoh_{A}\stX$ the associated Yoneda
functor $\Omega^{\shF}$ is always ``left exact'' and we will give
sufficient conditions assuring that ``left exact'' functors in $\L_{R}(\shD,A)$
are Yoneda functors associated with some quasi-coherent sheaf on $\stX$.
Since $\shD$ is not abelian, we introduce an ad hoc notion of left
exactness.
\begin{defn}
Let $\shD$ be a subcategory of $\QCoh\stX$. A map
\[
\alpha=(\alpha_{kj})\colon\bigoplus_{k\in K}\E_{k}\to\bigoplus_{j\in J}\E_{j}\text{ with }\E_{k},\E_{j}\in\shD
\]
where $K$ and $J$ are sets, is called locally finite (with respect
to the decomposition) if all the maps $\E_{k}\to\bigoplus_{j}\E_{j}$
factors through a finite sum. If $\Gamma\in\L_{R}(\shD,A)$ we set
\[
\Gamma_{\alpha}\colon\prod_{j\in J}\Gamma_{\E_{j}}\to\prod_{k}\Gamma_{\E_{k}}\comma x=(x_{j})_{j}\mapsto(\sum_{j}\Gamma_{\alpha_{kj}}(x_{j}))_{k}
\]
\end{defn}

\begin{rem}
It is easy to verify that the above association is functorial, that
is $\Gamma_{\beta\circ\alpha}=\Gamma_{\alpha}\circ\Gamma_{\beta}$
when it makes sense. If $K$ and $J$ are finite then $\Gamma_{\alpha}$
is obtained applying the extension $\Gamma\in\L_{R}(\shD^{\oplus},A)$
(see \ref{prop:find a directed category}) on the map $\alpha$.

We should warn the reader that even if $\alpha$ is an arrow in $\shD$
the map $\Gamma_{\alpha}$ in definition above is not obtained applying
$\Gamma$ on the arrow. The problem here is that $\Gamma$ may not
transform direct sums into direct products. This is clearly an abuse
of notation, but we hope it will not led to confusion. The map $\alpha$
as well as the notion of local finiteness should be interpreted in
a category of ``decompositions''.
\end{rem}

\begin{defn}
\label{def: left exact functors} A \emph{test sequence }for $\shD$
is an exact sequence
\begin{equation}
\bigoplus_{k\in K}\E_{k}\to\bigoplus_{j\in J}\E_{j}\to\E\to0\text{ with }\E,\E_{j},\E_{k}\in\shD\text{ for all }j\in J,k\in K\label{eq:exact sequence for functor exactness}
\end{equation}
in $\QCoh\stX$ where the first map is locally finite. We will also
say that it is a test sequence for $\E\in\shD$. A finite test sequence
for $\E\in\shD$ is an exact sequence
\[
\E''\to\E'\to\E\to0\text{ with }\E',\E''\in\shD^{\oplus}
\]
Given $\Gamma\in\L_{R}(\shD,A)$ we say that $\Gamma$ is exact on
the test sequence (\ref{eq:exact sequence for functor exactness})
if the sequence \begin{equation}
\begin{tikzpicture}[xscale=2.9,yscale=-0.8]     
\node (A0_2) at (2, 0) {$(x_j )_j$};     
\node (A0_3) at (3, 0) {$(\sum_j \Gamma_{u_{kj}}(x_j))_k$};     
\node (A1_0) at (0.4, 1) {$0$};     
\node (A1_1) at (1.1, 1) {$\Gamma_\E$};     
\node (A1_2) at (2, 1) {$\displaystyle{\prod_{j\in J}\Gamma_{\E_j}}$};     
\node (A1_3) at (3, 1) {$\displaystyle{\prod_{k\in K}\Gamma_{\E_k}}$};    
\node (A2_1) at (1.1, 2) {$x$};     
\node (A2_2) at (2, 2) {$(\Gamma_{u_j}(x))_j$};     
\path (A0_2) edge [|->,gray]node [auto] {$\scriptstyle{}$} (A0_3);     
\path (A1_0) edge [->]node [auto] {$\scriptstyle{}$} (A1_1);     
\path (A2_1) edge [|->,gray]node [auto] {$\scriptstyle{}$} (A2_2);     
\path (A1_2) edge [->]node [auto] {$\scriptstyle{}$} (A1_3);     
\path (A1_1) edge [->]node [auto] {$\scriptstyle{}$} (A1_2);   
\end{tikzpicture}
\label{Gamma on a test sequence}
\end{equation} is exact, where $u_{j}\colon\E_{j}\to\E$, $u_{kj}\colon\E_{k}\to\E_{j}$
denotes the maps in (\ref{eq:exact sequence for functor exactness}).
We say that $\Gamma$ is left exact if it is exact on all test sequences
in $\shD$. We define $\Lex_{R}(\shD,A)$ as the subcategory of $\L_{R}(\shD,A)$
of left exact functors.
\end{defn}

\begin{rem}
\label{rem:extending to test sequences} If $\shD\subseteq\QCoh^{\shD}\stX$
then any surjective map $\mu\colon\bigoplus_{j\in J}\E_{j}\to\E\to0$
can be extended to a test sequence in $\shC$. More generally if $u\colon\bigoplus_{j\in J}\E_{j}\to\bigoplus_{p\in P}\E_{p}$
is a locally finite map, then there exists a locally finite map $\bigoplus_{k}\E_{k}\to\bigoplus_{j}\E_{j}$
whose image is $\Ker(u)$. Indeed let $Q$ be the set of finite subsets
of $J$ and for any $T\in Q$ let $P_{T}\subseteq P$ be a finite
subsets such that each $\E_{j}$ for $j\in T$ maps inside $\bigoplus_{p\in P_{T}}\E_{p}$.
By \ref{prop:D autogenerates stable sum} $\shD^{\oplus}\subseteq\QCoh^{\shD}\stX$
and therefore we can find exact sequences
\[
\bigoplus_{k\in K_{T}}\E_{k}\to\bigoplus_{j\in T}\E_{j}\to\bigoplus_{p\in P_{T}}\E_{p}
\]
Since $u$ is locally finite we have an exact sequence
\[
\bigoplus_{T\in Q}\bigoplus_{k\in K_{T}}\E_{k}\to\bigoplus_{j\in J}\E_{j}\to\bigoplus_{p\in P}\E_{p}
\]
where the first map is locally finite. In particular given a Cartesian
diagram of solid arrows   \[   \begin{tikzpicture}[xscale=2.0,yscale=-1.2]     \node (A0_0) at (0, 0) {$\bigoplus_s \E_s$};     \node (A1_1) at (1, 1) {$\shH$};     \node (A1_2) at (2, 1) {$\bigoplus_t \E_t$};     \node (A2_1) at (1, 2) {$\bigoplus_q \E_q$};     \node (A2_2) at (2, 2) {$\bigoplus_p \E_p$};     \path (A0_0) edge [->,dashed,bend right=20]node [auto] {$\scriptstyle{}$} (A1_2);     \path (A1_1) edge [->]node [auto] {$\scriptstyle{}$} (A1_2);     \path (A1_2) edge [->]node [auto] {$\scriptstyle{}$} (A2_2);     \path (A0_0) edge [->,dashed]node [auto] {$\scriptstyle{}$} (A1_1);     \path (A1_1) edge [->]node [auto] {$\scriptstyle{}$} (A2_1);     \path (A2_1) edge [->]node [auto] {$\scriptstyle{}$} (A2_2);     \path (A0_0) edge [->,dashed,bend left=20]node [auto] {$\scriptstyle{}$} (A2_1);   \end{tikzpicture}   \] in
which all the maps are locally finite then there exists a surjective
map $\bigoplus_{s}\E_{s}\to\shH$ such that the other dashed arrows
are locally finite. This is because $\shH$ is the kernel of the difference
map
\[
(\bigoplus_{q}\E_{q})\bigoplus(\bigoplus_{t}\E_{t})\to\bigoplus_{p}\E_{p}
\]
which is locally finite.
\end{rem}

\begin{prop}
\label{prop:OmegaF is left exact} If $\shF\in\QCoh_{A}\stX$ then
$\Omega^{\shF}\in\Lex_{R}(\shD,A)$. More generally $\Omega^{\shF}$
is exact on all exact sequences of the form
\[
\bigoplus_{j\in J}\E_{j}\to\bigoplus_{q\in Q}\E_{q}\to0\text{ or }\bigoplus_{k\in K}\E_{k}\to\bigoplus_{j\in J}\E_{j}\to\bigoplus_{q\in Q}\E_{q}\to0\text{ for }\E_{k},\E_{j},\E_{q}\in\shC
\]
where the maps involved are locally finite.
\end{prop}

\begin{proof}
It is enough to observe that $\Hom_{\stX}(-,\shF)$ is left exact
in the usual sense and that
\[
\Hom_{\stX}(\bigoplus_{i}\E_{i},\shF)\simeq\prod_{i}\Hom(\E_{i},\shF)\simeq\prod_{i}\Omega_{\E_{i}}^{\shF}
\]
\end{proof}
\begin{prop}
\label{prop:excatness properties of Omega* and shF*} The functor
$\Omega^{*}\colon\QCoh_{A}\stX\to\L_{R}(\shD,A)$ is left exact. If
$\shC\subseteq\QCoh^{\shC}\stX$ then $\shF_{*,\shC}\colon\L_{R}(\shC,A)\to\QCoh_{A}\stX$
is exact.
\end{prop}

\begin{proof}
For the first claim it is enough to use that $\Hom_{\stX}(\E,-)$
is left exact. For the last part of the statement consider a set of
maps $\{U_{i}=\Spec B_{i}\to\stX\}_{i\in I}$ such that $\sqcup_{i}U_{i}\to\stX$
is an atlas. Let also $\Gamma'\to\Gamma\to\Gamma''$ be an exact sequence
in $\L_{R}(\shC,A)$. By \ref{prop:FGammaCB as a direct limit} the
sequence $\shF_{\Gamma',\shC}(B_{i})\to\shF_{\Gamma,\shC}(B_{i})\to\shF_{\Gamma'',\shC}(B_{i})$
are exact for all $i\in I$ because limit of exact sequences $\Gamma_{\E,\phi}'\to\Gamma_{\E,\phi}\to\Gamma_{\E,\phi}''$
over the category $(J_{B_{i},\shC})^{\op}$, which is filtered thanks
to \ref{lem:when J is filtered}. Applying \ref{prop:reducing to atlases for quasi-geometric categories}
we get the result.
\end{proof}
Recall that if $\Gamma\in\L_{R}(\shC,A)$ then $\gamma_{\Gamma,\E}\colon\Gamma_{\E}\to\Omega_{\E}^{\shF_{\Gamma,\shC}}=\Hom_{\stX}(\E,\shF_{\Gamma,\shC})$
is given by $\gamma_{\Gamma,\E}(x)(\phi)=x_{\E,\phi}$ for $\E\in\shC$,
$x\in\Gamma_{\E}$, $\Spec B\to\stX$ and $\phi\in\E(B)$ (see \ref{prop:FGammaC on affine schemes}).
\begin{lem}
\label{lem:injectivity natural transformation Omega} Assume $\shC\subseteq\QCoh^{\shC}\stX$.
If $\Gamma\in\Lex_{R}(\shC,A)$ and the map $\Omega^{*}\circ\shF_{*,\shC}(\gamma_{\Gamma})\colon\Omega^{\shF_{\Gamma,\shC}}\to\Omega^{\shF_{\Omega^{\shF_{\Gamma,\shC}},\shC}}$
is an isomorphism then the natural transformation $\gamma_{\Gamma}\colon\Gamma\to\Omega^{\shF_{\Gamma,\shC}}$
is an isomorphism.
\end{lem}

\begin{proof}
Let $\{U_{i}=\Spec B_{i}\to\stX\}_{i\in I}$ be a set of maps such
that $\sqcup_{i}U_{i}\to\stX$ is an atlas and let $\Psi\in\L_{R}(\shC,A)$
and $x\in\Ker\gamma_{\Psi,\E}$ for some $\E\in\shC$. We are going
to prove that there exists a surjective map $\mu=\oplus_{j}\mu_{j}\colon\oplus_{j\in J}\E_{j}\to\E$
with $\E_{j},\E\in\shC$ such that $\Psi_{u_{j}}(x)=0$ for all $j$.

If $\phi\in\E(U_{i})$, by \ref{lem:when J is filtered} and the fact
that $\gamma_{\Psi,\E}(x)(\phi)=x_{\E,\phi}$ is zero in $\shF_{\Psi,\shC}(U_{i})$,
there exists $(\E_{\phi},y_{\phi})\in J_{B_{i},\shC}$ and a map $u_{\phi}\colon(\E_{\phi},y_{\phi})\to(\E,\phi)$
such that $\Psi_{u_{\phi}}(x)=0$. Consider the induced map
\[
\bigoplus_{i\in I}\bigoplus_{\phi\in\E(U_{i})}\E_{\phi}\to\E
\]
which is surjective by \ref{prop:reducing to atlases for quasi-geometric categories}.
Writing all the $\E_{\phi}\in\shC^{\oplus}$ as sums of sheaves in
$\shC$ we get the desired surjective map.

We return now to the proof of the statement. Given $x\in\Ker\gamma_{\Gamma,\E}$,
considering a surjection $\mu$ as above for $\Psi=\Gamma$, extending
it via \ref{rem:extending to test sequences} and using the exactness
of $\Gamma$ we can conclude that $x=0$. This means that the natural
transformation $\gamma_{\Gamma}\colon\Gamma\to\Omega^{\shF_{\Gamma,\shC}}$
is injective. Set now $\Pi=\Coker\gamma_{\Gamma,\shC}$. By \ref{prop:excatness properties of Omega* and shF*}
we have an exact sequence
\[
0\to\shF_{\Gamma,\shC}\to\shF_{\Omega^{\shF_{\Gamma,\shC}},\shC}\to\shF_{\Pi,\shC}\to0
\]
This is a split sequence because the composition of $\shF_{*,\shC}(\gamma_{\Gamma})\colon\shF_{\Gamma,\shC}\to\shF_{\Omega^{\shF_{\Gamma,\shC}},\shC}$
and $\delta_{\shF_{\Gamma,\shC}}\colon\shF_{\Omega^{\shF_{\Gamma,\shC}},\shC}\to\shF_{\Gamma,\shC}$
is the identity. So $\Omega^{*}$ maintains the exactness of the above
sequence and therefore 
\[
\Omega^{\shF_{\Pi,\shC}}=\Coker(\Omega^{*}\circ\shF_{*,\shC}(\gamma_{\Gamma}))=0
\]
by hypothesis. We want to prove that $\Pi=0$. Let $x\in\Pi_{\E}$
for $\E\in\shC$. Since $\Omega^{\shF_{\Pi,\shC}}=0$ we have $x\in\Ker\gamma_{\Pi,\E}$.
Consider a surjection $\mu=\oplus_{j}\mu_{j}$ constructed as above
starting from $x\in\Pi_{\E}$ and $\Psi=\Pi$. By \ref{rem:extending to test sequences}
$\mu$ can be extended to a test sequence $\oplus_{k}\E_{k}\to\oplus_{j}\E_{i}\to\E$
because $\shC\subseteq\QCoh^{\shC}\stX$. Since $\Omega^{\shF_{\Gamma,\shC}}\in\Lex_{R}(\shC,A)$
by \ref{prop:OmegaF is left exact} we get a commutative diagram   \[   \begin{tikzpicture}[xscale=2.0,yscale=-1.4]     
\node (A0_1) at (1, 0) {$0$};     
\node (A0_2) at (2, 0) {$0$};     
\node (A1_0) at (0, 1) {$0$};     
\node (A1_1) at (1, 1) {$\Gamma_\E$};     
\node (A1_2) at (2, 1) {$\Omega^{\shF_{\Gamma,\shC}}_\E$};     
\node (A1_3) at (3, 1) {$\Pi_\E$};     
\node (A1_4) at (4, 1) {$0$};     
\node (A2_0) at (0, 2) {$0$};     
\node (A2_1) at (1, 2) {$\displaystyle\prod_{j\in J}\Gamma_{\E_j}$};     
\node (A2_2) at (2, 2) {$\displaystyle\prod_{j\in J}\Omega^{\shF_{\Gamma,\shC}}_{\E_j}$};     
\node (A2_3) at (3, 2) {$\displaystyle\prod_{j\in J}\Pi_{\E_j}$};     
\node (A2_4) at (4, 2) {$0$};     \node (A3_0) at (0, 3) {$0$};     
\node (A3_1) at (1, 3) {$\displaystyle\prod_{k\in K}\Gamma_{\E_k}$};     
\node (A3_2) at (2, 3) {$\displaystyle\prod_{k\in K}\Omega^{\shF_{\Gamma,\shC}}_{\E_k}$};    
\path (A2_1) edge [->]node [auto] {$\scriptstyle{}$} (A3_1);    
\path (A2_2) edge [->]node [auto] {$\scriptstyle{}$} (A2_3);     
\path (A1_2) edge [->]node [auto] {$\scriptstyle{}$} (A1_3);     
\path (A2_1) edge [->]node [auto] {$\scriptstyle{}$} (A2_2);     
\path (A1_0) edge [->]node [auto] {$\scriptstyle{}$} (A1_1);     
\path (A3_0) edge [->]node [auto] {$\scriptstyle{}$} (A3_1);     
\path (A1_3) edge [->]node [auto] {$\scriptstyle{\beta}$} (A2_3);     
\path (A1_1) edge [->]node [auto] {$\scriptstyle{}$} (A1_2);     
\path (A2_3) edge [->]node [auto] {$\scriptstyle{}$} (A2_4);     
\path (A2_2) edge [->]node [auto] {$\scriptstyle{}$} (A3_2);     
\path (A1_3) edge [->]node [auto] {$\scriptstyle{}$} (A1_4);     
\path (A1_1) edge [->]node [auto] {$\scriptstyle{}$} (A2_1);     
\path (A0_2) edge [->]node [auto] {$\scriptstyle{}$} (A1_2);     
\path (A0_1) edge [->]node [auto] {$\scriptstyle{}$} (A1_1);     
\path (A2_0) edge [->]node [auto] {$\scriptstyle{}$} (A2_1);     
\path (A3_1) edge [->]node [auto] {$\scriptstyle{}$} (A3_2);     
\path (A1_2) edge [->]node [auto] {$\scriptstyle{}$} (A2_2);   
\end{tikzpicture}   \]  in which all the rows and the first two columns are exact. By diagram
chasing it is easy to conclude that $\beta$ is injective. Since by
construction $\beta(x)=0$ we can conclude that $x=0$.
\end{proof}
\begin{defn}
We define $\Lex_{R}^{\shC}(A)$ as the subcategory of $\Lex_{R}(\shC,A)$
of functors $\Gamma$ such that $\shF_{\Gamma,\shC}\in\QCoh_{A}^{\shC}\stX$.
\end{defn}

\begin{thm}
\label{thm:relative theorem for equivalence} Assume $\shC\subseteq\QCoh^{\shC}\stX$.
Then the functors 
\[
\Omega^{*}\colon\QCoh_{A}^{\shC}\stX\to\Lex_{R}^{\shC}(A)\text{ and }\shF_{*,\shC}\colon\Lex_{R}^{\shC}(A)\to\QCoh_{A}^{\shC}\stX
\]
are quasi-inverses of each other. 
\end{thm}

\begin{proof}
Let $\Gamma\in\L_{R}(\shC,A)$ be such that $\shF_{\Gamma,\shC}\in\QCoh_{A}^{\shC}\stX$.
The composition
\[
\delta_{\shF_{\Gamma,\shC}}\circ\shF_{*,\shC}(\gamma_{\Gamma})\colon\shF_{\Gamma,\shC}\to\shF_{\Omega^{\shF_{\Gamma,\shC}},\shC}\to\shF_{\Gamma,\shC}
\]
is the identity and $\delta_{\shF_{\Gamma,\shC}}$ is an isomorphism
since $\shF_{\Gamma,\shC}\in\QCoh_{A}^{\shC}\stX$ by \ref{thm:when Omegastar is fully faithful}.
Thus $\shF_{*,\shC}(\gamma_{\Gamma})$ and therefore $\Omega^{*}\circ\shF_{*,\shC}(\gamma_{\Gamma})$
are isomorphisms. By \ref{lem:injectivity natural transformation Omega}
we can conclude that if $\Gamma\in\Lex_{R}^{\shC}(A)$ then $\gamma_{\Gamma}\colon\Gamma\to\Omega^{\shF_{\Gamma,\shC}}$
is an isomorphism. The result then follows from \ref{thm:when Omegastar is fully faithful}
and \ref{prop:OmegaF is left exact}.
\end{proof}
The following result allow us to extend results from small subcategories
of $\QCoh\stX$ to any subcategory.
\begin{prop}
\label{prop: QCoh X is Grothendieck} The category $\QCoh\stX$ is
generated by a small subcategory. Equivalently $\QCoh\stX$ has a
generator, that is there exists $\E\in\QCoh\stX$ such that $\{\E\}$
generates $\QCoh\stX$. %
\end{prop}

\begin{proof}
Follows from \ref{prop:reducing to atlases for quasi-geometric categories}
and \cite[Tag 0780]{SP014}.
\end{proof}
\begin{rem}
\label{rem: finding small subcategory generating QCoh X} If $\shD\subseteq\QCoh\stX$
generates $\QCoh\stX$ there always exists a small subcategory $\overline{\shD}\subseteq\shD$
that generates $\QCoh\stX$. Indeed if $\E$ is a generator of $\QCoh\stX$
it is enough to take a subset of sheaves in $\shD$ that generates
$\E$.
\end{rem}

\begin{thm}
\label{thm:general theorem for equivalences} Let $\shD\subseteq\QCoh\stX$
be a subcategory that generates $\QCoh\stX$. Then the functor
\[
\Omega^{*}\colon\QCoh_{A}\stX\to\Lex_{R}(\shD,A)
\]
is an equivalence of categories and, if $\shD$ is small, $\shF_{*,\shD}\colon\Lex_{R}(\shD,A)\to\QCoh_{A}\stX$
is a quasi-inverse. In particular if $\overline{\shD}\subseteq\shD$
is a subcategory that generates $\QCoh\stX$ the restriction functor
$\Lex_{R}(\shD,A)\to\Lex_{R}(\overline{\shD},A)$ is an equivalence.
\end{thm}

\begin{proof}
If $\shD$ is small we have $\QCoh_{A}\stX=\QCoh_{A}^{\shD}\stX$,
$\Lex_{R}(\shD,A)=\Lex_{R}^{\shD}(A)$ and everything follows from
\ref{thm:relative theorem for equivalence}. In particular the restriction
$\Lex_{R}(\shD,A)\to\Lex_{R}(\overline{\shD},A)$ is an equivalence
if $\overline{\shD}\subseteq\shD$, they are small and they generate
$\QCoh\stX$. Assume now that $\shD$ is not necessarily small and
consider a small subcategory $\shC\subseteq\shD$ that generates $\QCoh\stX$,
which exists thanks to \ref{rem: finding small subcategory generating QCoh X}.
The proof of the statement is complete if we prove that the restriction
functor $\Lex_{R}(\shD,A)\to\Lex_{R}(\shC,A)$ is an equivalence.
Given a set $I\subseteq\shD$ we set $\shC_{I}=\shC\cup I\subseteq\shD$.
For all sets $I$, quasi-coherent sheaves are generated by $\shC_{I}$.
Note that we have the restriction functor $-_{|C_{I}}\colon\Lex_{R}(\shD,A)\to\Lex_{R}(\shC_{I},A)$
and the composition $\QCoh_{A}\stX\to\Lex_{R}(\shD,A)\to\Lex_{R}(C_{I},A)$
is an equivalence for all $I$. In particular $-_{|C}\colon\Lex_{R}(\shD,A)\to\Lex_{R}(\shC,A)$
is essentially surjective. We will conclude by proving that it is
fully faithful. Let $\Gamma,\Gamma'\in\Lex_{R}(\shD,A)$. If $\overline{I}\subseteq I$
the restriction functor $\Lex_{R}(\shC_{I},A)\to\Lex_{R}(\shC_{\overline{I}},A)$
is an equivalence of categories. In particular the map $\Hom(\Gamma_{|\shC_{I}},\Gamma'_{|\shC_{I}})\to\Hom(\Gamma_{|\shC_{\overline{I}}},\Gamma'_{|\shC_{\overline{I}}})$
is bijective. Using this, it is elementary to prove that also the
map $\Hom(\Gamma,\Gamma')\to\Hom(\Gamma_{|\shC},\Gamma'_{|\shC})$
is bijective.
\end{proof}

\section{Cohomological exactness}

We keep the notation from the previous section. So $\stX$ is a pseudo-algebraic
fibered category over a ring $R$, $\shC\subseteq\QCoh(\stX)$ is
a small full subcategory and $\shD\subseteq\QCoh(\stX)$ is any full
subcategory. The symbol $A$ instead will always denote an $R$-algebra.

In this section we want to present a cohomological interpretation
of the functors in $\Lex_{R}(\shD,A)$, which will allow us to show
that it is often enough to consider just finite test sequences instead
of arbitrary test sequences.
\begin{rem}
In an abelian category $\alA$, given $X,Y\in\alA$ we can always
define the abelian group $\Ext^{1}(X,Y)$ as the group of extensions
(regardless if $\alA$ has enough injectives) and it has the usual
nice properties on short exact sequences. See 010J\cite[Tag 010J]{SP014}.
In order to avoid set-theoretic problems one should require that $\alA$
is locally small and that, given $X,Y\in\alA$, $\Ext^{1}(X,Y)$ is
a set. This is the case for $\alA=\L_{R}(\shC,R)$, for instance by
looking at the cardinalities of the $\Gamma_{\E}$ for $\Gamma\in\L_{R}(\shC,R)$
and $\E\in\shC$.
\end{rem}

\begin{defn}
\label{def: cohomologically left exact} Given a surjective map $\mu=\oplus_{j}\mu_{j}\colon\oplus_{j}\E_{j}\to\E$
with $\E,\E_{j}\in\shC$ we set $\Omega^{\mu}=\oplus_{j}\Omega^{\mu_{j}}\colon\oplus_{j}\Omega^{\E_{j}}\to\Omega^{\E}$.
A functor $\Gamma\in\L_{R}(\shC,A)$ is\emph{ cohomologically left
exact }on $\mu$ if
\begin{equation}
\Hom_{\L_{R}(\shC,R)}(\Omega^{\E}/\Imm(\Omega^{\mu}),\Gamma)=\Ext_{\L_{R}(\shC,R)}^{1}(\Omega^{\E}/\Imm(\Omega^{\mu}),\Gamma)=0\label{eq:cohomological vanishing left exact functors}
\end{equation}
It is cohomologically left exact if it is so on all surjections $\mu$
as above.
\end{defn}

We setup some notation. We denote by $\Phi_{\shC}(\E)$ for $\E\in\shC$
the set of subfunctor of $\Omega^{\E}$of the form $\Imm(\Omega^{\mu})$
for some surjective map $\mu\colon\oplus_{j}\E_{i}\to\E$ with $\E_{j}\in\shC$
and by $\Phi_{\shC}$ the disjoint union of all the $\Phi_{\shC}(\E)$:
an object of $\Delta\in\Phi_{\shC}$ is actually a pair $(\Delta,\E)$
with $\Delta\in\Phi_{\shC}(\E)$. Given $\Delta=\Imm(\Omega^{\mu})\in\Phi_{\shC}(\E)$
and $\Gamma\in\L_{R}(\shC,A)$ we will say that $\Gamma$ is cohomologically
left exact on $\Delta$ if it is cohomologically left exact on $\mu$.
Notice that, using Yoneda's lemma, we have an $A$-linear isomorphism
\[
\Hom_{\L_{R}(\shC,R)}(\oplus_{j}\Omega^{\E_{j}},\Gamma)\simeq\prod_{j}\Gamma_{\E_{j}}\text{ for all }\E_{j}\in\shC\comma\Gamma\in\L_{R}(\shC,A)
\]
In particular 
\[
\Hom_{\L_{R}(\shC,R)}(\Omega^{\E},\Gamma)\simeq\Gamma_{\E}
\]
 is the evaluation in $\E\in\shC$, is exact: an element $x\in\Gamma_{\E}$
can be thought of as a natural transformation $\Omega^{\E}\to\Gamma$.
which implies that $\Ext_{\L_{R}(\shC,R)}^{1}(\Omega^{\E},-)=0$.

By Yoneda's lemma, we obtain a functorial map 
\[
\Hom_{\L_{R}(\shC,R)}(\oplus_{j}\Omega^{\E_{j}},\oplus_{k}\Omega^{\E_{k}})\to\Hom(\oplus_{j}\E_{j},\oplus_{k}\E_{k})
\]
which is an isomorphism onto the set of locally finite maps.

Let $\mu\colon\oplus_{j}\E_{j}\to\E$ be a surjective map and set
$\Delta=\Imm(\Omega^{\mu})$. Then $\Delta_{\E'}$ is the set of maps
$\E'\to\E$ which factors through $\mu$ via a locally finite map
$\E'\to\oplus_{j}\E_{j}$. Given a map $u\colon\overline{\E}\to\E$
in $\shC$ we set $u^{-1}(\Delta)=\Delta\times_{\Omega^{\E}}\Omega^{\overline{\E}}\subseteq\Omega^{\overline{\E}}$:
$u^{-1}(\Delta)_{\E'}$ is the set of maps $\E'\to\overline{\E}$
such that $\E'\to\overline{\E}\to\E$ factors through $\mu$ via a
locally finite map $\E'\to\oplus_{j}\E_{j}$. Notice that if $\shC\subseteq\QCoh^{\shC}\stX$
then $u^{-1}(\Delta)\in\Phi_{\shC}(\overline{\E})$. Indeed applying
the last remark in \ref{rem:extending to test sequences} on the maps
$\overline{\E}\to\E$ and $\mu\colon\bigoplus_{j}\E_{j}\to\E$ and
denoting by $\shH$ their fiber product we obtain a surjective map
$\bigoplus_{s}\E_{s}\to\overline{\E}$ factoring through $\shH$ and
such that each $\E_{s}\to\overline{\E}$ belongs to $u^{-1}(\Delta)_{\E_{s}}$.
It follows that the obvious map
\[
\mu'\colon\bigoplus_{\tilde{\E}\in\shC}\bigoplus_{\omega\in u^{-1}(\Delta)_{\tilde{\E}}}\tilde{\E}\to\E
\]
is surjective and clearly $u^{-1}(\Delta)=\Imm(\Omega^{\mu'})$.
\begin{lem}
\label{lem:almost left exact functors and hom} Let $\mu\colon\oplus_{j}\E_{i}\to\E$
be a surjective map with $\E_{j},\E\in\shC$ and $\Gamma\in\L_{R}(\shC,A)$.
Then there is an exact sequence of $A$-modules
\[
0\to\Hom_{\L_{R}(\shC,R)}(\Omega^{\E}/\Imm(\Omega^{\mu}),\Gamma)\to\Gamma_{\E}\to\prod_{j}\Gamma_{\E_{j}}
\]
If $\shT:\oplus_{k}\E_{k}\to\oplus_{j}\E_{j}\to\E\to0$ is a test
sequence and $\Gamma$ is exact on $\shT$ then $\Gamma$ is cohomologically
left exact on $\mu$. The converse holds if the map 
\[
\Hom_{\L_{R}(\shC,R)}(\Ker(\Omega^{\mu}),\Gamma)\to\prod_{k}\Gamma_{\E_{k}}
\]
obtained applying $\Hom_{\L_{R}(\shC,R)}(-,\Gamma)$ to the map $\oplus_{k}\Omega^{\E_{k}}\to\Ker(\Omega^{\mu})$
is injective.
\end{lem}

\begin{proof}
Set $\Delta=\Imm(\Omega^{\mu})$, $K=\Ker(\Omega^{\mu})$. Consider
the diagram   \[   \begin{tikzpicture}[xscale=2.7,yscale=-1.1]     
\node (A0_0) at (0.47, 0) {$\Hom(\Omega^\E/\Delta,\Gamma)$};     
\node (A0_1) at (1.3, 0) {$\Gamma_\E$};     
\node (A0_2) at (2, 0) {$\Hom(\Delta,\Gamma)$};     
\node (A1_2) at (2, 1) {$\prod_j \Gamma_{\E_j}$};     
\node (A2_2) at (2, 2) {$\Hom(K,\Gamma)$};     
\node (A0_3) at (3, 0) {$\Ext^1(\Omega^\E/\Delta,\Gamma)$};     
\node (A0_4) at (4, 0) {$\Ext^1(\Omega^\E,\Gamma)$};     
\node (A2_3) at (2.8, 2) {$\prod_k \Gamma_{\E_k}$};     
\path (A0_1) edge [->]node [auto] {$\scriptstyle{\alpha}$} (A1_2);     
\path (A1_2) edge [->]node [auto] {$\scriptstyle{\beta}$} (A2_3);     \path (A0_0) edge [right hook->]node [auto] {$\scriptstyle{}$} (A0_1);     \path (A0_1) edge [->]node [auto] {$\scriptstyle{}$} (A0_2);     \path (A0_2) edge [right hook->]node [auto] {$\scriptstyle{}$} (A1_2);
\path (A2_2) edge [->]node [auto] {$\scriptstyle{\lambda}$} (A2_3);     
\path (A0_3)edge [->]node [auto] {$\scriptstyle{}$} (A0_4);     \path (A0_2) edge [->]node [auto] {$\scriptstyle{}$} (A0_3);     \path (A1_2) edge [->]node [auto] {$\scriptstyle{}$} (A2_2);   \end{tikzpicture}   \]  The convention here is that $\beta$ and $\lambda$ are defined only
when a test sequence $\shT$ as in the statement exists and we will
not use them for the first statement. All the other maps are obtained
splitting $\Omega^{\E_{j}}\to\Omega^{\E}\to0$ into two exact sequences
and applying $\Hom(-,\Gamma)$, so that the first line and the central
column are exact. The map $\alpha$ obtained as composition is the
map defined in the first sequence in the statement. In particular
the first claim follows. So let's focus on the second one. The map
$\lambda$ is the second map in the statement while the map $\beta$
together with $\alpha$ are the maps defining the sequence (\ref{Gamma on a test sequence}).
Since $\Ext^{1}(\Omega^{\E},\Gamma)=0$ also the second claim follows.
\end{proof}
\begin{lem}
\label{lem:key lemma for reducing test sequences} Let $\Gamma,K\in\L_{R}(\shC,R)$
and $u\colon\oplus_{q}\Omega^{\E_{q}}\to K$ be a map, where $\E_{q}\in\shC$.
If for all $\Omega^{\E}\to K$ with $\E\in\shC$ there exists a surjective
map $v\colon\oplus_{t}\E_{t}\to\E$ with $\E_{t}\in\shC$ such that
the composition $\oplus_{t}\Omega^{\E_{t}}\to\Omega^{\E}\to K$ factors
through $u$ and $\Gamma$ is cohomologically left exact on $v$ then
the map 
\[
\Hom_{\L_{R}(\shC,R)}(K,\Gamma)\to\Hom_{\L_{R}(\shC,R)}(\oplus_{q}\Omega^{\E_{q}},\Gamma)\simeq\prod_{q}\Gamma_{\E_{q}}
\]
is injective.
\end{lem}

\begin{proof}
Let $\E\in\shC$ and $x\in K_{\E}$, which corresponds to a map $\Omega^{\E}\to K$.
Consider the data given by hypothesis with respect to this last map.
We have commutative diagrams   \[   \begin{tikzpicture}[xscale=2.5,yscale=-1.2]     
\node (A0_0) at (0, 0) {$\oplus_t \Omega^{\E_t}$};     
\node (A0_1) at (1, 0) {$\Omega^{ \E}$};     
\node (A0_2) at (2, 0) {$\Hom(K,\Gamma)$};     
\node (A0_3) at (3, 0) {$\Gamma_{ \E}$};     
\node (A1_0) at (0, 1) {$\oplus_q \Omega^{\E_q}$};     
\node (A1_1) at (1, 1) {$K$};     
\node (A1_2) at (2, 1) {$\prod_q \Gamma_{\E_q}$};     
\node (A1_3) at (3, 1) {$\prod_t \Gamma_{\E_t}$};     
\path (A0_0) edge [->]node [auto] {$\scriptstyle{\Omega^v}$} (A0_1);     
\path (A0_1) edge [->]node [auto] {$\scriptstyle{x}$} (A1_1);     
\path (A1_0) edge [->]node [auto] {$\scriptstyle{u}$} (A1_1);     
\path (A0_3) edge [->]node [auto] {$\scriptstyle{\delta}$} (A1_3);     
\path (A0_2) edge [->]node [auto] {$\scriptstyle{\lambda}$} (A1_2);     
\path (A0_0) edge [->]node [auto] {$\scriptstyle{}$} (A1_0);     
\path (A0_2) edge [->]node [auto] {$\scriptstyle{\gamma}$} (A0_3);     
\path (A1_2) edge [->]node [auto] {$\scriptstyle{}$} (A1_3);   
\end{tikzpicture}   \]  where the second diagram is obtained by applying $\Hom(-,\Gamma)$
to the first one. The map $\lambda$ is the map in the statement,
while $\gamma$ is the evaluation in $x\in K_{\E}$. Thanks to \ref{lem:almost left exact functors and hom}
and since $\Gamma$ is cohomologically left exact on $v$ the map
$\delta$ is injective. So if $\phi\in\Hom(K,\Gamma)$ is such that
$\lambda(\phi)=0$ it follows that $\gamma(\phi)=\phi_{\E}(x)=0$,
as required.
\end{proof}
\begin{thm}
\label{prop:test sequences and Extone} If $\shC\subseteq\QCoh^{\shC}\stX$
then $\Lex_{R}(\shC,A)$ coincides with the subcategory of $\L_{R}(\shC,A)$
of cohomologically left exact functors.
\end{thm}

\begin{proof}
Let $\mu=\oplus_{j}\mu_{j}\colon\oplus_{j\in J}\E_{j}\to\E$ be a
surjective map with $\E,\E_{j}\in\shC$ and set $\Delta=\Imm(\Omega^{\mu})$,
$K=\Ker(\Omega^{\mu})$. By \ref{rem:extending to test sequences}
there exists a test sequence $\oplus_{k}\E_{k}\to\oplus_{j}\E_{j}\to\E\to0$.
Using \ref{lem:almost left exact functors and hom} we have to prove
that if $\Gamma\in\L_{R}(\shC,R)$ is cohomologically left exact then
$\lambda\colon\Hom(K,\Gamma)\to\prod_{k}\Gamma_{\E_{k}}$ is injective.
We are going to apply \ref{lem:key lemma for reducing test sequences}
with respect to the map $\oplus_{k}\Omega^{\E_{k}}\to K$. If $\overline{\E}\in\shC$,
a map $\Omega^{\overline{\E}}\to K$ is a locally finite map $\overline{\E}\to\oplus_{j}\E_{j}$
which is zero composed by $\mu$, or, equivalently, mapping in the
image of $\oplus_{k}\E_{k}\to\oplus_{j}\E_{j}$. We apply the last
remark in \ref{rem:extending to test sequences} to the maps $\oplus_{k}\E_{k}\to\oplus_{j}\E_{j}$
and $\overline{\E}\to\oplus_{j}\E_{j}$. If $\shH$ is their fiber
product there is a surjective map $\oplus_{s}\E_{s}\to\shH$ with
$\E_{t}\in\shC$ such that $\oplus_{s}\E_{s}\to\oplus_{k}\E_{k}$
is locally finite and $\oplus_{s}\E_{s}\to\overline{\E}$ is surjective.
This gives the desired factorization for applying \ref{lem:key lemma for reducing test sequences}.
\end{proof}
We now show how to reduce the number of test sequences in order to
check when a $\Gamma\in\L_{R}(\shC,A)$ belongs to $\Lex_{R}(\shC,A)$.
The following is the key lemma:
\begin{lem}
\label{lem: reducing the set of test sequences}Let $\Phi'\subseteq\Phi_{\shC}$
such that, for all $\E\in\shC$ and $\Delta\in\Phi_{\shC}(\E)$ there
exists $\Delta'\in\Phi'\cap\Phi_{\shC}(\E)$ such that $\Delta'\subseteq\Delta$
(inside $\Omega^{\E}$). If $\Gamma\in\L_{R}(\shC,A)$, $\shC\subseteq\QCoh^{\shC}\stX$
and $\Gamma$ is cohomologically left exact on all the elements of
$\Phi'$ then $\Gamma$ is cohomologically left exact.
\end{lem}

\begin{proof}
Consider $\Delta\in\Phi_{\shC}(\E)$, $\Delta'\subseteq\Delta$ with
$\Delta'\in\Phi'$ and the exact sequence $0\to\Delta/\Delta'\to\Omega^{\E}/\Delta'\to\Omega^{\E}/\Delta\to0$.
Applying $\Hom(-,\Gamma)$, the only non trivial vanishing to check
is $\Hom(\Delta/\Delta',\Gamma)=0$. In other words we need to prove
that 
\[
\Hom(\Delta,\Gamma)\to\Hom(\Delta',\Gamma)
\]
is injective. Write $\Delta'=\Imm(\Omega^{u})$, where $u\colon\oplus_{q}\E_{q}\to\E$.
Since $\bigoplus_{q}\Omega^{\E_{q}}\to\Delta'$ is surjective it is
enough to show that $\Hom(-,\Gamma)$ maps $\Omega^{u}\colon\bigoplus_{q}\Omega^{\E_{q}}\to\Delta$
to an injective map. We are going to prove that the hypothesis of
\ref{lem:key lemma for reducing test sequences} for $K=\Delta$ are
met. If $\Omega^{\E'}\to\Delta\subseteq\Omega^{\E}$ is a map corresponding
to $\psi\colon\E'\to\E$, then $\psi^{-1}(\Delta')\in\Phi_{\shC}$
and, by hypothesis, we can find $\Phi'\ni\Imm(\Omega^{v})\subseteq\psi^{-1}(\Delta')$;
the last inclusion tells us that $v$ is the factorization required
for \ref{lem:key lemma for reducing test sequences}. Moreover $\Gamma$
is cohomologically left exact on $v$ by hypothesis.
\end{proof}
\begin{prop}
\label{ prop: Lex and left exact functors} Let $\shD\subseteq\QCoh\stX$
be a subcategory and $\Gamma\in\L_{R}(\shD,A)$. If $\Gamma\in\Lex_{R}(\shD,A)$
then $\Gamma$ is exact on finite test sequences in $\shD$ and transforms
any arbitrary direct sum of objects of $\shD$ and which belongs to
$\shD$ into a product. The converse holds if one of the following
conditions is satisfied:

\begin{itemize}
\item the category $\shD$ is stable by arbitrary direct sums;
\item all the sheaves in $\shD$ are finitely presented, $\shD\subseteq\QCoh^{\shD}\stX$
and $\stX$ is quasi-compact. In this case $\Gamma\in\Lex_{R}(\shD,A)$
if and only if it is cohomologically left exact on all surjective
maps $\E'\to\E$ with $\E\in\shD$ and $\E'\in\shD^{\oplus}$.
\end{itemize}
In any of the above cases, if moreover $\shD$ is additive and all
surjections in $\shD$ have kernel in $\shD$ then $\Lex_{R}(\shD,A)$
is the subcategory of $\L_{R}(\shD,A)$ of functors which are left
exact on short exact sequences in $\shD$ and transforms arbitrary
direct sums in products.
\end{prop}

\begin{proof}
If $\Gamma\in\Lex_{R}(\shD,A)$ then it is clearly exact on finite
test sequences. Given a set $\{\E_{j}\in\shD\}_{j\in J}$ set $\E=\bigoplus_{j}\E_{j}$.
If $\E\in\shD$, then the sequence
\[
0\to\bigoplus_{j\in J}\E_{j}\to\E\to0
\]
is a test sequence and therefore we get that the natural map $\Gamma_{\E}\to\prod_{j}\Gamma_{\E_{j}}$
is an isomorphism. Moreover the last part of the statement follows
easily from the first part. Finally if $\shD$ is stable by arbitrary
direct sums is easy to see that the converse holds.  

So we focus on the second point and we assume that all the sheaves
in $\shD$ are finitely presented, $\shD\subseteq\QCoh^{\shD}\stX$
and that $\stX$ is quasi-compact. Since the class of finitely presented
quasi-coherent sheaves on $\stX$ modulo isomorphism is a set, we
can assume $\shD=\shC$ small. Let $\Phi'\subseteq\Phi_{\shC}$ be
the subset of functors of the form $\Imm(\Omega^{\mu})$ for some
surjective map $\mu\colon\E'\to\E$ with $\E\in\shC$ and $\E'\in\shC^{\oplus}$.
The set $\Phi'$ satisfies the hypothesis of \ref{lem: reducing the set of test sequences}:
if $v\colon\oplus_{j\in J}\E_{j}\to\E$ is a surjective map then there
exists a finite subset $J_{0}\subseteq J$ such that $v_{|\E'}\colon\E'=\oplus_{j\in J_{0}}\E_{j}\to\E$
is surjective because $\E$ is of finite type and $\stX$ is quasi-compact.
In particular, taking into account \ref{prop:test sequences and Extone},
the last claim of the second point follows. It remains to show that
if $\Gamma\in\L_{R}(\shC,A)$ is exact on finite test sequences then
$\Gamma$ is cohomologically left exact on all the elements of $\Phi'$.
Let $\mu\colon\E'\to\E$ be a surjective map with $\E\in\shC$ and
$\E'\in\shC^{\oplus}$. Since $\E$ is finitely presented and $\E'$
is of finite type it follows that $\Ker(\mu)$ is of finite type and,
since $\stX$ is quasi-compact and $\shC\subseteq\QCoh^{\shC}\stX$,
there exists $\E''\in\shC^{\oplus}$ and a surjective map $\E''\to\Ker(\mu)$.
Thus $\E''\to\E'\to\E\to0$ is a finite test sequence and by \ref{lem:almost left exact functors and hom}
it follows that $\Gamma$ is cohomologically left exact on $\mu$
as required. %
\end{proof}
There is another characterization of $\Lex_{R}(\shC,A)$ in terms
of sheaves on a site. Although we will not use it in this paper, I
think it is worth to point out. We refer to \cite[Tag 00YW]{SP014}
for general definitions and properties. We start by comparing $\Lex_{R}(\shC,A)$
and $\Lex_{R}(\shC^{\oplus},A)$.
\begin{prop}
\label{prop:Lex for Coplus and C} If $\shC\subseteq\QCoh^{\shC}\stX$
then the equivalence $\L_{R}(\shC^{\oplus},A)\simeq\L_{R}(\shC,A)$
maps $\Lex_{R}(\shC^{\oplus},A)$ to $\Lex_{R}(\shC,A)$.
\end{prop}

\begin{proof}
We can assume $A=R$. Let $\Gamma\in\L_{R}(\shC^{\oplus},R)$ such
that $\Gamma\in\Lex_{R}(\shC,R)$ and consider $\Phi'\subseteq\Phi_{\shC^{\oplus}}$
the set of subfunctors $\Delta\subseteq\Omega^{\E}$ that can be written
as follows: $\E=\E_{1}\oplus\cdots\oplus\E_{r}$ and there are surjective
maps $\mu_{k}\colon\oplus_{q}\E_{q,k}\to\E_{k}$ for $\E_{k},\E_{q,k}\in\shC$
such that $\Delta=\Imm(\Omega^{\mu_{1}})\oplus\cdots\oplus\Imm(\Omega^{\mu_{r}})$.
Since for such $\Delta$ we have 
\[
\Omega^{\E}/\Delta\simeq\oplus_{i}(\Omega^{\E_{i}}/\Delta_{i})
\]
 it follows that $\Gamma$ is cohomologically left exact on all the
elements of $\Phi'$. Taking into account \ref{prop:test sequences and Extone},
in order to conclude that $\Gamma\in\Lex_{R}(\shC^{\oplus},R)$ we
can show that $\Phi'\subseteq\Phi_{\shC^{\oplus}}$ satisfies the
hypothesis of \ref{lem: reducing the set of test sequences}. So let
$\Delta=\Imm(\Omega^{\mu})\in\Phi_{\shC^{\oplus}}$ where $\mu\colon\oplus_{q}\E_{q}\to\E$
where $\E,\E_{q}\in\shC^{\oplus}$. If $\E=\E'_{1}\oplus\cdots\oplus\E'_{r}$
with $\E'_{i}\in\shC$ and $\psi_{i}\colon\E'_{i}\to\E$ are the inclusions
then $\Delta_{i}=\psi_{i}^{-1}(\Delta)\in\Phi_{\shC^{\oplus}}(\E'_{i})=\Phi_{\shC}(\E'_{i})$
and it is easy to see that $\Phi'\ni\Delta_{1}\oplus\cdots\oplus\Delta_{r}\subseteq\Delta$
as required.
\end{proof}
\begin{thm}
\label{prop:Lex and Grothendieck topology} Assume $\shC\subseteq\QCoh^{\shC}\stX$
and let $\shJ$ be the smallest Grothendieck topology on $\shC^{\oplus}$
containing $\Phi_{\shC^{\oplus}}$ then $\Lex_{R}(\shC^{\oplus},A)$
is the category of sheaves of $A$-modules on $(\shC^{\oplus},\shJ)$
which are $R$-linear. In other words $\Gamma\in\L_{R}(\shC^{\oplus},A)$
is left exact if and only if it a sheaf on $(\shC^{\oplus},\shJ)$.
\end{thm}

\begin{proof}
We can assume $A=R$ and $\shC=\shC^{\oplus}$. If $\Delta\subseteq\Omega^{\E}$
is a sieve and $f\colon\E'\to\E$ we set $f^{-1}(\Delta)=\Delta\times_{\Omega^{\E}}\Omega^{\E'}\subseteq\Omega^{\E'}$.
Let $\widetilde{\shJ}$ be the set of sieves $\Delta\subseteq\Omega^{\E}$
of $\shC$ such that, for all $\Gamma\in\Lex_{R}(\shC,R)$ and maps
$f\colon\E'\to\E$ the map
\[
\Hom_{\sets}(\Omega^{\E'},\Gamma)\to\Hom_{\sets}(f^{-1}(\Delta),\Gamma)
\]
is bijective. Here $\Hom_{\sets}$ denotes the set of natural transformation
of functors with values in $\sets$. The set $\widetilde{\shJ}$ is
a Grothendieck topology on $\shC$ such that all functors in $\Lex_{R}(\shC,R)$
are sheaves: see \cite[Tag 00Z9]{SP014}. Notice that, by \ref{rem: natural transf are linear for additive categories},
if $A,B\in\L_{R}(\shC,R)$ then $\Hom_{\sets}(A,B)=\Hom_{\L_{R}(\shC,R)}(A,B)$.
Moreover if $\Delta\in\Phi_{\shC}$ and $f\colon\E'\to\E$ is a map
in $\shC$ then $\Delta'=f^{-1}(\Delta)\in\Phi_{\shC}$. If $\Gamma\in\L_{R}(\shC,R)$
then, applying $\Hom_{\L_{R}(\shC,R)}(-,\Gamma)$ on the exact sequence
$0\to\Delta'\to\Omega^{\E'}\to\Omega^{\E'}/\Delta'\to0$ and taking
into account that $\Ext^{1}(\Omega^{\E'},\Gamma)=0$ we obtain an
exact sequence 
\[
0\to\Hom(\Omega^{\E'}/\Delta',\Gamma)\to\Hom_{\sets}(\Omega^{\E'},\Gamma)\to\Hom_{\sets}(\Delta',\Gamma)\to\Ext^{1}(\Omega^{\E'}/\Delta',\Gamma)\to0
\]
By \ref{prop:test sequences and Extone} $\Phi_{\shC}\subseteq\widetilde{J}$
and therefore $\shJ\subseteq\widetilde{\shJ}$, which means that if
$\Gamma\in\Lex_{R}(\shC,A)$ then $\Gamma$ is a sheaf on $(\shC,\shJ)$.
Conversely if $\Gamma$ is a sheaf on $(\shC,\shJ)$ we immediately
see from the above sequence that $\Gamma$ is cohomologically left
exact, which ends the proof.
\end{proof}

\section{Some special cases}

We now apply \ref{thm:general theorem for equivalences} and \ref{ prop: Lex and left exact functors}
in some (more) concrete situations. The symbol $A$ will denote an
$R$-algebra.

This is Gabriel-Popescu's theorem for the category $\QCoh\stX$.
\begin{thm}
\label{thm: Gabriel Popescu's theorem}{[}Gabriel-Popescu's theorem{]}
Let $\stX$ be a pseudo-algebraic fibered category. If $\E$ is a
generator of $\QCoh\stX$ then the functor
\[
\Hom_{\stX}(\E,-)\colon\QCoh_{A}\stX\to\Mod_{\text{right}}(\End_{\stX}(\E)\otimes_{R}A)
\]
is fully faithful and has an exact left adjoint.
\end{thm}

\begin{proof}
It follows from \ref{thm:general theorem for equivalences} and \ref{prop:excatness properties of Omega* and shF*}
with $\shD=\shC=\{\E\}$: in this case 
\[
\L_{R}(\shC,A)\simeq\Mod_{\text{right}}(\End_{\stX}(\E)\otimes_{R}A)\comma\Gamma\longmapsto\Gamma_{\E}
\]
\end{proof}
\begin{thm}
\label{thm:equivalence for QCoh X} Let $\stX$ be a pseudo-algebraic
fibered category over $R$. The category $\Lex_{R}(\QCoh\stX,A)$
is the category of contravariant, $R$-linear and left exact functors
$\Gamma\colon\QCoh\stX\to\Mod A$ which transform arbitrary direct
sums into products. Moreover the functor
\[
\Omega^{*}\colon\QCoh_{A}\stX\to\Lex_{R}(\QCoh\stX,A)
\]
is an equivalence of categories.
\end{thm}

\begin{proof}
Follows from \ref{thm:general theorem for equivalences} and \ref{ prop: Lex and left exact functors}
with $\shD=\QCoh\stX$.
\end{proof}
In what follows $\QCoh_{\text{fp}}\stX$ denotes the category of quasi-coherent
sheaves of finite presentation.
\begin{thm}
\label{thm:equivalence for finite presented sheaves} Let $\stX$
be a quasi-compact fibered category over $R$ such that $\QCoh_{\text{fp}}\stX$
generates $\QCoh\stX$ (e.g. a quasi-compact and quasi-separated scheme
by \cite[Section 6.9]{EGAI}). Then $\Lex_{R}(\QCoh_{\text{fp}}\stX,A)$
is the category of contravariant, $R$-linear functors $\QCoh_{\text{fp}}\stX\to\Mod A$
which are left exact on right exact sequences. Moreover the functors
\[
\Omega^{*}\colon\QCoh_{A}\stX\to\Lex_{R}(\QCoh_{\text{fp}}\stX,A)\comma\shF_{*,\QCoh_{\text{fp}}(\stX)}\colon\Lex_{R}(\QCoh_{\text{fp}}\stX,A)\to\QCoh_{A}\stX
\]
are quasi-inverses of each other.
\end{thm}

\begin{proof}
Follows from \ref{thm:general theorem for equivalences} and \ref{ prop: Lex and left exact functors}
\end{proof}
When $\stX$ is a noetherian algebraic stack then $\QCoh_{\text{fp}}\stX=\Coh\stX$
is abelian and generates $\QCoh\stX$ (see \cite[Prop 15.4]{Laumon1999}).
In particular we obtain:
\begin{thm}
\label{thm:equivalence for noetherian algebraic stacks} Let $\stX$
be a noetherian algebraic stack. The category $\Lex_{R}(\Coh\stX,A)$
is the category of contravariant, $R$-linear and left exact functors
$\Coh\stX\to\Mod A$. Moreover the functor
\[
\Omega^{*}\colon\QCoh_{A}\stX\to\Lex_{R}(\Coh\stX,A)\comma\shF_{*,\Coh(\stX)}\colon\Lex_{R}(\Coh\stX,A)\to\QCoh_{A}\stX
\]
are quasi-inverses of each other.
\end{thm}

\begin{thm}
\label{thm:equivalence for Loc X} Let $\stX$ be a quasi-compact
fibered category over $R$ such that $\Loc\stX$ generates $\QCoh\stX$.
Then $\Lex_{R}(\Loc\stX,A)$ is the category of contravariant, $R$-linear
and left exact functors $\Loc\stX\to\Mod A$. Moreover the functors
\[
\Omega^{*}\colon\QCoh_{A}\stX\to\Lex_{R}(\Loc\stX,A)\comma\shF_{*,\Loc(\stX)}\colon\Lex_{R}(\Loc\stX,A)\to\QCoh_{A}\stX
\]
are quasi-inverses of each other.
\end{thm}

\begin{proof}
Follows from \ref{thm:general theorem for equivalences} and \ref{ prop: Lex and left exact functors},
taking into account that all surjections in $\Loc\stX$ have kernels
in $\Loc\stX$.
\end{proof}
\begin{thm}
\label{thm:equivalence for affine schemes} Let $B$ be an $R$-algebra
and $\shD\subseteq\Mod B$ be a subcategory that generates $\Mod B$,
that is there exists $\E_{1},\dots,\E_{r}\in\shD$ with a surjective
map $\bigoplus_{i}\E_{i}\to B$. Then the functor
\[
\Omega^{*}\colon\Mod(A\otimes_{R}B)\to\Lex_{R}(\shD,A)
\]
is an equivalence of categories. Moreover if $\shD\subseteq\Loc B$
then $\Lex_{R}(\shD,A)=\L_{R}(\shD,A)$.
\end{thm}

\begin{proof}
If $\stX=\Spec B$, then $\QCoh_{A}\stX\simeq\Mod(A\otimes_{R}B)$
and the first part follows from \ref{thm:general theorem for equivalences}.
For the last claim, observe that any $\Gamma\colon\Loc B\to\Mod A$
is exact because any short exact sequence in $\Loc B$ splits. By
\ref{thm:equivalence for Loc X} we can conclude that $\L_{R}(\Loc B,A)=\Lex_{R}(\Loc B,A)$.
If now $\shD\subseteq\Loc B$ and $\Gamma\in\L_{R}(\shD,A)$, we can
extend it to $\overline{\Gamma}\in\L_{R}(\Loc B,A)$ and therefore
$\Gamma=\overline{\Gamma}_{|\shD}\in\Lex_{R}(\shD,A)$.
\end{proof}
We want to extend Theorem \ref{thm:general theorem for equivalences}
to functors with monoidal structures.
\begin{defn}
If $\shD$ is a monoidal subcategory of $\QCoh\stX$ we define $\PMLex_{R}(\shD,A)$
(resp. $\MLex_{R}(\shD,A)$) as the subcategory of $\PML_{R}(\shD,A)$
(resp. $\ML_{R}(\shD,A)$) of functors $\Gamma$ such that $\Gamma\in\Lex_{R}(\shD,A)$.
\end{defn}

\begin{thm}
\label{thm:fundamental for sheaf of algebras and monoidal functors}
Let $\shD$ be a monoidal subcategory of $\QCoh\stX$ that generates
it. Then the functors
\[
\Omega^{*}\colon\Rings_{A}\stX\to\PMLex_{R}(\shD,A)\text{ and }\Omega^{*}\colon\QAlg_{A}\stX\to\MLex_{R}(\shD,A)
\]
(see \ref{prop:Omegastar on algebras}) are equivalence of categories.
If $\shD$ is small a quasi inverse is given by $\alA_{*,\shD}\colon\PMLex_{R}(\shD,A)\to\Rings_{A}\stX$
and $\alA_{*,\shD}\colon\MLex_{R}(\shD,A)\to\QAlg_{A}\stX$ respectively
(see \ref{prop: from sheaves to algebras}). Moreover if $\overline{\shD}\subseteq\shD$
is a monoidal subcategory that generates $\QCoh\stX$ the restriction
functors $\PMLex_{R}(\shD,A)\to\PMLex_{R}(\overline{\shD},A)$ and
$\MLex_{R}(\shD,A)\to\MLex_{R}(\overline{\shD},A)$ are equivalences.
\end{thm}

\begin{proof}
Assume that $\shD$ is small. Then $\Omega^{*}\colon\Rings_{A}\stX\to\PMLex_{R}(\shD,A)$
and $\alA_{*,\shD}\colon\PMLex_{R}(\shD,A)\to\Rings_{A}\stX$ are
quasi-inverses of each other because, by \ref{prop: from sheaves to algebras},
we have natural transformations $\id\to\Omega^{*}\circ\alA_{*,\shD}$
and $\alA_{*,\shD}\circ\Omega^{*}\to\id$ which are isomorphisms thanks
to \ref{thm:general theorem for equivalences}. In particular if $\overline{\shD}\subseteq\shD$
is a monoidal subcategory that generates $\QCoh\stX$ then the restriction
functor $\PMLex_{R}(\shD,A)\to\PMLex_{R}(\overline{\shD},A)$ is an
equivalence. Since $\gamma$ and $\beta$ preserve unities by \ref{prop: from sheaves to algebras},
the same holds if we replace $\Rings_{A}\stX$ by $\QAlg_{A}\stX$
and $\PMLex_{R}(\shD,A)$ by $\MLex_{R}(\shD,A)$.

Now assume that $\shD$ is general. Notice that there exists a small
subcategory $\shC'\subseteq\shD$ that generates $\QCoh\stX$ thanks
to \ref{rem: finding small subcategory generating QCoh X}. If $I\subseteq\shD$
is a set we set $\shC_{I}\subseteq\shD$ for the category whose objects
are isomorphic to a (multiple) tensor product with factors in $C\cup I$
and $\odi{\stX}$. We have that $\shC_{I}$ is a collection of small
monoidal subcategories of $\shD$, that generate $\QCoh\stX$ and
such that $\shC_{I}\subseteq\shC_{I'}$ if $I\subseteq I'$. We can
show that the restrictions $\PMLex_{R}(\shD,A)\to\PMLex_{R}(\shC_{\emptyset},A)$
and $\MLex_{R}(\shD,A)\to\MLex_{R}(\shC_{\emptyset},A)$ are equivalences
by proceeding as in the proof of \ref{thm:general theorem for equivalences}.
All the other claims in the statement follow easily from this fact.
\end{proof}
\begin{thm}
\label{thm:rewriting of general equivalences for monoidal functors}
The theorems \ref{thm:equivalence for QCoh X}, \ref{thm:equivalence for finite presented sheaves},
\ref{thm:equivalence for noetherian algebraic stacks}, \ref{thm:equivalence for Loc X}
continue to hold if we replace $\Lex_{R}$ by $\PMLex_{R}$ (resp.
$\MLex_{R}$), $\QCoh_{A}\stX$ by $\QRings_{A}\stX$ (resp. $\QAlg_{A}\stX$),
$\shF_{*,\shC}$ by $\alA_{*,\shC}$ and the word ``functors'' by
``pseudo-monoidal functors'' (resp. ``monoidal functors'').
\end{thm}

\section{\label{sec:Quasi-projective-schemes and more}Quasi projective schemes
and more}

The goal of this section is to show the relation between the sheafification
functors we defined and the classical one for projective schemes.
The dictionary is explained in the following result. In what follows
$\stX$ is a pseudo-algebraic fibered category over a ring $R$ and
$A$ is an $R$-algebra. If $S$ is a graded algebra we denote by
$\GMod(S)$ the category of graded $S$-modules.
\begin{prop}
\label{prop:dictionary invertible sheaves projective} Let $\shL$
be an invertible sheaf over $\stX$ and set $\shC_{\shL}=\{\shL^{\otimes n}\}_{n\in T}$,
where $T$ is either $\N$ or $\Z$, and 
\[
S_{\shL}=\bigoplus_{n\in T}\Hl^{0}(\shL^{-\otimes n})
\]
with its canonical $\Hl^{0}(\odi{\stX})$-algebra structure. Then
\[
\L_{R}(\shC_{\shL},A)\to\GMod(S_{\shL}\otimes_{R}A)\comma\Gamma\longmapsto\bigoplus_{n\in T}\Gamma_{\shL^{\otimes n}}
\]
is well defined and an equivalence of categories. Under this equivalence
$\Omega^{*}\colon\QCoh_{A}\stX\to\L_{R}(\shC_{\shL},A)$ become
\[
\QCoh_{A}\stX\to\GMod(S_{\shL}\otimes_{R}A)\comma\shG\longmapsto\bigoplus_{n\in T}\Hl^{0}(\shG\otimes\shL^{-\otimes n})
\]
while $\shF_{\Gamma,\shC}\colon\L_{R}(\shC,A)\to\QCoh_{A}\stX$ become
the functor $\widetilde{-}\colon\GMod(S_{\shL}\otimes_{R}A)\to\QCoh_{A}\stX$
defined by
\[
\widetilde{M}(\xi\colon\Spec B\to\stX)=(M\otimes_{S_{\shL}}S_{\xi^{*}\shL})_{0}
\]
\end{prop}

\begin{proof}
Given a $T$-graded $A$-module $\bigoplus_{n}T_{n}$ a graded $S_{\shL}\otimes_{R}A$-module
structure is given by $R$-linear maps
\[
\Hl^{0}(\shL^{-\otimes q})\to\Hom_{A}(T_{m},T_{m+q})\text{ for }m,q\in T
\]
satisfying the obvious conditions. Instead a structure of contravariant
$R$-linear functor $\Gamma\colon\shC_{\shL}\to\Mod A$ with $\Gamma_{\shL^{\otimes n}}=T_{n}$
is given by maps
\[
\Hom_{\stX}(\shL^{\otimes(m+q)},\shL^{\otimes m})\to\Hom_{A}(T_{m},T_{m+q})\text{ for }m,q\in T
\]
again satisfying certain compatibility conditions. Thus one has to
check that the isomorphisms
\[
\Hl^{0}(\shL^{-\otimes q})\to\Hom_{\stX}(\shL^{\otimes(m+q)},\shL^{\otimes m})
\]
are compatible with tensor products, which is an elementary computation.
The description of $\Omega^{*}$ follows from definition because
\[
\Omega_{\shL^{\otimes n}}^{\shG}=\Hom_{\stX}(\shL^{\otimes n},\shG)=\Hl^{0}(\shG\otimes\shL^{-\otimes n})
\]
Let $\xi\colon\Spec B\to\stX$ be a map. Notice that $\xi^{*}\shC_{\shL}\simeq\shC_{\xi^{*}\shL}$
and that there is a canonical graded algebra morphism $S_{\shL}\to S_{\xi^{*}\shL}$.
It is clear that the restriction $\xi_{*}\colon\L_{R}(\shC_{\xi^{*}\shL},A)\to\L_{R}(\shC_{\shL},A)$
corresponds to the restriction $\GMod(S_{\xi^{*}\shL}\otimes_{R}A)\to\GMod(S_{\shL}\otimes_{R}A)$.
It therefore follows that $\xi^{*}\colon\L_{R}(\shC_{\shL},A)\to L_{R}(\shC_{\xi^{*}\shL},A)$
correspond to the tensor product $-\otimes_{S_{\shL}}S_{\xi^{*}\shL}\colon\GMod(S_{\shL}\otimes_{R}A)\to\GMod(S_{\xi^{*}\shL}\otimes_{R}A)$.
On the other hand one can check easily that
\[
\Omega^{*}\colon\QCoh_{A}\Spec B=\Mod(B\otimes_{R}A)\to\GMod(S_{\xi^{*}\shL}\otimes_{R}A)\simeq\L_{R}(\shC_{\xi^{*}\shL},A)\comma N\longmapsto N\otimes_{(B\otimes_{R}A)}S_{\xi^{*}\shL}
\]
and that this is natural in $\xi\in\stX$. Its adjoint is 
\[
\GMod(S_{\xi^{*}\shL}\otimes_{R}A)\to\Mod(B\otimes_{R}A)\comma M\longmapsto M_{0}
\]
By \ref{prop:pullbacks and compatibilities} we obtain the description
of $\shF_{\Gamma,\shC}$.
\end{proof}
\begin{thm}
\label{thm:for projective schemes} Let $\stX$ be a pseudo-algebraic
fibered category over $R$, $\shL$ a line bundle on $\stX$ and assume
that $\shC_{\shL}=\{\shL^{\otimes n}\}_{n\in T}$ generates $\QCoh(\stX)$,
where $T$ is either $\N$ or $\Z$. Set $S_{\shL}=\bigoplus_{n\in T}\Hl^{0}(\shL^{-\otimes n})$.
Then the functors 
\[
\Gamma_{*}\colon\QCoh_{A}\stX\to\GMod(S_{\shL}\otimes_{R}A)\comma\shG\longmapsto\Gamma_{*}(\shG)=\bigoplus_{n\in T}\Hl^{0}(\shG\otimes\shL^{-\otimes n})
\]
\[
\widetilde{-}\colon\GMod(S_{\shL}\otimes_{R}A)\to\QCoh_{A}\stX\comma\widetilde{M}(\xi\colon\Spec B\to\stX)=(M\otimes_{S_{\shL}}S_{\xi^{*}\shL})_{0}
\]
are well defined and the first one is right adjoint to the second
one. Moreover $\Omega^{*}$ is fully faithful and left exact, $\widetilde{-}$
is exact and the morphism 
\[
\widetilde{\Gamma_{*}(\shG)}\to\shG\text{ for }\shG\in\QCoh_{A}\stX
\]
is an isomorphism.
\end{thm}

\begin{proof}
All claims are consequence of \ref{thm:when Omegastar is fully faithful}
, \ref{prop:excatness properties of Omega* and shF*} and \ref{prop:dictionary invertible sheaves projective}.
\end{proof}
\begin{rem}
\label{rem:quasi-affine on sheaves} If $p\colon\stY\to\stX$ is a
representable and quasi-affine map of pseudo-algebraic fibered categories
and $\shF\in\QCoh(\stX)$ then $p^{*}p_{*}\shF\to\shF$ is surjective.
In particular if $\shC\subseteq\QCoh(\stX)$ generates $\QCoh(\stX)$
then $p^{*}\shC\subseteq\QCoh(\stY)$ generates $\QCoh(\stY)$.

The last claim follows from the first because if $\bigoplus_{j}\E_{j}\to p_{*}\shF$
is a surjective map with $\E_{j}\in\shC$ then 
\[
\bigoplus_{j}p^{*}\shE_{j}\to p^{*}p_{*}\shF\to\shF
\]
is a surjective map with $p^{*}\E_{j}\in p^{*}\shC$. For the first
claim, since the problem is fpqc local on the target, we can assume
$\stX$ affine, so that $\stY$ would be a quasi-affine scheme. In
this situation we have to prove that all quasi-coherent sheaves on
$\stY$ are generated by global sections. If $j\colon\stY\to\Spec C$
is an open immersion and $\shF\in\QCoh(\stY)$ then one find a surjective
map $\shH\to j_{*}\shF$ from a free sheaf and $j^{*}\shH\to j^{*}j_{*}\shF\simeq\shF$
will be again a surjective map from a free sheaf.
\end{rem}

\begin{rem}
Here we describe some situations in which Theorem \ref{thm:for projective schemes}
can be applied.
\begin{itemize}
\item If $f\colon X\to\PP_{R}^{n}$ is a quasi-affine map and $\shL=f^{*}\odi{\PP_{R}^{n}}(-1)$
then $\shC_{\shL}=\{\shL^{\otimes n}\}_{n\in\N}$ generates $\QCoh(X)$.
This follows applying \ref{rem:quasi-affine on sheaves} to the quasi-affine
map $X\to\PP_{R}^{n}\to\PP_{\Z}^{n}$ and the analogous and classical
result on $\PP_{\Z}^{n}$.
\item If $X$ is a quasi-compact and quasi-projective scheme over $R$,
$j\colon X\to\PP_{R}^{n}$ is an immersion and $\shL=j^{*}\odi{\PP_{R}^{n}}(-1)$
then $\shC_{\shL}=\{\shL^{\otimes n}\}_{n\in\N}$ generates $\QCoh(X)$.
This is a particular case of the previous one. Since $X$ is quasi-compact
and $\PP_{R}^{n}$ is quasi-separated it follows that $X\to\PP_{R}^{n}$
is a quasi-compact immersion. By \cite[Tag 01QV]{SP014} it follows
that $X\to\PP_{R}^{n}$ is quasi-affine.
\item If $U$ is a quasi-affine scheme with an action of $\Gm$, $\stX=[U/\Gm]$
and $\shL$ is the line bundle corresponding to $\stX\to\Bi\Gm$ then
$\shC_{\shL}=\{\shL^{\otimes n}\}_{n\in\Z}$ generates $\QCoh(\stX)$.
Indeed applying \ref{rem:quasi-affine on sheaves} to the quasi-affine
map $\stX\to\Bi\Gm$ one reduces to $\stX=\Bi\Gm$ as a stack over
$\Spec\Z$. Quasi-coherent sheaves on $\Bi\Gm$ (over $\Spec\Z$)
are all of the form $\bigoplus_{n\in\Z}G_{n}\otimes\shL^{\otimes n}$
for abelian groups $G_{n}$. Thus $\shC_{\shL}$ generates $\QCoh(\stX)$.
\item If $U$ is a quasi-affine scheme with an action of $\mu_{d}$, $\stX=[U/\mu_{d}]$
and $\shL$ is the $d$-torsion line bundle corresponding to $\stX\to\Bi\mu_{d}$
then $\shC_{\shL}=\{\shL^{\otimes n}\}_{n\in\N}$ generates $\QCoh(\stX)$.
Indeed applying \ref{rem:quasi-affine on sheaves} to the quasi-affine
map $\stX\to\Bi\mu_{d}$ one reduces to $\stX=\Bi\mu_{d}$ as a stack
over $\Spec\Z$. Quasi-coherent sheaves on $\Bi\mu_{d}$ (over $\Spec\Z$)
are all of the form $\bigoplus_{n=0}^{d-1}G_{n}\otimes\shL^{\otimes n}$
for abelian groups $G_{n}$. Thus $\shC_{\shL}$ generates $\QCoh(\stX)$.
\end{itemize}
\end{rem}

\section{\label{sec: group schemes representation} Group schemes and representations.}

Let $G$ be a flat and affine group scheme over $R$. In this section
we want to interpret the results obtained in the case $\stX=\Bi_{R}G$,
the stack of $G$-torsors for the \emph{fpqc} topology, which is a
quasi-compact fpqc stack with affine diagonal.

If $A$ is an $R$-algebra, by standard theory we have that $\QCoh_{A}\Bi_{R}G$
is the category $\Mod^{G}A$ of $G$-comodules over $A$. Recall that
the regular representation $R[G]$ of $G$ is by definition $\Hl^{0}(G,\odi G)$
with the coaction induced by the right action of $G$ on itself. By
definition it comes equipped with a morphism of $R$-algebras $\varepsilon\colon R[G]\to R$
induced by the unit section of $G$.
\begin{rem}
\label{rem: crucial isomorphism invariants RG and modules} If $M\in\Mod^{G}A$
then the composition
\[
(M\otimes_{R}R[G])^{G}\to M\otimes_{R}R[G]\arrdi{\id\otimes\varepsilon}M
\]
 is an isomorphism. This follows from \cite[3.4]{Jantzen2003} applied
to $G=H$.
\end{rem}

We start with a criterion to find a set of generators for $\QCoh\Bi_{R}G$.
\begin{prop}
\label{thm:generating sheaves for BG} If the regular representation
$R[G]$ is a filtered direct limit of modules $B_{i}\in\Mod^{G}R$
which are finitely presented as $R$-modules then $\{\duale B_{i}\}_{i\in I}$
generates $\Mod^{G}R$.
\end{prop}

\begin{proof}
Set $B=R[G]$ and $\varepsilon_{i}\colon B_{i}\to R$ for the composition
$B_{i}\to B\arrdi{\varepsilon}R$ and let $M\in\Mod^{G}R$. Since
filtered direct limits commute with tensor products and taking invariants,
by \ref{rem: crucial isomorphism invariants RG and modules} we have
that the limit of the maps $(\varepsilon_{i}\otimes\id_{M})_{|(B_{i}\otimes M)^{G}}\colon(B_{j}\otimes M)^{G}\to M$
is an isomorphism. This means that for any $m\in M$ there exists
$i_{m}\in I$ and an element $\psi_{m}\in(B_{i_{m}}\otimes M)^{G}$
such that $(\varepsilon_{i}\otimes\id_{M})(\psi_{m})=m$. There is
a commutative diagram   \[   \begin{tikzpicture}[xscale=2.0,yscale=-1.2]     \node (A0_0) at (0, 0.5) {$b\otimes m$};     \node (A0_2) at (2, 0.5) {$(\phi \mapsto \phi(b)m)$};     \node (A1_0) at (0, 1) {$B_i \otimes M$};     \node (A1_2) at (2, 1) {$\Hom(B_i^\vee,M)$};     \node (A1_3) at (3, 1) {$\psi$};     \node (A2_1) at (1, 2) {$M$};     \node (A2_2) at (2, 2) {$\psi(\varepsilon_i)$};     \path (A1_3) edge [serif cm->]node [auto] {$\scriptstyle{}$} (A2_2);     \path (A0_0) edge [->]node [auto] {$\scriptstyle{}$} (A0_2);     \path (A1_0) edge [->]node [auto] {$\scriptstyle{\varepsilon_i\otimes \id}$} (A2_1);     \path (A1_0) edge [->]node [auto] {$\scriptstyle{}$} (A1_2);     \path (A1_2) edge [->]node [auto] {$\scriptstyle{}$} (A2_1);   \end{tikzpicture}   \] 
and the horizontal map is $G$-equivariant. Therefore we obtain a
$\delta_{m}\in\Hom^{G}(\duale{B_{i}},M)$ such that $\delta_{m}(\varepsilon_{i})=m$.
This implies that the map
\[
\bigoplus_{m\in M}\delta_{m}\colon\bigoplus_{m\in M}\duale B_{i_{m}}\to M
\]
is surjective and therefore that $M$ is generated by $\{\duale{B_{i}}\}_{i\in I}$.
\end{proof}
\begin{rem}
\label{rem: resolution property for BG cases} The class $\shG_{R}$
of flat, affine group schemes $G$ over $R$ such that $R[G]$ is
a direct limit of modules in $\Loc(\Bi_{R}G)$ is stable by arbitrary
products and projective limits. Moreover by construction contains
all groups which are flat, finite and finitely presented over $R$,
i.e. $R[G]\in\Loc(\Bi_{R}G)$, and thus all profinite groups. Since
any $G$-comodule is the union of the sub $G$-comodules which are
finitely generated $R$-modules (see \cite[2.13]{Jantzen2003}), we
see that $\shG_{R}$ contains all flat groups defined over a Dedekind
domain or a field, such as $\GL_{r}$, $\text{SL}_{r}$ and all diagonalizable
groups. Proposition \ref{thm:generating sheaves for BG} tells us
that if $G\in\shG_{R}$ then $\Bi_{R}G$ has the resolution property,
that is $\Loc(\Bi_{R}G)$ generates $\QCoh(\Bi_{R}G)$.
\end{rem}

Let $A$ be an $R$-algebra. We denote by $\Loc^{G}A$ the subcategory
of $\Mod^{G}A$ of $G$-comodules that are locally free of finite
rank (projective of finite type) as $A$-modules, so that $\Loc(\Bi_{R}G)\simeq\Loc^{G}R$.
We define $\QAdd^{G}A$ ($\QPMon^{G}A$, $\QMon^{G}A$) as the category
of \emph{covariant }$R$-linear (pseudo-monoidal, monoidal) functors
$\Loc^{G}R\to\Mod A$. We set $\QRings^{G}A$ for the category of
$M\in\Mod^{G}A$ with a $G$-equivariant map $M\otimes_{A}M\to M$
and $\QAlg^{G}A$ for the (not full) subcategory of $\QRings^{G}A$
of commutative $R$-algebras.
\begin{defn}
The group $G$ is called \emph{linearly reductive} if the functor
$(-)^{G}\colon\Mod^{G}R\to\Mod R$ is exact.
\end{defn}

\begin{rem}
\label{rem: LocG R semisimple if G linearly reductive} If $G$ is
linearly reductive then any short exact sequence in $\Loc^{G}R$ splits.
Indeed if $M\to N$ is surjective then $\Hom_{R}^{G}(N,M)\to\Hom_{R}^{G}(N,N)$
is surjective, yielding a $G$-equivariant section $N\to M$.
\end{rem}

\begin{thm}
\label{thm:general theorem for BG} If $\Bi_{R}G$ has the resolution
property then the functors
\[
\Mod^{G}A\to\QAdd^{G}A\comma\QRings^{G}A\to\QPMon^{G}A\comma\QAlg^{G}A\to\QMon^{G}A
\]
which maps $M$ to the functor $(-\otimes_{R}M)^{G}\colon\Loc^{G}R\to\Mod A$
are well defined, fully faithful and have essential image the subcategory
of functors which are left exact on short exact sequences in $\Loc^{G}R$.
In particular they are equivalences if $G$ is a linearly reductive
group.
\end{thm}

\begin{proof}
Set $\shC=\Loc^{G}R$. The functor $\duale{(-)}\colon\Loc^{G}R\to\Loc^{G}R$
is an equivalence and therefore we get equivalences $\QAdd^{G}A\simeq\L_{R}(\shC,A)$,
$\QPMon^{G}A\simeq\PML_{R}(\shC,A)$ and $\QMon^{G}A\simeq\ML_{R}(\shC,A)$.
Left exact functors are sent to left exact functors. Under those equivalences
$\Omega^{M}$ corresponds to $(-\otimes_{R}M)^{G}$ because $\Hom_{\Bi_{R}G}(\duale{\E},M)\simeq\Hl^{0}(\E\otimes_{\odi{\Bi_{R}G}}M)\simeq(\E\otimes_{R}M)^{G}$.
Thus the result follows from \ref{prop:excatness properties of Omega* and shF*},
\ref{thm:equivalence for Loc X}, \ref{thm:rewriting of general equivalences for monoidal functors}
and \ref{rem: LocG R semisimple if G linearly reductive}.
\end{proof}
\bibliographystyle{amsalpha}
\bibliography{biblio2}

\end{document}